\documentclass[10pt]{article}

\usepackage{hyperref}

\usepackage{titlesec}

\usepackage{amsmath}

\DeclareMathOperator{\erf}{erf}
\DeclareMathOperator{\erfc}{erfc}

\usepackage{epsfig, amssymb, amsfonts, dsfont}
\usepackage{amsthm}
\usepackage{amstext}
\usepackage{amsopn}
\usepackage{mathrsfs}
\usepackage{subfigure}
\usepackage{graphicx}
\usepackage[left=2.3cm,top=2cm,right=2.3cm,bottom=2cm, nohead,foot=1cm]{geometry}
\usepackage[makeroom]{cancel}
\usepackage{color}
\usepackage{enumerate}
\usepackage{comment}
\usepackage{stackrel}
\usepackage{stmaryrd}
\usepackage{mathtools}
\usepackage{amsbsy}

\usepackage{tikz}
\usepackage{tikz-cd}

\SetSymbolFont{stmry}{bold}{U}{stmry}{m}{n}

\def\stackrel#1#2{\mathrel{\mathop{#2}\limits^{#1}}}

\numberwithin{equation}{section}

\newtheorem{theorem}{Theorem}[section]

\newtheorem{lemma}[theorem]{Lemma}

\newtheorem{corollary}[theorem]{Corollary}
\newtheorem{proposition}[theorem]{Proposition}

\theoremstyle{definition}

\newtheorem{remark}[theorem]{Remark}

\newcommand{\R}{{\mathbb R}}

\DeclareFontFamily{U}{mathx}{\hyphenchar\font45}
\DeclareFontShape{U}{mathx}{m}{n}{
      <5> <6> <7> <8> <9> <10>
      <10.95> <12> <14.4> <17.28> <20.74> <24.88>
      mathx10
      }{}
\DeclareSymbolFont{mathx}{U}{mathx}{m}{n}
\DeclareMathSymbol{\bigtimes}{1}{mathx}{"91}

\usepackage{relsize}

\newcommand{\vsm}{{\mathsmaller{\mathsmaller{V}}}}

\newcommand{\Da}{{D}}

\title{A dyadic construction of a three-dimensional attractive point interaction Markov family}

\date{  }

 \author{\textbf{Barkat Mian}\footnote{ {\tt bmian@utk.edu}} \vspace{.1cm}  \\  University of Tennessee Knoxville, Department of Mathematics   }

\begin{document}
\maketitle

\begin{abstract}
We discuss a probabilistic approximation framework for the three-dimensional attractive point interaction on a finite time horizon. By iterating the Doob transforms of the explicit heat kernel associated with the singular Schr\"odinger operator formally given by
\[
\frac12\Delta \,+\, \frac{\beta}{2}\, \delta_0(\cdot),
\qquad \beta>0,
\]
we obtain sub-probability kernels along finite partitions on the punctured domain
\[
E_\varepsilon=\{x\in\mathbb R^3:\ |x|>\varepsilon\},
\]
which yield a limiting sub-probability kernel via refinement along global dyadic partitions, and we extend this limit to a transition probability kernel on an enlarged space obtained by adjoining a cemetery state. These kernels determine a time-inhomogeneous
Markov process on the set of dyadic times, and its step-function interpolations yield c\`adl\`ag processes with consistent finite-dimensional distributions and partial tightness properties. The present work may also be viewed as an alternative direct probabilistic approximation scheme for the three-dimensional zero-range homopolymer measure of~\cite{CranstonKoralovMolchanovVainberg2}, which is constructed as a weak limit of Gibbs measures associated with regularized Schr\"odinger operators.
\end{abstract}

\vspace{.2cm}

\section{Introduction}

A three-dimensional heat equation with a one-point attractive potential at the
origin may formally be written as
\begin{align} \label{3dHeatEqFormal}
\frac{\partial u}{\partial t}(t,x)
\,=\,
\frac{1}{2}\Delta u(t,x)
+
\frac{\beta}{2}\,\delta_0(x)\,u(t,x),
\qquad t>0,\quad x\in\R^3, 
\end{align}
where $\Delta=\sum_{i=1}^3 \frac{\partial^2}{\partial x_i^2}$ is the Laplacian
on $\R^3$, $\delta_0$ denotes the Dirac mass at the origin, and $\beta>0$ is the
strength of the interaction, corresponding to the attractive regime considered
in this work. The associated formal expression
\begin{align}\label{HamiltonianFormal}
\mathcal L_{\textup{form}}^{\beta}
\,=\,
\Delta+\beta\,\delta_0(x)
\end{align}
may be interpreted heuristically as the three-dimensional Schr\"odinger operator
with a zero-range potential at the origin; see~\cite{AGHH} for a survey of
Schr\"odinger operators with potentials supported on discrete sets; in
particular, see Chapter~II.1 for the three-dimensional case. The factor \(1/2\)
in the heat equation corresponds formally to the generator
\(\frac12\mathcal L_{\textup{form}}^{\beta}\). One way to give rigorous meaning to~\eqref{HamiltonianFormal} is through
approximation by Schr\"odinger operators with regularized potentials supported on shrinking balls around the origin. More precisely, let
$\{\delta_\varepsilon\}_{\varepsilon>0}$ be a $\delta$-sequence on $\R^3$, that is, a family of non-negative functions with unit mass converging to $\delta_0$ in the sense of distributions. For instance, one may take
\[
\delta_\varepsilon(x)
\,:=\,
\frac{3}{4\pi \varepsilon^3}\,\mathbf{1}_{B_\varepsilon(0)}(x),
\]
where $B_\varepsilon(y)$ denotes the open ball in $\R^3$ centered at $y$ with
radius $\varepsilon$. Consider then the family of operators
\begin{align}
\Delta^{\beta}_\varepsilon
\,:=\,
\Delta
+
\lambda_\varepsilon^{\beta}\,\delta_\varepsilon(x), \nonumber
\end{align}
where the coupling parameter is tuned according to
\begin{align}
\lambda_\varepsilon^{\beta}
\,:=\,
\frac{4\pi}{3}
\Big[
\Big(k+\tfrac{1}{2}\Big)^2\pi^2\,\varepsilon
-
8\pi^2\beta\,\varepsilon^2
-
\zeta\,\varepsilon^3
\Big],
\qquad \varepsilon>0, \nonumber
\end{align}
with $k\in\mathbb Z$ and $\zeta\in\R$; see~\cite[Eqs.~(H.73)--(H.75)]{AGHH} for the corresponding approximation and renormalization scheme in their notation. Then the family \(\Delta_\varepsilon^{\beta}\) converges, in an appropriate resolvent sense as \(\varepsilon\downarrow0\), to a self-adjoint point interaction operator, which we denote by \(\Delta^{\beta}\). The limiting operator \(\Delta^{\beta}\) is independent of the auxiliary parameters \(k\) and \(\zeta\), and may be regarded as the rigorous operator associated with the formal expression~\eqref{HamiltonianFormal}. The operator \(\Delta^{\beta}\) acts as the free Laplacian away from the
origin, while the point interaction is encoded through a singular boundary
condition at \(x=0\). More precisely, the domain of the self-adjoint operator \(\Delta^{\beta}\) admits the following decomposition involving the Green kernel $G_k(x) := (4\pi|x|)^{-1} e^{ik|x|}$, where $x\in\R^3\setminus\{0\}$ and $\textup{Im } k>0$,
\[
\mathcal D\big(\Delta^{\beta}\big)
=
\left\{
\phi_k
+
(-\beta-ik/4\pi)^{-1}\phi_k(0)\,G_k
\;:\;
\phi_k\in H^{2,2}(\R^3)\,,
k\in\mathbb C\,,
\textup{Im } k>0\,,
k^2\notin\{-16\pi^2\beta^2\}\cup[0,\infty)
\right\},
\]
see~\cite[Thm.~1.1.3]{AGHH} or~\cite[Eq.~(2.1)]{GrummtKolb}. In particular, functions in the domain of \(\Delta^{\beta}\) admit
the asymptotic expansion
\begin{align}\label{PsiAsymptotics}
\psi(x)
\,\stackrel{|x|\downarrow0}{=}\,
\mathbf c_{\psi}
\left(
\frac{1}{|x|}
-
4\pi\beta
\right)
+
o(1)  \,, \hspace{.4cm}\text{where} \,\hspace{.2cm}
\mathbf c_{\psi}
\,:=\,
\lim_{|x|\downarrow0}
|x|\,\psi(x).
\end{align}
See Appendix~\ref{AppendixProofPsiAsymptotics} for the verification of the above identity and the precise expression for the constant \(\mathbf c_{\psi}\in\mathbb C\) depending on \(\psi\). The operators in the semigroup $\{e^{\frac{t}{2}\Delta^{\beta}}\}_{t\in [0,\infty)}$ have nonnegative integral kernels $P_t^\beta(x,y)$, which serve as the fundamental solutions of the heat equation associated with $\frac{1}{2}\Delta^{\beta}$. More precisely, \(P_t^\beta(x,y)\) is the kernel defined through
\begin{align} \nonumber
\big(e^{\frac{t}{2}\Delta^{\beta}}\varphi \big)(x)
\,=\,
\int_{\R^3} P_t^\beta(x,y)\,\varphi(y)\,dy,
\qquad t>0,\;\; x\in\R^{3}\setminus\{0\},
\end{align}
for every bounded measurable function \(\varphi\) with compact support in \(\R^{3}\setminus\{0\}\). Since the family $\{e^{\frac{t}{2}\Delta^{\beta}}\}_{t\in [0,\infty)}$ forms a semigroup, i.e.\ \(e^{\frac{t}{2}\Delta^{\beta}} e^{\frac{s}{2}\Delta^{\beta}} = e^{\frac{t+s}{2}\Delta^{\beta}}\), the associated kernels in the family \(\{P_t^\beta\}_{t\in [0,\infty)}\) satisfy the semigroup property: for every \(s,t>0\) and every \(x,z\in\R^3\setminus\{0\}\),
\begin{align}\label{EqSemigroupKernelReliable}
\int_{\R^3} P_s^\beta(x,y)\,P_t^\beta(y,z)\,dy
\,=\,
P_{s+t}^\beta(x,z).
\end{align}
In dimension three, the attractive point interaction heat kernel admits the explicit representation
\begin{align}\label{DefPointKer3dBeta}
P_t^{\beta}(x,y)
\,:=\,
P_t(x,y)
+
\frac{2t}{|x||y|}\,P_t(|x|+|y|)
+
\frac{8\pi\beta\,t}{|x||y|}
\int_{0}^{\infty} e^{\,4\pi\beta u}\,
P_t\!\big(u+|x|+|y|\big)\,du,
\end{align}
for \(t>0\) and \(x,y\in\R^3\setminus\{0\}\). Here \(P_t(x,y)\) denotes the three-dimensional free heat kernel, and, with a slight abuse of notation, \(P_t(r)\) denotes its radial version, that is
\begin{align}\label{DefFreeHeat3d}
P_t(x,y)
\,:=\,
\frac{1}{(4\pi t)^{\frac{3}{2}}}\,e^{-\frac{|x-y|^{2}}{4t}},
\quad \textup{and} \quad
P_t(r)
\,:=\,
\frac{1}{(4\pi t)^{\frac{3}{2}}}\,e^{-\frac{r^{2}}{4t}},
\end{align}
for \(r\ge0\). The heat kernel representation~\eqref{DefPointKer3dBeta} appears in the
analysis of~\cite{Albeverio}; see, in particular, the formulas in
Section~3 therein for the case of a single point interaction in
dimension three. The explicit expression~\eqref{DefPointKer3dBeta} can also be found in~\cite[Eq.~(37)]{Fleischmann} and~\cite[Eq.~(2.2)]{GrummtKolb};
see also~\cite[Eqs.~(3.1)--(3.4)]{APT} for related propagator
formulas for the three-dimensional point interaction Schr\"odinger equation. In these works, the point interaction is parameterized by \(\alpha\in\R\); the
attractive regime considered in the present work corresponds to the choice \(\alpha=-\beta<0\). See Appendix~\ref{AppendixProofofSemigroupProperty} for a direct proof of the semigroup property~\eqref{EqSemigroupKernelReliable}.

Fix $T>0$. For $t \in [0,T]$, define the function $Q_t^\beta:\R^3 \to [0,\infty]$ by
\begin{align}\label{DefH3d}
Q_t^\beta(x)
\,:=\,
\int_{\R^3} P_t^\beta(x,y)\,dy-1
\,=\,
\frac{2t}{\beta\,|x|}
\int_{0}^{\infty}
\big(e^{4\pi\beta w}-1\big)\,
P_t(|x|+w)\,dw
\end{align}
for \(x\in\R^3\setminus\{0\}\), while \(Q_t^\beta(0):=\infty\). The second equality follows from~\eqref{DefH3d2nd}; see also~\eqref{DefH3dIst} and~\eqref{DefH3d3rd} for equivalent representations of \(Q_t^\beta(x)\). For \(0\le s<t\le T\), define the function $\mathlarger{p}^{\,T,\beta}_{s,t}:
(\R^3\setminus\{0\})\times(\R^3\setminus\{0\})
\to [0,\infty]$ by the Doob transform of the heat kernel \(P_{t-s}^\beta(x,y)\) as follows:
\begin{align}\label{FirstTrans3d}
\mathlarger{p}^{\,T,\beta}_{s,t}(x,y)
\,:=\,
\frac{1+Q^{\beta}_{T-t}(y)}{1+Q^{\beta}_{T-s}(x)}
\, P^{\beta}_{t-s}(x,y),
\qquad x,y\in\R^3\setminus\{0\}.
\end{align}
The proof of the following lemma is in Section~\eqref{SubProofLemTranKern}.
\begin{lemma}\label{LemTranKern}
Fix $T,\beta>0$.  The family $\{\mathlarger{p}_{s,t}^{T,\beta}\}_{0\le s<t\le T}$ forms a family of transition probability densities on $[0,T]$. In particular, for all $0\le r<s<t\le T$ and $x,z\in\R^3 \setminus \{0\}$,
\begin{enumerate}[(i)]
    \item $ \int_{\R^3} \,
\mathlarger{p}_{s,t}^{T,\beta}(x,y)\,dy
 \,= \,1 \,,$

 \item $\int_{\R^3} \, \mathlarger{p}_{r,s}^{T,\beta}(x,y)\,
\mathlarger{p}_{s,t}^{T,\beta}(y,z)\,dy
 \,= \,
\mathlarger{p}_{r,t}^{T,\beta}(x,z) \,.$
\end{enumerate}
\end{lemma}
The goal of the present work is to construct, in a probabilistic framework, a time-inhomogeneous Markov family $\mathbf{X}=\{\mathbf{X}_t\}_{t \in [0,T]}$ with law $\mathbf{P}_x^{T,\beta}$ associated with the attractive ($\beta>0$) point interaction on a finite time horizon ($T>0$) in dimension three, that is, with transition densities $p^{T,\beta}_{s,t}(x,y)$ given by~\eqref{FirstTrans3d}. A direct construction of such a process is delicate due to the singular nature of the interaction and the absence of a straightforward pathwise description. Instead, we approach this problem in two stages:
\begin{enumerate}
\item[--] For fixed parameters $T,\beta,\varepsilon>0$, we consider the punctured domain $E_\varepsilon$ together with its extension $\overline E_\varepsilon$ given by
\[
E_\varepsilon\,:=\,\{x\in\mathbb{R}^3:\ |x|>\varepsilon\},
\qquad
\overline E_\varepsilon\,:=\,E_\varepsilon\cup\{\Delta\},
\]
where $\Delta\notin\mathbb{R}^3$ denotes a cemetery state\footnote{Recall that the symbol $\Delta$ was used earlier to denote the Laplacian; from this point onward, except in
Sections~\ref{3DimCase}--\ref{2DimCase}, it denotes the cemetery state unless otherwise stated.}. We aim to construct a time-inhomogeneous Markov family with law $\mathbf{P}_x^{T,\beta,\varepsilon}$ on $\overline E_\varepsilon$ whose transition mechanism is governed by the Doob-transformed densities $p^{T,\beta}_{s,t}(x,y)$ and which is killed upon reaching the boundary $\{|x|=\varepsilon\}$. This intermediate object captures both the singular interaction at the origin and the survival constraint over the finite time horizon.

\item[--] To recover the desired law $\mathbf{P}_x^{T,\beta}$ from
$\mathbf{P}_x^{T,\beta,\varepsilon}$, we aim to pass to the limit as
$\varepsilon\downarrow 0$ in a suitable sense, thereby removing the
artificial boundary while retaining the effect of the point interaction.
\end{enumerate}

The contribution of the present work is to carry out the first step in a canonical and fully probabilistic manner; see Section~\ref{ModelFormulation}. More precisely, we construct the family of transition probability kernels associated with $\mathbf{P}_x^{T,\beta,\varepsilon}$, establish a consistent Markovian structure on a dense set of times via a dyadic skeleton, and develop a sequence of c\`adl\`ag approximations through step-function interpolations together with their induced path-space laws. These results provide the groundwork for the construction of the
corresponding time-inhomogeneous c\`adl\`ag Markov process with law
$\mathbf{P}_x^{T,\beta,\varepsilon}$ on the full time interval $[0,T]$, as well as for the analysis of the limiting procedure
$\varepsilon\downarrow 0$, which are left for future work.

\textit{
After the completion of the first draft of the present work, we became aware that the path-space law $\mathbf{P}_x^{T,\beta}$ had already been constructed in~\cite{CranstonKoralovMolchanovVainberg2}. However, their approach is fundamentally different: they use the norm resolvent convergence of regularized Schr\"odinger operators to obtain the weak convergence of the associated regularized Gibbs measures to $\mathbf{P}_x^{T,\beta}$; see Section~\ref{subsecComparisonZeroRange} for a detailed discussion of their approach, where the law is denoted by \(P_{\gamma,T}^x\) in their notation. Thus, the present work provides an alternative approach to the construction of \(P_{\gamma,T}^x\), and therefore the dyadic construction discussed here may also be viewed as a probabilistic approximation scheme for \(P_{\gamma,T}^x\).
}

\subsection{Related models in dimensions \texorpdfstring{$\boldsymbol{d\geq 3}$}{d>=3}}  \label{3DimCase}

In~\cite{CranstonKoralovMolchanovVainberg1}, a homopolymer model is formulated by perturbing Wiener measure with an exponential weight generated by a compactly supported potential. More precisely, if $P_{0,T}$ denotes Wiener measure on $C([0,T],\mathbb R^3)$ corresponding to Brownian motion started at the origin, then the homopolymer measure $P_{\beta,T}$ is defined by the Gibbs density
\[
\frac{dP_{\beta,T}}{dP_{0,T}}(\omega)
=
\frac{1}{Z_{\beta,T}}
\exp\bigg\{
\beta\int_0^T v(\omega(t))\,dt
\bigg\},
\qquad
\omega\in C([0,T],\mathbb R^3),
\]
where $Z_{\beta,T} = \mathbb E_{0,T}\!\big[ \exp\big\{ \beta\int_0^T v(\omega(t))\,dt \big\} \big]$ is the associated partition function, $v$ is a compactly supported
potential, and $\beta$ is the coupling parameter;
see~\cite[Eq.~(3)]{CranstonKoralovMolchanovVainberg1}.  Next, consider the
heat kernel $p_\beta(t,x,y)$ defined as the fundamental solution of the heat
equation
\begin{align}\label{RegularHE}
\frac{\partial p_\beta}{\partial t}(t,y,x)
=
\frac12\Delta_x p_\beta(t,y,x)
+
\beta v(x)p_\beta(t,y,x),
\qquad
p_\beta(0,y,x)=\delta(x-y),
\end{align}see~\cite[Eq.~(5)]{CranstonKoralovMolchanovVainberg1}.  In
\cite[Thm.~3.1]{CranstonKoralovMolchanovVainberg1}, it is shown that the
resulting homopolymer is a time-inhomogeneous Markov process with transition
density
\[
q^T_\beta((s,y),(t,x))
=
p_\beta(t-s,y,x)\,
\frac{Z_{\beta,T-t}(x)}{Z_{\beta,T-s}(y)}.
\]
See~\cite[Eq.~(6)]{CranstonKoralovMolchanovVainberg1}. This is precisely a Doob transform of the heat kernel $p_\beta(t,x,y)$ by the
future partition function $Z_{\beta,T-t}(x)$, which, by the Feynman--Kac
formula, is given by
\[
Z_{\beta,T-t}(x)
=
\int_{\mathbb R^3} p_\beta(T-t,x,y)\,dy.
\]
Thus the above formula is analogous to the transition density
\eqref{FirstTrans3d}, where the point interaction heat kernel
$P_t^\beta(x,y)$ is transformed by the total-mass function
\[
1+Q^\beta_{T-t}(x)
\, \overset{\eqref{DefH3d}}{:=} \,
\int_{\mathbb R^3}P^\beta_{T-t}(x,y)\,dy.
\]
The essential difference, however, is that the potential $v$ in
\cite{CranstonKoralovMolchanovVainberg1} is regular and compactly supported, so $p_\beta(t,x,y)$ is a regular Schr\"odinger heat kernel on $\mathbb R^3$. By contrast, the present work starts directly from the singular three-dimensional point interaction heat kernel $P_t^\beta(x,y)$ associated with the formal equation~\eqref{3dHeatEqFormal}; in this case the origin is a distinguished singular point, and the corresponding mass function satisfies $Q_t^\beta(0)=\infty$. 

The present work is also related to the recent work of
Wang~\cite{WangSHE} on singular stochastic heat equation models in dimensions $d\geq 3$, as well as to the work of Fitzsimmons and Li~\cite{FitzsimmonsLi} on three-dimensional Brownian motion conditioned to hit the origin. However, the model most closely related to the present setting is the three-dimensional zero-range homopolymer model constructed in~\cite{CranstonKoralovMolchanovVainberg2}, which we discuss in more detail below.

\subsubsection{The three-dimensional zero-range homopolymer model~\texorpdfstring{\cite{CranstonKoralovMolchanovVainberg2}}{}}
\label{subsecComparisonZeroRange}

The paper \cite{CranstonKoralovMolchanovVainberg2} is especially close to the
present work, since it studies a continuous homopolymer in $\mathbb R^3$ with an
attractive zero-range interaction at the origin by replacing the regular
compactly supported potential $v$ by a delta mass $\delta_0$ at the
origin. Consequently, the usual Hamiltonian, Gibbs measure, and Schr\"odinger
operator corresponding to a smooth potential are meaningless in this model. In particular, the right-hand side of~\eqref{RegularHE}
formally becomes
\[
\frac12\Delta+\gamma\delta_0 ,
\]
where the parameter \(\gamma\in\mathbb R\) denotes the strength (coupling parameter)
of the point interaction at the origin. A rigorous realization of the corresponding singular Schr\"odinger operator
is obtained through a self-adjoint extension \(L_\gamma\) of
\[
    \frac12\Delta:
    C_0^\infty(\mathbb R^3\setminus\{0\})
    \longrightarrow
    C_0^\infty(\mathbb R^3\setminus\{0\});
\]
see \cite[Sec.~2]{CranstonKoralovMolchanovVainberg2}. The semigroup kernel
\(p_\gamma(t,x,y)\) of \(e^{tL_\gamma}\) is given by the explicit contour
formula
\begin{align}\label{PointInteractionHeatKernel2}
p_\gamma(t,x,y)
\,=\,
\frac{e^{-|x-y|^2/(2t)}}{(2\pi t)^{3/2}}
+
\frac{1}{4\pi^2 i}
\int_{\Gamma(a)}
\frac{
e^{-\sqrt{2\lambda}(|x|+|y|)+\lambda t}
}{
(\sqrt{2\lambda}-\gamma)|x||y|
}
\,d\lambda ,
\end{align}
where \(x,y\neq0\), \(a>\gamma^2/2\), and \(\Gamma(a) := \{a+i\tau: \tau\in\mathbb R\}\) is the vertical contour
passing through \(a\); see~\cite[Eq.~(6)]{CranstonKoralovMolchanovVainberg2}. Since \(p_\gamma(t,x,y)\) is the point interaction heat kernel, the contour representation~\eqref{PointInteractionHeatKernel2} is equivalent to~\eqref{DefPointKer3dBeta}; see
Remark~\ref{RemarkHeatKernelsEq}. Thus \(p_\gamma\) is the fundamental solution of
\[
    \frac{\partial u}{\partial t}
    =
    L_\gamma u .
\]
For \(t>0\) and \(x\neq0\), the corresponding zero-range partition function is
\begin{align} 
    Z_\gamma(t,x)
    :=
    \int_{\mathbb R^3}p_\gamma(t,x,y)\,dy
    =
    1+
    \frac{1}{2\pi i}
    \int_{\Gamma(a)}
    e^{\lambda t}
    \frac{1}{\sqrt{2\lambda}-\gamma}
    \frac{e^{-\sqrt{2\lambda}|x|}}{\lambda |x|}
    \,d\lambda , \nonumber
\end{align}
where the second equality follows from~\cite[Eq.~(7)]{CranstonKoralovMolchanovVainberg2}. The associated zero-range
Gibbs measure \(P_{\gamma,T}^x\) is then defined through its
finite-dimensional distributions by
\begin{align}
& P_{\gamma,T}^x
  \bigl(\omega(t_1)\in A_1,\ldots,\omega(t_k)\in A_k\bigr)
  \notag\\
&\quad =
  Z_\gamma(T,x)^{-1}
  \int_{A_1\times\cdots\times A_k\times\mathbb R^3}
  \left(
  \prod_{i=1}^{k}
  p_\gamma(t_i-t_{i-1},x_{i-1},x_i)
  \right)
  p_\gamma(T-t_k,x_k,y)
  \,
  dy\,dx_k\cdots dx_1 , \nonumber
\end{align}
where \(x_0:=x\), \(t_0:=0\), and \(p_\gamma(0,x,y)=\delta_x(y)\).

Recall the approximation scheme below~\eqref{HamiltonianFormal} that the
zero-range Schr\"odinger operator \(L_\gamma\) may also be obtained as the
norm resolvent limit of regularized Schr\"odinger operators
\(L_\gamma^\varepsilon\) of the form
\[
    L_\gamma^\varepsilon
    =
    \frac12\Delta+v_\gamma^\varepsilon .
\]
Here \(v_\gamma^\varepsilon\) is a suitably tuned approximation of the
zero-range potential; see \cite[Eq.~(5)]{CranstonKoralovMolchanovVainberg2}
for the corresponding asymptotic scaling in their setting. Using this resolvent convergence, the authors prove the convergence of the associated partition functions and finite-dimensional distributions, together with the tightness needed to conclude the weak convergence of the regularized Gibbs measures
\[
    P_{\gamma,T}^{x,\varepsilon}
    \Longrightarrow
    P_{\gamma,T}^{x}
\]
on \(C([0,T],\mathbb R^3)\) as \(\varepsilon\downarrow0\); see
\cite[Thm.~2.3 and Sec.~5]{CranstonKoralovMolchanovVainberg2}.

\begin{remark} \label{RemarkHeatKernelsEq}
The kernel \(p_\gamma(t,x,y)\) in~\eqref{PointInteractionHeatKernel2} coincides
with the point interaction heat kernel \(P_t^\beta(x,y)\) in~\eqref{DefPointKer3dBeta} under the identification
\(\gamma=4\pi\beta\) and the time scaling \(t\mapsto t/2\). To see this, recall from~\eqref{DefFreeHeat3d} that $P_t(x,y) := (4\pi t)^{-\frac{3}{2}}e^{-\frac{|x-y|^{2}}{4t}}$ denotes the three dimensional free heat kernel. Writing \(R:= |x|+|y|\) and using the identity
\[
\frac{1}{\sqrt{2\lambda}-\gamma}
=
\frac{1}{\sqrt{2\lambda}}
+
\frac{\gamma}{\sqrt{2\lambda}(\sqrt{2\lambda}-\gamma)},
\]
we split the contour term in~\eqref{PointInteractionHeatKernel2} into two parts.  
\begin{align}
p_\gamma(t,x,y)
&=
P_{t/2}(x,y)
+
\frac{1}{4\pi^2 i |x||y|}
\int_{\Gamma(a)}
\frac{
e^{-\sqrt{2\lambda}R+\lambda t}
}{
\sqrt{2\lambda}
}
\,d\lambda
+
\frac{\gamma}{4\pi^2 i |x||y|}
\int_{\Gamma(a)}
\frac{
e^{-\sqrt{2\lambda}R+\lambda t}
}{
\sqrt{2\lambda}(\sqrt{2\lambda}-\gamma)
}
\,d\lambda \nonumber
\end{align}
For the second term, observe that $F(\lambda) = \frac{e^{-\sqrt{2\lambda}R}}{\sqrt{2\lambda}}$ is the Laplace transform of $f(t) = \frac{1}{\sqrt{2\pi t}}
e^{-R^2/(2t)}$. Applying the inverse Laplace transform formula $f(t)= \frac{1}{2\pi i} \int_{\Gamma(a)} e^{\lambda t}F(\lambda)\,d\lambda$ therefore yields the first equality below.
\begin{align} \label{LaplaceCOmputation}
\frac{1}{4\pi^2 i |x||y|}
\int_{\Gamma(a)}
\frac{
e^{-\sqrt{2\lambda}R+\lambda t}
}{
\sqrt{2\lambda}
}
\,d\lambda
=
\frac{1}{2\pi |x||y|}
\frac{1}{\sqrt{2\pi t}}
e^{-R^2/(2t)}
=
\frac{t}{|x||y|}P_{t/2}(R)
\end{align}
For the last term, since $\operatorname{Re}\sqrt{2\lambda}>\gamma$ on the chosen contour, we have
\[
\frac{1}{\sqrt{2\lambda}-\gamma}
=
\int_0^\infty e^{-u(\sqrt{2\lambda}-\gamma)}\,du
=
\int_0^\infty e^{\gamma u}e^{-u\sqrt{2\lambda}}\,du .
\]
Substituting this representation into the last contour integral and applying Fubini's theorem, we obtain
\begin{align}
\frac{\gamma}{4\pi^2 i |x||y|}
\int_{\Gamma(a)}
\frac{
e^{-\sqrt{2\lambda}R+\lambda t}
}{
\sqrt{2\lambda}(\sqrt{2\lambda}-\gamma)
}
\,d\lambda 
= &
\frac{\gamma}{|x||y|}
\int_0^\infty e^{\gamma u}
\bigg[
\frac{1}{4\pi^2 i}
\int_{\Gamma(a)}
\frac{
e^{-\sqrt{2\lambda}(R+u)+\lambda t}
}{
\sqrt{2\lambda}
}
\,d\lambda
\bigg]du \nonumber \\
=&
\frac{\gamma t}{|x||y|}
\int_0^\infty
e^{\gamma u}P_{t/2}(R+u)\,du , \nonumber
\end{align}
where the last equality follows from~\eqref{LaplaceCOmputation} with \(R\) replaced by \(R+u\). Combining the above identities and using \(R=|x|+|y|\), we obtain
\[
p_\gamma(t,x,y)
=
P_{t/2}(x,y)
+
\frac{t}{|x||y|}
P_{t/2}(|x|+|y|)
+
\frac{\gamma t}{|x||y|}
\int_0^\infty
e^{\gamma u}
P_{t/2}(u+|x|+|y|)
\,du .
\]
Comparing this expression with~\eqref{DefPointKer3dBeta} and taking $\gamma=4\pi\beta$, we obtain $p_{4\pi\beta}(t,x,y) = P_{t/2}^{\beta}(x,y)$. Equivalently after the time scaling \(t\mapsto t/2\), we have
\[
P_t^\beta(x,y)
=
p_{4\pi\beta}(2t,x,y).
\]
\end{remark}

\subsection{Related two-dimensional model \texorpdfstring{$\boldsymbol{(d=2)}$}{(d=2)}} \label{2DimCase}

The probabilistic framework considered in the present work is motivated by the two-dimensional model studied in~\cite{CM}, where a Doob transform is applied to planar Brownian motion in the presence of a singular point interaction. In that setting, one seeks to describe planar Brownian motion subject to an attractive zero-range potential at the origin. As in the three-dimensional setting discussed earlier, such an interaction corresponds formally to a zero-range Schr\"odinger operator of the form
\begin{align}
\mathcal{L}
\, = \,
-\frac{1}{2}\,\Delta
\,- \,
\vsm\,\delta(x)\, , \nonumber
\end{align}
on $\R^2$, where $\Delta$ is the two-dimensional Laplacian, $\delta(x)$ is the Dirac delta function on $\R^2$, and $\vsm>0$ is a coupling constant. However, this description is also heuristic, since planar Brownian motion almost surely does not hit a fixed point. To give a rigorous meaning to this interaction, fix $\varepsilon>0$ and consider a family of mollified potentials of the form $\delta_\varepsilon(x) = \frac{1}{\varepsilon^2}\,D\!\left(\frac{x}{\varepsilon}\right)$, where $x\in\R^2$ and $D\in C_c(\R^2)$ is a non-negative function with unit integral. Next, for  each parameter value  $\lambda\in (0,\infty)$, consider the family of Schr\"odinger-type operators $\mathcal L_\varepsilon^\lambda := \Delta + 2\,\mathsmaller{V}^{\lambda}_{\varepsilon}\,\delta_\varepsilon(x)$, where, in contrast to the three-dimensional setting, the coupling
constant $\mathsmaller{V}^{\lambda}_{\varepsilon}$ vanishes according to the logarithmic renormalization asymptotics
\begin{align}
\vsm^{\lambda}_{\varepsilon}\,\stackrel{\varepsilon\rightarrow 0}{=}\,\frac{ \pi  }{ \log \frac{1}{\varepsilon} }\bigg(1+ \frac{ \frac{1}{2}\log \frac{\lambda}{2}+\gamma_{\mathsmaller{\textup{EM}}} +I_{\Da} }{ \log \frac{1}{\varepsilon} }\bigg) +\mathit{o}\bigg(\frac{ 1}{ \log^2 \frac{1}{\varepsilon} } \bigg)\,, \nonumber
\end{align}
for $I_{\Da}:=\int_{(\R^2)^2}\log |x-y| \Da(x)\Da(y)dxdy $ and  $\gamma_{\mathsmaller{\textup{EM}}}:=-\int_0^{\infty}e^{-x} \log x\,dx $ being the Euler-Mascheroni constant. This renormalization yields a non-trivial limiting zero-range
Schr\"odinger operator $\mathcal L^\lambda$ in dimension two, in the
norm resolvent sense. The functions $\psi:\R^2\to\mathbb C$ in the domain of the
two-dimensional limiting operator \(\mathcal L^\lambda\) admit the
logarithmic asymptotic expansion
\begin{align}
\psi(x)
\,\stackrel{|x|\downarrow0}{=}\,
\mathbf c_{\psi}
\left(
\log |x|
+
\frac{1}{2}\log\frac{\lambda}{2}
+
\gamma_{\mathsmaller{\textup{EM}}}
\right)
+
o(1)  \,, \hspace{.4cm}\text{where} \,\hspace{.2cm}
\mathbf c_{\psi}
\,:=\,
\lim_{|x|\downarrow0}
\frac{\psi(x)}{\log |x|} \,.
\nonumber
\end{align}
This asymptotic behavior corresponds to the logarithmic singularity of the two-dimensional Green function $G(x):=-\frac{1}{2\pi}\log|x|$ for $x\in\R^2\setminus\{0\}$ of the Laplacian on \(\R^2\); see~\cite[Thm.~5.3]{AGHH} for a precise characterization of the domain of \(\mathcal L^\lambda\). The associated semigroup $\{e^{t\mathcal{L}^{\lambda}}\}_{t\ge0}$ admits a positive integral kernel $f_t^\lambda:\R^2\times\R^2\to[0,\infty]$, given explicitly by
\begin{align}
f_t^{\lambda}(x,y)
\,=\,
g_t(x-y)
+
2\pi\lambda
\int_{0<r<s<t}
g_r(x)\,
\nu'\!\big(\lambda(s-r)\big)\,
g_{t-s}(y)\,ds\,dr, \nonumber
\end{align}
where $g_t(x):=(2\pi t)^{-1}e^{-\frac{|x|^2}{2t}}$ is the two-dimensional heat kernel and $\nu'$ denotes the derivative of the Volterra function $\nu(a):=\int_0^\infty \frac{a^s}{\Gamma(s+1)}\,ds$. This representation for $f^{\lambda}_t(x,y)$ coincides with~\cite[Eq.~(2.7)]{GQT} and has been used in~\cite{Chen1,CSZ5}. It is equivalent, up to a reparametrization of the integration domain, to~\cite[Eq.~(3.11)]{Albeverio}; this equivalence is established in~\cite[Rmk.~2.1]{GQT}. A direct proof of the semigroup property of $\{ f^{\lambda}_t(x,y) \}_{t \in [0,\infty)}$ is given in~\cite[Lem.~1.6.2]{Mian}.

As in the three-dimensional framework, the corresponding singular diffusion is obtained through a Doob transform of the underlying heat kernel. Fix $T,\lambda >0$, and let $0\leq s<t\leq T$. The function $\mathlarger{d}_{s,t}^{T,\lambda} : \R^2\times \R^2 \rightarrow (0,\infty]$, defined by the Doob transform of the integral kernel $f_{t-s}^{\lambda}(x,y)$ as
\begin{align}\label{CMDensity}
\mathlarger{d}_{s,t}^{T,\lambda}(x,y)
\,:=\,
f_{t-s}^{\lambda}(x,y)\,
\frac{
\int_{\R^2} f_{T-t}^{\lambda}(y,z)\,dz
}{
\int_{\R^2} f_{T-s}^{\lambda}(x,z)\,dz
}\,,
\end{align}
is a transition density function; see~\cite[Lem.~2.1]{CM}. The family of densities $\{\mathlarger d_{s,t}^{T,\lambda}\}_{0\le s<t\le T}$ determines a singular diffusion $X=\{X_t\}_{t\in [0,T]}$ with law $\mathbb P_x^{T,\lambda}$ on the space of continuous paths (see~\cite[Prop.~2.2]{CM}), which exhibits several characteristics that are anomalous when compared to planar Brownian motion.
\begin{itemize}
    \item[--] The process $X$ under $\mathbb P_x^{T,\lambda}$ visits the origin over the time interval $(0,T)$ with strictly positive probability,
\[
0 \,<\, \mathbb{P}_{x}^{T,\lambda}\!\big(
X_t = 0 \ \text{for some } t\in[0,T]
\big) \,<\, 1\, .
\]
See~\cite[Thm.~2.4]{CM}. In particular, the path measure $\mathbb{P}_{x}^{T,\lambda}$ is not absolutely continuous with respect to the planar Wiener measure. Furthermore, over the finite time interval $[0,T]$, the origin is neither absorbing nor avoidable, reflecting the critical nature of the two-dimensional point interaction.

    \item[--] An alternative characterization of the process \(X\) is provided by the SDE
\begin{align}
dX_t
\,=\,
dW_t + b_{T-t}^{\lambda}(X_t)\,dt\,, \hspace{1cm}
t\in[0,T]\,, \nonumber
\end{align}
where $\{W_t\}_{t\in[0,T]}$ is a two-dimensional Brownian motion
starting from $x\in\R^2$, and the time-dependent drift
$b_t^\lambda(x)\in\R^2$ is given by the gradient with respect to
$x\in\R^2$ of
$\log \big(\int_{\R^2} f_{t}^{\lambda}(x,z)\,dz \big)$; see
~\cite[Prop.~2.6]{CM}. Without the time dependence in
the drift, the above SDE would correspond to a distorted Brownian
motion~\cite{Fukushima,Streit,Trutnau}, while related multidimensional
time-dependent SDEs with singular drifts have also been studied~\cite{Krylov,Zhang,Jin,Kinzebulatov}, although none of the existing
general weak existence and uniqueness results apply directly to the
above SDE.

    \item[--] For fixed $T,\lambda>0$ and $x\in\R^2$, there exists a continuous process  $L=\{L_t\}_{t\in[0,T]}$ which arises as the limit, as  $\varepsilon\searrow 0$, of the family $\{L_t^\varepsilon\}_{t\in[0,T]}$ in the sup-$L^1\big(\mathbb{P}^{T,\lambda}_{x}\big)$ sense, where
\begin{align}
L_t^{\varepsilon}
\,:=\,
\frac{1}{2\varepsilon^2\log^2 \frac{1}{\varepsilon}}
\,\textup{meas}\Big(\big\{r\in[0,t]: |X_r|\le\varepsilon\big\}\Big), \nonumber
\end{align}
and $\textup{meas}(E)$ denotes the Lebesgue measure of a set $E\subset\R$. 
The limiting process $L$ is referred to as the \textit{local time at the origin}; see~\cite[Thm.~2.10]{CM}.

    \item[--] Almost surely, the \textit{zero set} $\{t\ge 0 :X_t=0\}$ is uncountable, has Hausdorff dimension $0$, and carries the full mass of the local time measure $\vartheta(\omega,\cdot)$ defined by $\vartheta(\omega,[0,t])=L_t(\omega)$; see~\cite[Prop.~2.26]{CM}.
\end{itemize}
The law $\mathbb P_x^{T,\lambda}$ also arises naturally in connection
with the second moment measures associated with the continuum polymer
measures corresponding to the critical two-dimensional stochastic heat flow~\cite{CSZ5}; see~\cite[Sec.~5.3]{CM2}. The Doob transform~\eqref{CMDensity} corresponds to the choice of driving family $h_t^{\lambda}(x):=(f_t^{\lambda}\mathbf{1})(x)$, obtained by integrating the kernel $f_t^{\lambda}(x,y)$ against Lebesgue measure. In~\cite{Mian2}, a general class of driving families is considered, which leads to an axiomatic characterization of planar diffusions with a point interaction at the origin. In particular, the choice $h_t^{\lambda}(x):=e^{\lambda t}\,K_0(\sqrt{2\lambda}\,|x|)$, where $K_0$ denotes the modified Bessel function of the second kind of order zero, formally yields the diffusion constructed in~\cite{Chen2}; see~\cite[Sec.~3]{Mian2}.

\section{Model formulation}\label{ModelFormulation}

The explicit representation~\eqref{DefPointKer3dBeta} of the three-dimensional point interaction heat kernel $P_t^{\beta}(x,y)$, together with the Doob transform~\eqref{FirstTrans3d}, serves as the basis for a probabilistic construction of diffusions subject to a singular interaction at the origin, as in the two-dimensional framework discussed in Section~\ref{2DimCase}. However, the underlying singular structure in dimension three differs substantially from the planar case. In particular, whereas the two-dimensional interaction is characterized by logarithmic singularities associated with the Green function $-\frac{1}{2\pi}\log|x|$, the three-dimensional kernel $P_t^\beta(x,y)$ contains singular terms of order $(|x||y|)^{-1}$ near the origin, reflecting the stronger zero-range interaction encoded by the operator $\Delta^\beta$. In the present work, the Doob-transformed transition structure induced by the singular kernel $P_t^\beta(x,y)$ is further combined with a survival constraint on a punctured domain $E_\varepsilon$, which necessitates a careful construction of the corresponding Markov family.

Fix \(T,\beta,\varepsilon>0\) and recall the punctured domain $E_\varepsilon:=\{x\in\mathbb{R}^3:\ |x|>\varepsilon\}$. Since \(E_\varepsilon\) is an open (hence \(G_\delta\)) subset of the Polish space \(\R^3\),  it is itself Polish; see \cite[Thm.~3.11]{Kechris}. In particular, the metric
\begin{align} \label{dMetric}
d_\varepsilon(x,y)
\,:=\,
|x-y|
+
\bigg|
\frac{1}{|x|-\varepsilon}
-
\frac{1}{|y|-\varepsilon}
\bigg|,
\qquad x,y\in E_\varepsilon,
\end{align}
is complete on \(E_\varepsilon\) and induces the Euclidean topology. 
The metric \(d_\varepsilon\) is related to the complete-metric construction used in the proof of \cite[Thm.~3.11]{Kechris}, in the case of the closed set \(F=\mathbb{R}^3\setminus E_\varepsilon\). 
We equip \(E_\varepsilon\) with the Borel \(\sigma\)--algebra \(\mathcal B(E_\varepsilon)\). With this state space in place, we aim to construct a time-inhomogeneous Markov process whose transition mechanism is given by the Doob-transformed density \(p^{T,\beta}_{s,t}(x,y)\) defined in~\eqref{FirstTrans3d}, and which is killed upon hitting the boundary \(\{|x|=\varepsilon\}\). The main contribution of the present work is to construct the associated transition probability kernels (Section~\ref{TransitionProbabilityKernel}) and to develop a consistent Markovian framework on a dense set of times (Section~\ref{TheAssociatedMarkovProcess}).

\subsection{Transition probability kernel} \label{TransitionProbabilityKernel}

To construct the transition kernel, we first define, for each partition $\pi$ of $[s,t]$, a sub-probability kernel $\widehat{\mathsf K}^{T,\beta,\varepsilon,\pi}_{[s,t]}$ on $E_\varepsilon$. 
We then pass to the limit along the canonical sequence of global dyadic partitions $\{\pi_m[s,t]\}_{m\ge1}$ to obtain a sub-probability kernel $\widehat{\mathsf K}^{T,\beta,\varepsilon}_{[s,t]}$ on $E_\varepsilon$. 
Finally, we enlarge the state space by adjoining a cemetery state $\Delta$ and extend this limit to a transition probability kernel $\mathsf K^{T,\beta,\varepsilon}_{[s,t]}$, which satisfies the Chapman--Kolmogorov property on dyadic times. 
The construction is summarized schematically below and is discussed in detail in the subsequent three steps.
\begin{center}
\begin{tikzpicture}[scale=1]

\tikzset{
box/.style={
draw,
rectangle,
minimum width=2cm,
minimum height=0.9cm,
align=center
}
}

\node[box] (A) at (-0.5,0)
{
$\widehat{\mathsf K}^{T,\beta,\varepsilon,\pi}_{[s,t]}$
};

\draw[->, thick] (1.5,0) -- (2.8,0)
node[midway, above=2pt] {\scriptsize limit along $\pi_m[s,t]$};

\node[box] (B) at (4.8,0)
{
$\widehat{\mathsf K}^{T,\beta,\varepsilon}_{[s,t]}$
};

\draw[->, thick] (6.3,0) -- (7.6,0)
node[midway, above=2pt] {\scriptsize adjoin $\Delta$};

\node[box] (C) at (9.0,0)
{
$\mathsf K^{T,\beta,\varepsilon}_{[s,t]}$
};

\end{tikzpicture}
\end{center}

\vspace{.3cm}

\noindent \textbf{Step 1$^{\boldsymbol{\circ}}$ (Sub-probability kernels along partitions).} For fixed $0\le s<t\le T$ and a partition $\pi=\{t_0=s<t_1<\cdots<t_n=t\}$ of $[s,t]$, define the map $\widehat{\mathsf K}^{T,\beta,\varepsilon,\pi}_{[s,t]}:
E_\varepsilon \times \mathcal B(E_\varepsilon)\to [0,1]$
given, for $x\in E_\varepsilon$ and $A\in\mathcal B(E_\varepsilon)$, by the product
\begin{align}\label{FirstTrans3dSub}
\widehat{\mathsf K}^{T,\beta,\varepsilon,\pi}_{[s,t]}(x,A)
\,:=\,
\int_{(E_\varepsilon)^n}
\mathbf 1_A(y_n)\,
\prod_{i=1}^{n}
p^{\,T,\beta}_{t_{i-1},t_i}(y_{i-1},y_i)\,
dy_1\cdots dy_n \,,
\end{align}
where we set $y_0:=x$ and $\{p_{s,t}^{T,\beta}(x,y)\}_{0\le s<t\le T}$ is the family of transition probability densities given by~\eqref{FirstTrans3d} on $\R^3\setminus\{0\}$. The map $\widehat{\mathsf K}^{T,\beta,\varepsilon,\pi}_{[s,t]}$ defined in~\eqref{FirstTrans3dSub} is a sub-probability kernel on $E_\varepsilon$ by Lemma~\ref{LemMonotoneOnEeps}. \vspace{.3cm}

\noindent \textbf{Step 2$^{\boldsymbol{\circ}}$ (Limiting sub-probability kernel).} 
The sub-probability kernel $\widehat{\mathsf K}^{T,\beta,\varepsilon,\pi}_{[s,t]}$ in~\eqref{FirstTrans3dSub} describes the evolution from time $s$ to $t$ along the discrete time grid specified by $\pi$. The next natural step is to remove this $\pi$-dependence to obtain a sub-probability kernel $\widehat{\mathsf K}^{T,\beta,\varepsilon}_{[s,t]}$ that captures the evolution over the full time interval $[s,t]$. 
We achieve this by passing to the limit along a canonical sequence of global dyadic partitions.

For $m\ge1$, define the global dyadic grid $\mathcal D_m$ in $[0,T]$ and the set of all dyadic times $\mathbb D_T$ by
\begin{align}\label{dyadicGrid}
\mathcal D_m\,:=\,\big\{k2^{-m}T:\ k=0,1,\dots,2^m\big\}\,,
\quad 
\mathbb D_T \,:=\, \bigcup_{m\ge1} \mathcal D_m\,.
\end{align}
For $0\le s<t\le T$, define the induced partition of $[s,t]$ by
\begin{align}\label{DefGlobalDyadics}
\pi_m[s,t] \, := \, \{s,t\}\cup\big(\mathcal D_m\cap(s,t)\big)\,,
\end{align}
listed in increasing order. The partitions $\{\pi_m[s,t]\}_{m\ge1}$ are nested, with
$\pi_{m+1}[s,t]\succeq\pi_m[s,t]$ and $|\pi_m[s,t]|\downarrow0$. See the beginning of Section~\ref{SecLimitingSubProbKernel} for the verification of these facts, together with explicit examples of the first few grids and partitions. Next, consider the map $\widehat{\mathsf K}^{T,\beta,\varepsilon}_{[s,t]}:\;
E_\varepsilon \times \mathcal B(E_\varepsilon)\to [0,1]$
by, for $x\in E_\varepsilon$ and $A\in\mathcal B(E_\varepsilon)$,
\begin{align}\label{EqDefGlobalDyadicLimitKernel_Inf}
\widehat{\mathsf K}^{T,\beta,\varepsilon}_{[s,t]}(x,A)
\,:=\,
\inf_{m\ge1} \, \widehat{\mathsf K}^{T,\beta,\varepsilon,\pi_m[s,t]}_{[s,t]}(x,A)
\,=\,
\lim_{m\to\infty}\, \widehat{\mathsf K}^{T,\beta,\varepsilon,\pi_m[s,t]}_{[s,t]}(x,A) \,.
\end{align}
The map $\widehat{\mathsf K}^{T,\beta,\varepsilon}_{[s,t]}$ is a sub-probability kernel on $E_\varepsilon$ and satisfies the Chapman--Kolmogorov property on dyadic times; see Proposition~\ref{PropExistLimitSubprobKernelGlobalDyadic}. \vspace{.3cm}

\noindent \textbf{Step 3$^{\boldsymbol{\circ}}$ (Extension to the transition probability kernel).} Let $\Delta\notin\mathbb R^3$ be a cemetery state and set $\overline E_\varepsilon:=E_\varepsilon\cup\{\Delta\}$. We equip $\overline E_\varepsilon$ with the bounded metric
\begin{align} \label{rhoMetric}
\rho(x,y)
\,=\,
\bigl(1\wedge d_\varepsilon(x,y)\bigr)\mathbf 1_{\{x,y\in E_\varepsilon\}}
\,+\,
\mathbf 1_{\{x\neq y,\ \Delta\in\{x,y\}\}}\,,
\qquad x,y\in\overline E_\varepsilon \, ,
\end{align}
where $d_{\varepsilon}$ is the complete metric defined in~\eqref{dMetric}, under which $(E_\varepsilon,d_\varepsilon)$ is Polish. Consequently, $(\overline E_\varepsilon,\rho)$ is also a Polish space, and the corresponding Borel $\sigma$--algebra is $\mathcal B(\overline E_\varepsilon) = \sigma\big(\mathcal B(E_\varepsilon),\{\Delta\}\big)$. Define the map $\mathsf K^{T,\beta,\varepsilon}_{[s,t]}:\;
\overline E_\varepsilon \times \mathcal B(\overline E_\varepsilon) \to [0,1]$
by, for $x\in \overline E_\varepsilon$ and $B\in\mathcal B(\overline E_\varepsilon)$,
\begin{align}\label{EqDefMeshLimitFull}
\mathsf K^{T,\beta,\varepsilon}_{[s,t]}(x,B)
\,:=\,
\mathbf 1_{E_\varepsilon}(x)\Big[
\widehat{\mathsf K}^{T,\beta,\varepsilon}_{[s,t]}(x,B\cap E_\varepsilon)
+
\mathbf 1_{\{\Delta\in B\}}
\big(1-\widehat{\mathsf K}^{T,\beta,\varepsilon}_{[s,t]}(x,E_\varepsilon)\big)
\Big]
+\mathbf 1_{\{x=\Delta\}}\mathbf 1_{\{\Delta\in B\}}.
\end{align}
The following proposition, which we prove in Section~\ref{SecTransitionProbKernel}, shows that the family $\bigl\{\mathsf K^{T,\beta,\varepsilon }_{[s,t]}\bigr\}_{\substack{s,t\in\mathbb D_T \\ s<t}}$ consists of time-inhomogeneous Markov transition probability kernels on $\overline E_\varepsilon$.
\begin{proposition} \label{PropCKFullProbKernel}
Fix $T,\beta,\varepsilon>0$ and $0\le r< s<t\le T$.
\begin{enumerate}[(i)]
\item The map $\mathsf K^{T,\beta,\varepsilon}_{[s,t]}
:\; \overline E_\varepsilon\times\mathcal B(\overline E_\varepsilon)
\longrightarrow [0,1]$ defined in~\eqref{EqDefMeshLimitFull} is a probability kernel.

\item If further $r,s,t\in\mathbb D_T$, then the family $\big \{\mathsf K^{T,\beta,\varepsilon}_{[u,v]}\big\}_{\substack{u,v\in\mathbb D_T\\ u<v}}$ satisfies the Chapman-Kolmogorov identity on $\overline E_\varepsilon$, meaning that, for every
$x\in\overline E_\varepsilon$ and every $A\in\mathcal B(\overline E_\varepsilon)$,
\begin{align}\label{EqCKExtendedProbKernel}
\mathsf K^{T,\beta,\varepsilon}_{[r,t]}(x,A)
\,=\,
\int_{\overline E_\varepsilon}
\mathsf K^{T,\beta,\varepsilon}_{[r,s]}(x,dy)\,
\mathsf K^{T,\beta,\varepsilon}_{[s,t]}(y,A)\, .
\end{align}
\end{enumerate}
\end{proposition}

\subsection{The associated Markov process}\label{TheAssociatedMarkovProcess}

Let $D([0,T];\overline E_\varepsilon)$ denote the space of all c\`adl\`ag paths with
values in $\overline E_\varepsilon$, that is,
\[
D([0,T];\overline E_\varepsilon)
\,:=\,
\big\{
p:[0,T]\to\overline E_\varepsilon \;:\;
p \text{ is right--continuous and has left limits on } [0,T]
\big\}.
\]
Recall that $\overline E_\varepsilon$ is equipped with the metric $\rho$ given by~\eqref{rhoMetric}. Let $\Lambda$ denote the collection of all strictly increasing continuous
bijections $\lambda:[0,T]\to[0,T]$. For $p,q\in D([0,T];\overline E_\varepsilon)$,
define
\[
d_{J_1}(p,q)
\,:=\,
\inf_{\lambda\in\Lambda}
\bigg\{
\sup_{t\in[0,T]} |\lambda(t)-t|
\;\vee\;
\sup_{t\in[0,T]} \rho\big(p(t),q(\lambda(t))\big)
\bigg\}.
\]
This is the usual Skorohod $J_1$ metric, written in an equivalent form;
see \cite[Sec.~12, eqs.~(12.11)--(12.13)]{Billingsley}. We equip
$D([0,T];\overline E_\varepsilon)$ with the Skorohod $J_1$ topology
induced by $d_{J_1}$. Since $\overline E_\varepsilon$ is Polish, the space
$D([0,T];\overline E_\varepsilon)$, equipped with the Skorohod $J_1$
topology, is Polish; see, for instance, \cite[Thm.~12.2]{Billingsley} for the case
$D=D([0,1];\mathbb{R})$, where $D$ is separable under $d_{J_1}$ and complete under
an equivalent metric $d^0$. We denote by
$\mathcal B\big(D([0,T];\overline E_\varepsilon)\big)$ the corresponding
Borel $\sigma$--algebra.

We now turn to the construction of a Markov process corresponding to the
transition probability kernels
$\{\mathsf K^{T,\beta,\varepsilon}_{[s,t]}\}_{0\le s<t\le T}$
given by~\eqref{EqDefMeshLimitFull}. Since the
Chapman--Kolmogorov property~\eqref{EqCKExtendedProbKernel}
holds on the dyadic set \(\mathbb D_T\), the construction proceeds in
three steps.
First, we construct a time-inhomogeneous Markov process $X$ on the dyadic
grid $\mathbb D_T$ whose finite-dimensional distributions are determined
by the kernel $\mathsf K^{T,\beta,\varepsilon}_{[s,t]}$.
Next, for each $m\ge1$, we extend this dyadic skeleton to a step-function
process $X^{(m)}$ by keeping the path constant on each interval
$[k2^{-m}T,(k+1)2^{-m}T)$, thereby obtaining a family of c\`adl\`ag processes.
Finally, we aim to pass to the limit as $m\to\infty$ in the Skorohod space
$D([0,T];\overline E_\varepsilon)$ in order to construct a c\`adl\`ag Markov
process with transition kernel
$\mathsf K^{T,\beta,\varepsilon}_{[s,t]}$\footnote{The implementation of
this step, including tightness and identification of the limiting law,
is deferred to future work.}. A schematic summary of this construction is given below; see also
Figure~\ref{FigDyadicSkeletonApproximation}.
\begin{center}
\begin{tikzpicture}[scale=1]

\tikzset{
box/.style={
draw,
rectangle,
minimum height=0.9cm,
align=center
}
}

\node[box, minimum width=2.4cm] (A) at (-0.5,0)
{
Dyadic\\[-1mm] skeleton
};

\draw[->, thick] (1.2,0) -- (2.6,0)
node[midway, above=2pt] {\scriptsize interpolation};

\node[box, minimum width=4.8cm] (B) at (5.4,0)
{
$X^{(m)} \in D([0,T];\overline E_{\varepsilon})$
};

\draw[->, thick] (8.0,0) -- (9.4,0)
node[midway, above=2pt] {\scriptsize $m\to\infty$};

\node[box, minimum width=2.6cm] (C) at (11.1,0)
{
c\`adl\`ag\\[-1mm] process
};
\end{tikzpicture}
\end{center}
\begin{figure}[ht]
\centering
\includegraphics[width=1\textwidth]{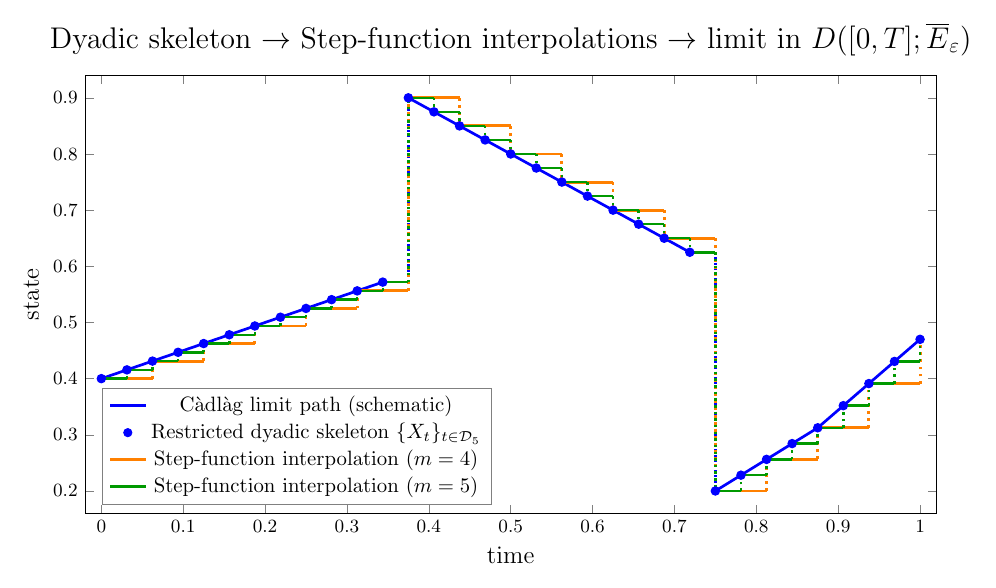}
\caption{Schematic illustration of the passage from the dyadic skeleton to the continuous-time c\`adl\`ag process through step-function interpolations. The figure is intended only to visualize the construction and does not represent an exact simulation of the process. The vertical dotted segments indicate jump locations. Recall that $|\mathcal D_4|=2^4+1$ and $|\mathcal D_5|=2^5+1$. For each \(m\ge1\), the step-function interpolation is constant on every interval \( [k2^{-m}T,(k+1)2^{-m}T) \), where \(k=0,1,\dots,2^m-1\).}
\label{FigDyadicSkeletonApproximation}
\end{figure}

We now describe the three steps in detail. \vspace{.2cm}

\noindent \textbf{Step 1 (Construction of dyadic skeleton).}
We consider the canonical path space
$\Omega^{\mathbb D_T}_\varepsilon$ equipped with the product
$\sigma$--algebra $\mathcal F^{\mathbb D_T}_\varepsilon$, given by
\[
\Omega^{\mathbb D_T}_\varepsilon
\,:=\,
\overline E_\varepsilon^{\mathbb D_T}
\,=\,
\{\omega:\mathbb D_T\to\overline E_\varepsilon\}\,,
\qquad
\mathcal F^{\mathbb D_T}_\varepsilon
\,:=\,
\mathcal B(\overline E_\varepsilon)^{\otimes\mathbb D_T}.
\]
For each $t\in\mathbb D_T$, let
$X_t:\Omega^{\mathbb D_T}_\varepsilon\to\overline E_\varepsilon$
denote the coordinate map defined by $X_t(\omega):=\omega(t)$.
Each $X_t$ is
$\mathcal F^{\mathbb D_T}_\varepsilon/\mathcal B(\overline E_\varepsilon)$--measurable.
We refer to the process $X=\{X_t\}_{t\in\mathbb D_T}$ as the
\textit{dyadic skeleton}.
Let $\{\mathcal F_t^{\mathbb D_T,\varepsilon}\}_{t\in\mathbb D_T}$
be the canonical filtration generated by $X$, that is,
\[
\mathcal F_t^{\mathbb D_T,\varepsilon}
\,:=\,
\sigma\big(X_r:\, r\in\mathbb D_T,\ 0\le r\le t\big),
\qquad t\in\mathbb D_T.
\]
The following proposition constructs a unique probability measure on $\big(\Omega^{\mathbb D_T}_\varepsilon,\mathcal F^{\mathbb D_T}_\varepsilon\big)$ under which the coordinate process $X=\{X_t\}_{t\in\mathbb D_T}$ is a Markov process with transition kernels $\{\mathsf K^{T,\beta,\varepsilon}_{[s,t]}\}$; its proof is given in Section~\ref{TheDyadicSkeleton}.

\begin{proposition} \label{PropDyadicSkeletonProcess}
Fix $T,\beta,\varepsilon>0$ and $x\in \overline E_\varepsilon$. There exists a unique probability measure $\mathbb P_{x}^{T,\beta,\varepsilon}$ on $\big(\Omega^{\mathbb D_T}_\varepsilon,\mathcal F^{\mathbb D_T}_\varepsilon\big)$ such that, under $\mathbb P_{x}^{T,\beta,\varepsilon}$, the coordinate process $\{X_t\}_{t\in\mathbb D_T}$ has initial distribution $\delta_x$ and is a time--inhomogeneous Markov process with transition kernels $\big\{\mathsf K^{T,\beta,\varepsilon}_{[s,t]}\big\}_{\substack{s,t\in\mathbb D_T\\ s<t}}$ with respect to the canonical filtration $\{\mathcal F_t^{\mathbb D_T,\varepsilon}\}_{t\in\mathbb D_T}$, in the sense that for every $s,t\in\mathbb D_T$ with $0\le s<t\le T$ and every bounded Borel function
$f:\overline E_\varepsilon\to\mathbb R$,
\begin{align}\label{EqMarkovPropertyDyadic}
\mathbb E_x^{T,\beta,\varepsilon}\!\left[
f(X_t)\,\big|\,\mathcal F_s^{\mathbb D_T,\varepsilon}
\right]
\,=\,
\int_{\overline E_\varepsilon}
f(y)\,
\mathsf K^{T,\beta,\varepsilon}_{[s,t]}(X_s,dy)
\qquad
\mathbb P_x^{T,\beta,\varepsilon}\textup{-a.s.}
\end{align}
\end{proposition}

The dyadic skeleton $X$ is a killed Markov process with absorbing
cemetery state $\Delta$, and at dyadic times,
$X_t=\Delta$ if and only if the process has exited
$E_\varepsilon$ almost surely; see Proposition~\ref{PropDyadicKillingTime}. \vspace{.3cm}

\noindent \textbf{Step 2 (Step-function interpolations of the dyadic skeleton).}
Our goal is to extend the dyadic skeleton
$\{X_t\}_{t\in\mathbb D_T}$ from Proposition~\ref{PropDyadicSkeletonProcess}
to c\`adl\`ag processes on the full time interval $[0,T]$.
For each $m\ge1$, we define a step-function interpolation
$X^{(m)}=\{X_t^{(m)}\}_{t\in[0,T]}$ by setting
\[
X_t^{(m)} := X_{t_k}
\quad \textup{for } t\in [t_k,t_{k+1}),
\quad \textup{where } t_k\in \mathcal D_m.
\]
The interpolated process $X^{(m)}$ has the following properties:
\begin{itemize}
\item[--] \emph{Path regularity and consistency with the dyadic skeleton.}
For every $\omega$, the path $t\mapsto X_t^{(m)}(\omega)$ is piecewise constant and c\`adl\`ag on $[0,T]$ with values in $\overline E_\varepsilon$, and $\omega\mapsto X^{(m)}_{\cdot}(\omega)$ is measurable as a map into
$D([0,T];\overline E_\varepsilon)$. Moreover, the finite-dimensional distributions of $X^{(m)}$ at times belonging to $\mathcal D_m$ agree with those of the skeleton $\{X_t\}_{t\in\mathbb D_T}$ under $\mathbb P_x^{T,\beta,\varepsilon}$; see Proposition~\ref{PropDyadicStepInterpolation}.

\item[--] \emph{Finite-dimensional distributions.} Under $\mathbb P_{x}^{T,\beta,\varepsilon}$, the finite-dimensional distributions of $\{X_t^{(m)}\}_{t\in[0,T]}$ are completely determined by those of the restricted dyadic skeleton $\{X_t\}_{t\in\mathcal D_m}$; see Lemma~\ref{LemKernelRepresentationInterpolatedTimes}.

\item[--] \emph{Interpolated path-space law.}
The process $X^{(m)}$ induces a probability measure
$\mathbb P_x^{T,\beta,\varepsilon,(m)}$ on the Skorokhod space
$D([0,T];\overline E_\varepsilon)$ under which the canonical process on $D([0,T];\overline E_\varepsilon)$ has the same finite-dimensional distributions as $X^{(m)}$; see Proposition~\ref{PropCanonicalProcessInterpolatedLaw}.

\item[--] \emph{Tightness.}
For each fixed $t\in[0,T]$, the family of one-time marginals
$\{\mathbb P_x^{T,\beta,\varepsilon,(m)}(X_t^{(m)}\in\cdot)\}_{m\ge1}$
is tight on $\overline E_\varepsilon$; see Proposition~\ref{PropTightnessOneTimeMarginals}.
\end{itemize}

\vspace{.2cm}

\noindent \textbf{Step 3 (Towards the construction of a limiting c\`adl\`ag process).}
The ultimate goal is to pass to the limit as $m\to\infty$ in the interpolated path-space laws
$\{\mathbb P_{x}^{T,\beta,\varepsilon,(m)}\}_{m\ge1}$, and thereby obtain a probability measure on $D([0,T];\overline E_\varepsilon)$ that extends the dyadic law $\mathbb P_{x}^{T,\beta,\varepsilon}$ from Proposition~\ref{PropDyadicSkeletonProcess} to the full time interval $[0,T]$. More precisely, our aim is to construct a limiting probability measure
\[
\mathbb P_{x}^{T,\beta,\varepsilon,\infty}
\,=\,
\lim_{m\to\infty}
\mathbb P_{x}^{T,\beta,\varepsilon,(m)},
\]
in a suitable sense, under which the associated canonical process on $D([0,T];\overline E_\varepsilon)$ has the same finite-dimensional distributions as those of the dyadic skeleton $\{X_t\}_{t\in\mathbb D_T}$. The construction and analysis of this limit, including tightness and convergence in $D([0,T];\overline E_\varepsilon)$, are deferred to future work.

\subsection{Organization of the paper}
The paper is organized as follows.
\begin{itemize}

\item[--] In Section~\ref{DetailedConstructionOfTheTransitionKernels}, we give the detailed construction of the kernels
$\widehat{\mathsf K}^{T,\beta,\varepsilon,\pi}_{[s,t]}$,
$\widehat{\mathsf K}^{T,\beta,\varepsilon}_{[s,t]}$, and
$\mathsf K^{T,\beta,\varepsilon}_{[s,t]}$ introduced in
Section~\ref{TransitionProbabilityKernel}. We also discuss their main
properties and the corresponding proofs.

\item[--] Section~\ref{DetailedConstructionOfTheMarkovprocess}
focuses on the construction and properties of the dyadic skeleton and
of the corresponding step-function interpolations introduced in
Section~\ref{TheAssociatedMarkovProcess}.

\item[--] Section~\ref{SpecialFunctions} collects basic analytic properties of the functions
$P_T^\beta(x,y)$, $Q_T^\beta(x)$, and
$p_{s,t}^{T,\beta}(x,y)$ used throughout the paper.

\item[--] In Appendix~\ref{AppendixSemigroupProperty}, we provide a direct proof of the semigroup property~\eqref{EqSemigroupKernelReliable} for the attractive point interaction heat kernel $P_t^{\beta}(x,y)$.

\item[--] Finally, Appendix~\ref{AppendixProofPsiAsymptotics} contains the proof of the asymptotic expansion~\eqref{PsiAsymptotics}.

\end{itemize}

\section{Detailed construction of the transition kernels} \label{DetailedConstructionOfTheTransitionKernels}

The goal of the present section is to develop in detail the
construction of the transition kernels introduced in
Section~\ref{TransitionProbabilityKernel} and to provide the
corresponding proofs. In Section~\ref{SecProductSubProbKernel}, we verify that the partition-level functions $\widehat{\mathsf K}^{T,\beta,\varepsilon,\pi}_{[s,t]}$ define sub-probability kernels and establish their monotonicity under refinement together with the associated partition-level concatenation property. In Section~\ref{SecLimitingSubProbKernel}, we study the limiting kernels $\widehat{\mathsf K}^{T,\beta,\varepsilon}_{[s,t]}$, prove their measurability and sub-probability kernel properties, and establish the corresponding Chapman--Kolmogorov identity on dyadic times. Finally, in Section~\ref{SecTransitionProbKernel}, we prove
Proposition~\ref{PropCKFullProbKernel} for the enlarged kernels
$\mathsf K^{T,\beta,\varepsilon}_{[s,t]}$ on $\overline E_\varepsilon$.

\subsection{The partition-level sub-probability kernel \texorpdfstring{$\boldsymbol{\widehat{\mathsf K}^{T,\beta,\varepsilon,\pi}_{[s,t]}}$}{Lg}}\label{SecProductSubProbKernel}

\begin{lemma}\label{LemMonotoneOnEeps}
Fix $T,\beta,\varepsilon>0$ and $0\le s<t\le T$. For each partition $\pi$ of $[s,t]$, the map $\widehat{\mathsf K}^{T,\beta,\varepsilon,\pi}_{[s,t]}$ defined in~\eqref{FirstTrans3dSub} is a sub-probability kernel on $E_\varepsilon$ and satisfies the following:
\begin{enumerate}[(i)]
\item The family $\{\widehat{\mathsf K}^{T,\beta,\varepsilon,\pi}_{[s,t]}\}_\pi$ is monotone under refinement pointwise in $(x,A)$: if $\pi'$ refines $\pi$, then for every $x\in E_\varepsilon$ and $A\in\mathcal B(E_\varepsilon)$,
\begin{align}\label{EqMonotoneOnEeps}
\widehat{\mathsf K}^{T,\beta,\varepsilon,\pi'}_{[s,t]}(x,A)
\;\le\;
\widehat{\mathsf K}^{T,\beta,\varepsilon,\pi}_{[s,t]}(x,A).
\end{align}

\item Fix $0\le r<s<t\le T$ and let $\pi$ be a partition of $[r,t]$ with $s\in\pi$. Then for every $x\in E_\varepsilon$ and $A\in\mathcal B(E_\varepsilon)$,
\begin{align}\label{EqCK_PartitionLevel}
\widehat{\mathsf K}^{T,\beta,\varepsilon,\pi}_{[r,t]}(x,A)
\,=\,
\int_{E_\varepsilon}
\widehat{\mathsf K}^{T,\beta,\varepsilon,\pi|_{[r,s]}}_{[r,s]}(x,dy)\,
\widehat{\mathsf K}^{T,\beta,\varepsilon,\pi|_{[s,t]}}_{[s,t]}(y,A).
\end{align}
Here $\pi|_{[r,s]}$ and $\pi|_{[s,t]}$ denote the restrictions of $\pi$ to $[r,s]$ and $[s,t]$, respectively.
\end{enumerate}
\end{lemma}

\begin{proof}
Since each $p^{\,T,\beta}_{u,v}(x,\cdot)$ is a nonnegative probability density, the integrand in~\eqref{FirstTrans3dSub} is nonnegative and measurable. Repeated applications of Tonelli's theorem therefore imply that the integral is well-defined for every $x\in E_\varepsilon$ and $A\in\mathcal B(E_\varepsilon)$, and that
\[
0 \le \widehat{\mathsf K}^{T,\beta,\varepsilon,\pi}_{[s,t]}(x,A) \le 1.
\]
Moreover, for each fixed $x\in E_\varepsilon$, the map
$A\mapsto \widehat{\mathsf K}^{T,\beta,\varepsilon,\pi}_{[s,t]}(x,A)$
is a sub-probability measure on $\mathcal B(E_\varepsilon)$, since Tonelli's theorem yields countable additivity. Furthermore, since the kernel $P_t^\beta(x,y)$ defined in~\eqref{DefPointKer3dBeta} is jointly continuous on
$(\R^3\setminus\{0\})\times(\R^3\setminus\{0\})$ by~\cite[Cor.~7]{Fleischmann}, and $Q_t^\beta(\cdot)$ is continuous on $\R^3\setminus\{0\}$, it follows that $p^{\,T,\beta}_{s,t}(x,y)$ is jointly continuous (hence Borel measurable) on $(\R^3\setminus\{0\})^2$. Consequently, the integrand in~\eqref{FirstTrans3dSub} is jointly Borel measurable in $(x,y_1,\dots,y_n)$, and another application of Tonelli's theorem yields that, for every $A\in\mathcal B(E_\varepsilon)$, the map
$x\mapsto \widehat{\mathsf K}^{T,\beta,\varepsilon,\pi}_{[s,t]}(x,A)$
is Borel measurable. Therefore, $\widehat{\mathsf K}^{T,\beta,\varepsilon,\pi}_{[s,t]}$ is a sub-probability kernel on $E_\varepsilon$. \vspace{.3cm}

\noindent Part (i).  \textit{Monotonicity under a single refinement.} Fix $x \in E_\varepsilon$ and $A\in\mathcal B(E_\varepsilon)$. Let $\pi=\{t_0=s<t_1<\cdots<t_n=t\}$ be a partition of $[s,t]$ and $\pi'=\{t_0,\dots,t_{k-1},u,t_k,\dots,t_n\}$ be a refinement of $\pi$, obtained by inserting a point $u$ into some interval $(t_{k-1},t_k)$. Using~\eqref{FirstTrans3dSub} we obtain the first and last equalities below. 
\begin{align}
\widehat{\mathsf K}^{T,\beta,\varepsilon,\pi'}_{[s,t]}(x,A)
\,=\,&
\int_{(E_\varepsilon)^{n+1}}
\mathbf 1_A(y_n)\,
\bigg(\prod_{i\neq k}p^{\,T,\beta}_{t_{i-1},t_i}(y_{i-1},y_i)\bigg)
p^{\,T,\beta}_{t_{k-1},u}(y_{k-1},z)\,
p^{\,T,\beta}_{u,t_k}(z,y_k)\,
dz\,dy_1\cdots dy_n \nonumber \\
\,\le\,&
\int_{(E_\varepsilon)^{n}}
\mathbf 1_A(y_n)\,
\bigg(\prod_{i\neq k}p^{\,T,\beta}_{t_{i-1},t_i}(y_{i-1},y_i)\bigg)
\bigg(\int_{\R^3 \setminus \{0\}}p^{\,T,\beta}_{t_{k-1},u}(y_{k-1},z)\,
p^{\,T,\beta}_{u,t_k}(z,y_k)\,dz\bigg)\,
dy_1\cdots dy_n \nonumber \\
\,=\,&
\int_{(E_\varepsilon)^{n}}
\mathbf 1_A(y_n)\,
\prod_{i=1}^{n}p^{\,T,\beta}_{t_{i-1},t_i}(y_{i-1},y_i)\,
dy_1\cdots dy_n 
\,=\,
\widehat{\mathsf K}^{T,\beta,\varepsilon,\pi}_{[s,t]}(x,A) \nonumber 
\end{align}
The second equality uses the Chapman-Kolmogorov identity for the (nonnegative) transition density $p^{\,T,\beta}$ on $\R^3\setminus\{0\}$ from Lemma~\ref{LemTranKern}(ii). \vspace{.1cm}

\noindent \textit{General refinements.}
If $\pi'$ refines $\pi$, then $\pi'$ can be obtained from $\pi$ by finitely many single insertions. Iterating the single refinement over these insertions yields~\eqref{EqMonotoneOnEeps} for general refinements. \vspace{.3cm}

\noindent Part (ii). Fix $x\in E_\varepsilon$ and $A\in\mathcal B(E_\varepsilon)$, and write $y_0:=x$.
Since $s\in\pi$, we may write
\[
\pi=\{r=t_0<t_1<\cdots<t_k=s<t_{k+1}<\cdots<t_n=t\}
\]
for some $1\le k\le n-1$. Then
\[
\pi|_{[r,s]}\,=\,\{r=t_0<t_1<\cdots<t_k=s\},
\qquad
\pi|_{[s,t]}\,=\,\{s=t_k<t_{k+1}<\cdots<t_n=t\}.
\]
By splitting the product in~\eqref{FirstTrans3dSub} at the time index corresponding to $s=t_k$ we obtain
\begin{align}\label{EqCK_ProofStart}
\widehat{\mathsf K}^{T,\beta,\varepsilon,\pi}_{[r,t]}(x,A)
&\,=\,
\int_{(E_\varepsilon)^n}
\mathbf 1_A(y_n)\,
\bigg(\prod_{i=1}^k
p^{\,T,\beta}_{t_{i-1},t_i}(y_{i-1},y_i)\bigg)
\bigg(\prod_{i=k+1}^n
p^{\,T,\beta}_{t_{i-1},t_i}(y_{i-1},y_i)\bigg)
dy_1\cdots dy_n .
\end{align}
Since the integrand is nonnegative, Tonelli's theorem allows us to integrate first
with respect to $(y_{k+1},\dots,y_n)$ while keeping $(y_1,\dots,y_k)$ fixed. Hence, the right side equals
\begin{align}
\int_{(E_\varepsilon)^k}
\bigg(\prod_{i=1}^k
p^{\,T,\beta}_{t_{i-1},t_i}(y_{i-1},y_i)\bigg)
\bigg[
\int_{(E_\varepsilon)^{n-k}}
\mathbf 1_A(y_n)\,
\prod_{i=k+1}^n
p^{\,T,\beta}_{t_{i-1},t_i}(y_{i-1},y_i)\,
dy_{k+1}\cdots dy_n
\bigg]
dy_1\cdots dy_k \,. \nonumber 
\end{align}
For fixed $y_k\in E_\varepsilon$, the bracketed inner integral is exactly $\widehat{\mathsf K}^{T,\beta,\varepsilon,\pi|_{[s,t]}}_{[s,t]}(y_k,A)$ by definition~\eqref{FirstTrans3dSub} with the restricted partition
$\pi|_{[s,t]}$. Therefore~\eqref{EqCK_ProofStart} becomes
\begin{align}\label{EqCKPartitionAfterInner}
\widehat{\mathsf K}^{T,\beta,\varepsilon,\pi}_{[r,t]}(x,A)
&\,=\,
\int_{(E_\varepsilon)^k}
\bigg(\prod_{i=1}^k
p^{\,T,\beta}_{t_{i-1},t_i}(y_{i-1},y_i)\bigg)
\widehat{\mathsf K}^{T,\beta,\varepsilon,\pi|_{[s,t]}}_{[s,t]}(y_k,A)\,
dy_1\cdots dy_k \,.
\end{align}
Now, by the iterated-integral representation of
$\widehat{\mathsf K}^{T,\beta,\varepsilon,\pi|_{[r,s]}}_{[r,s]}(x,\cdot)$, we have for every
nonnegative Borel function $f$ on $E_\varepsilon$,
\[
\int_{E_\varepsilon}
f(y)\,\widehat{\mathsf K}^{T,\beta,\varepsilon,\pi|_{[r,s]}}_{[r,s]}(x,dy)
\,=\,
\int_{(E_\varepsilon)^k}
f(y_k)\,
\prod_{i=1}^k
p^{\,T,\beta}_{t_{i-1},t_i}(y_{i-1},y_i)\,
dy_1\cdots dy_k .
\]
Applying this with $f(y):=\widehat{\mathsf K}^{T,\beta,\varepsilon,\pi|_{[s,t]}}_{[s,t]}(y,A)$, which is a nonnegative Borel function on $E_\varepsilon$, the right-hand side of
\eqref{EqCKPartitionAfterInner} becomes
\[
\int_{E_\varepsilon}
\widehat{\mathsf K}^{T,\beta,\varepsilon,\pi|_{[r,s]}}_{[r,s]}(x,dy)\,
\widehat{\mathsf K}^{T,\beta,\varepsilon,\pi|_{[s,t]}}_{[s,t]}(y,A).
\]
This proves~\eqref{EqCK_PartitionLevel}.
\end{proof}

For a fixed partition $\pi$ of $[s,t]$, the product kernel 
$\widehat{\mathsf K}^{T,\beta,\varepsilon,\pi}_{[s,t]}$ defined in~\eqref{FirstTrans3dSub} 
evolves the system through successive one--step transitions and restricts mass to trajectories that remain in $E_\varepsilon$ at the times of~$\pi$. This suggests refining the partition and considering the regime $|\pi|\downarrow 0$, so that these checkpoints become dense in $[s,t]$ and the resulting limit enforces survival throughout the entire interval. The following proposition shows that such a limiting object can be constructed along a suitable refining sequence of partitions. We write 
$\widehat{\mu}^{T,\beta,\varepsilon}_{[s,t]}$ for this limit to emphasize that it encodes survival over the whole interval $[s,t]$.

\begin{proposition} \label{PropExistLimitSubprobKernel}
Fix $T,\beta,\varepsilon>0$ and $0\le s<t\le T$. Let $\Pi_{[s,t]}$ denote the collection of all finite partitions of $[s,t]$, partially ordered by
refinement (so $\pi'\succeq\pi$ means $\pi'$ refines $\pi$). The map $\widehat{\mu}^{T,\beta,\varepsilon}_{[s,t]}:\;
E_\varepsilon \times \mathcal B(E_\varepsilon)\longrightarrow [0,1]$
defined, for $x\in E_\varepsilon$ and $A\in\mathcal B(E_\varepsilon)$, by
\begin{align}\label{EqDefLimitKernelInf}
\widehat{\mu}^{T,\beta,\varepsilon}_{[s,t]}(x,A)
\,:=\,
\inf_{\pi\in\Pi_{[s,t]}}
\widehat{\mathsf K}^{T,\beta,\varepsilon,\pi}_{[s,t]}(x,A)\,
\end{align}
is well-defined and satisfies the following.
\begin{enumerate}[(i)]
\item Fix $x\in E_\varepsilon$ and $A\in\mathcal B(E_\varepsilon)$.
There exists a refining sequence $(\sigma_m)_{m\ge1}\subset\Pi_{[s,t]}$ with
$|\sigma_m|\downarrow0$ as $m\to\infty$ such that
\begin{align}\label{EqLimitAlongApproximatingSequence}
\widehat{\mu}^{T,\beta,\varepsilon}_{[s,t]}(x,A)
\,=\,
\lim_{m\to\infty}
\widehat{\mathsf K}^{T,\beta,\varepsilon,\sigma_m}_{[s,t]}(x,A).
\end{align}

\item For every $x\in E_\varepsilon$, the set function
$A\mapsto \widehat{\mu}^{T,\beta,\varepsilon}_{[s,t]}(x,A)$ is a sub-probability measure on
$(E_\varepsilon,\mathcal B(E_\varepsilon))$.
\end{enumerate}
\end{proposition}

Before proceeding to the proof, we describe the refining sequence of partitions
used in the diagonal argument. At each stage $m$, we select one partition
$\pi^{(m)}$ associated with the set $A=\cup_{j\ge1}A_j$, and partitions
$\pi^{(m)}_j$ associated with the individual sets $A_j$ for $1\le j\le m$.
These choices may be organized in a triangular scheme:
\[
\begin{array}{c|ccccc}
 & A & A_1 & A_2 & A_3 & \cdots \\
\hline
m=1 & \pi^{(1)} & \pi^{(1)}_1 &  &  & \cdots \\
m=2 & \pi^{(2)} & \pi^{(2)}_1 & \pi^{(2)}_2 &  & \cdots \\
m=3 & \pi^{(3)} & \pi^{(3)}_1 & \pi^{(3)}_2 & \pi^{(3)}_3 & \cdots \\
m=4 & \pi^{(4)} & \pi^{(4)}_1 & \pi^{(4)}_2 & \pi^{(4)}_3 & \cdots \\
\vdots & \vdots & \vdots & \vdots & \vdots & \ddots
\end{array}
\]
We then construct a sequence $(\sigma_m)_{m\ge1}$ inductively by letting
$\sigma_m$ be a common refinement of $\sigma_{m-1}$, $\pi^{(m)}$, and
$\pi^{(m)}_1,\dots,\pi^{(m)}_m$. In particular, the sequence
$(\sigma_m)_{m \ge 1}$ is increasing under refinement.

\begin{proof}[Proof of Lemma~\ref{PropExistLimitSubprobKernel}]
By Lemma~\eqref{LemMonotoneOnEeps}, for each fixed $x\in E_\varepsilon$ and
$A\in\mathcal B(E_\varepsilon)$ the family
$\{\widehat{\mathsf K}^{T,\beta,\varepsilon,\pi}_{[s,t]}(x,A)\}_{\pi\in\Pi_{[s,t]}}$ is bounded in $[0,1]$ and is decreasing under refinement. Hence the infimum in~\eqref{EqDefLimitKernelInf} exists in $[0,1]$. \vspace{.2cm}

\noindent Part (i). Fix $x\in E_\varepsilon$ and $A\in\mathcal B(E_\varepsilon)$. For each $m\ge1$, by the definition of infimum~\eqref{EqDefLimitKernelInf}, there exists a partition $\pi^{(m)}\in\Pi_{[s,t]}$ such that
\begin{align}\label{EqAlmostMinimizer}
\widehat{\mathsf K}^{T,\beta,\varepsilon,\pi^{(m)}}_{[s,t]}(x,A)
\,\le\,
\widehat{\mu}^{T,\beta,\varepsilon}_{[s,t]}(x,A) \, +\, \frac1m\,.
\end{align}
Define a sequence $(\sigma_m)_{m\ge 1}$ inductively by $\sigma_1:=\pi^{(1)}$ and $\sigma_m:=\sigma_{m-1}\vee \pi^{(m)}$ for $m\ge2$, so that $\sigma_m\succeq \sigma_{m-1}$ for all $m$ (hence $(\sigma_m)_{m \ge 1}$ is refining) and $\sigma_m\succeq \pi^{(m)}$. For each $m\ge1$, using the definition of infimum~\eqref{EqDefLimitKernelInf}, we obtain the first inequality below.
\begin{align}
\widehat{\mu}^{T,\beta,\varepsilon}_{[s,t]}(x,A)
\le
\widehat{\mathsf K}^{T,\beta,\varepsilon,\sigma_m}_{[s,t]}(x,A)
\le
\widehat{\mathsf K}^{T,\beta,\varepsilon,\pi^{(m)}}_{[s,t]}(x,A)
\le
\widehat{\mu}^{T,\beta,\varepsilon}_{[s,t]}(x,A) \, +\, \frac1m \nonumber
\end{align}
The second inequality uses Lemma~\ref{LemMonotoneOnEeps}(i) and the final inequality follows from~\eqref{EqAlmostMinimizer}. Since $(\sigma_m)_{m \ge 1}$ is refining, Lemma~\ref{LemMonotoneOnEeps}(i) implies that
$m\mapsto \widehat{\mathsf K}^{T,\beta,\varepsilon,\sigma_m}_{[s,t]}(x,A)$
is decreasing and bounded below by $0$, hence its limit $\lim_{m\to\infty}\widehat{\mathsf K}^{T,\beta,\varepsilon,\sigma_m}_{[s,t]}(x,A)$ exists. Passing to the limit $m\to\infty$ in the above display yields~\eqref{EqLimitAlongApproximatingSequence}. Finally, notice that each $\sigma_m$ may be refined further by inserting midpoints finitely many times to ensure $|\sigma_m|\downarrow0$ as $m\to\infty$, without changing the limit, since further refinement only decreases the values and preserves the bounds above. \vspace{.2cm}

\noindent Part (ii). Fix $x\in E_\varepsilon$ and define the map $\mu^{T,\beta,\varepsilon,x}_{[s,t]}:\; \mathcal B(E_\varepsilon)\longrightarrow [0,1]$ by
\begin{align}\label{DefmuinProof}
\mu^{T,\beta,\varepsilon,x}_{[s,t]}(A)
\,:=\,
\widehat{\mu}^{T,\beta,\varepsilon}_{[s,t]}(x,A)
\,=\,
\inf_{\pi\in\Pi_{[s,t]}}
\widehat{\mathsf K}^{T,\beta,\varepsilon,\pi}_{[s,t]}(x,A),
\qquad A\in\mathcal B(E_\varepsilon)\,,
\end{align}
where the last equality holds by~\eqref{EqDefLimitKernelInf}. Clearly $\mu^{T,\beta,\varepsilon,x}_{[s,t]}(\varnothing)=0$ and, since $\widehat{\mathsf K}^{T,\beta,\varepsilon,\pi}_{[s,t]}(x,E_\varepsilon)\le 1$
for every $\pi$, we have
\[
0\le 
\mu^{T,\beta,\varepsilon,x}_{[s,t]}(E_\varepsilon)
\,=\,
\inf_{\pi\in\Pi_{[s,t]}}
\widehat{\mathsf K}^{T,\beta,\varepsilon,\pi}_{[s,t]}(x,E_\varepsilon)\le 1\, .
\]
It remains to prove countable additivity. Let $(A_j)_{j\ge1}\subset\mathcal B(E_\varepsilon)$ be pairwise disjoint and set $A:=\bigcup_{j\ge1}A_j$. For each $m\ge1$, by the definition of infimum in~\eqref{DefmuinProof}, we may choose partitions $\pi^{(m)}\in\Pi_{[s,t]}$ and $\pi^{(m)}_j\in\Pi_{[s,t]}$ for $1\le j\le m$ such that
\begin{align}\label{EqAlmostMinimizersForAdditivity}
\widehat{\mathsf K}^{T,\beta,\varepsilon,\pi^{(m)}}_{[s,t]}(x,A)
\,\le \,
\mu^{T,\beta,\varepsilon,x}_{[s,t]}(A) + \frac1m,
\quad \textup{ and } \quad
\widehat{\mathsf K}^{T,\beta,\varepsilon,\pi^{(m)}_j}_{[s,t]}(x,A_j)
\,\le \,
\mu^{T,\beta,\varepsilon,x}_{[s,t]}(A_j) + \frac1m\,,
\quad 1\le j\le m\,.
\end{align}
Define $\sigma_m\in\Pi_{[s,t]}$ inductively by
\begin{align}
\sigma_1
\,:=\,\pi^{(1)}\vee \pi^{(1)}_1\,,
\qquad
\sigma_m
\,:=\,\sigma_{m-1}\vee \pi^{(m)} \vee \pi^{(m)}_1 \vee \cdots \vee \pi^{(m)}_m,
\qquad m\ge2. \nonumber
\end{align}
Then $(\sigma_m)_{m\ge1}$ is refining, and moreover $\sigma_m\succeq \pi^{(m)}$ and
$\sigma_m\succeq \pi^{(m)}_j$ for all $1\le j\le m$. Since $\widehat{\mathsf K}^{T,\beta,\varepsilon,\cdot}_{[s,t]}(x,A)$ is decreasing under refinement (Lemma~\ref{LemMonotoneOnEeps}(i)), we obtain for every $m\ge1$, the second inequality below.
\begin{align}\label{EqDiagonalSandwichUnion}
\mu^{T,\beta,\varepsilon,x}_{[s,t]}(A)
\le
\widehat{\mathsf K}^{T,\beta,\varepsilon,\sigma_m}_{[s,t]}(x,A)
\le
\widehat{\mathsf K}^{T,\beta,\varepsilon,\pi^{(m)}}_{[s,t]}(x,A)
\overset{\eqref{EqAlmostMinimizersForAdditivity}}{\le}
\mu^{T,\beta,\varepsilon,x}_{[s,t]}(A)+\frac1m
\end{align}
The first inequality follows from the definition~\eqref{DefmuinProof} of
$\mu^{T,\beta,\varepsilon,x}_{[s,t]}(A)$ as the infimum of
$\widehat{\mathsf K}^{T,\beta,\varepsilon,\pi}_{[s,t]}(x,A)$ over
$\pi\in\Pi_{[s,t]}$, and the fact that $\sigma_m\in\Pi_{[s,t]}$. Similarly, for each $1\le j\le m$, we have
\begin{align}\label{EqDiagonalSandwichPieces}
\mu^{T,\beta,\varepsilon,x}_{[s,t]}(A_j)
\le
\widehat{\mathsf K}^{T,\beta,\varepsilon,\sigma_m}_{[s,t]}(x,A_j)
\le
\widehat{\mathsf K}^{T,\beta,\varepsilon,\pi^{(m)}_j}_{[s,t]}(x,A_j)
\le
\mu^{T,\beta,\varepsilon,x}_{[s,t]}(A_j)+\frac1m.
\end{align}
Taking $m\to\infty$ in~\eqref{EqDiagonalSandwichUnion} yields the first equality below.
\begin{align}\label{EqLimitUnionAlongSigma}
\mu^{T,\beta,\varepsilon,x}_{[s,t]}(A)
\,=\,
\lim_{m\to\infty}\, 
\widehat{\mathsf K}^{T,\beta,\varepsilon,\sigma_m}_{[s,t]}(x,A)
\,=\,
\lim_{m\to\infty}\, 
\sum_{j\ge1} \widehat{\mathsf K}^{T,\beta,\varepsilon,\sigma_m}_{[s,t]}(x,A_j)
\end{align}
The second equality uses Lemma~\ref{LemMonotoneOnEeps}, which implies that the map $\widehat{\mathsf K}^{T,\beta,\varepsilon,\sigma_m}_{[s,t]}(x,\cdot)$
is a sub-probability measure on $(E_\varepsilon,\mathcal B(E_\varepsilon))$ for each $m\ge1$. Moreover, since $(\sigma_m)$ is refining, we have $\sigma_m\succeq \sigma_1$ for all $m$, hence by monotonicity in Lemma~\ref{LemMonotoneOnEeps}(i), we have
\[
0\le
\widehat{\mathsf K}^{T,\beta,\varepsilon,\sigma_m}_{[s,t]}(x,A_j)
\le
\widehat{\mathsf K}^{T,\beta,\varepsilon,\sigma_1}_{[s,t]}(x,A_j),
\qquad j\ge1,\; m\ge1,
\]
and the dominating series is summable because $\sum_{j\ge1}\widehat{\mathsf K}^{T,\beta,\varepsilon,\sigma_1}_{[s,t]}(x,A_j)
\,=\,
\widehat{\mathsf K}^{T,\beta,\varepsilon,\sigma_1}_{[s,t]}(x,A)
\le 1$. Therefore, by dominated convergence for series we have the first equality below.
\[
\lim_{m\to\infty}
\sum_{j\ge1}
\widehat{\mathsf K}^{T,\beta,\varepsilon,\sigma_m}_{[s,t]}(x,A_j)
\,=\,
\sum_{j\ge1}
\lim_{m\to\infty}
\widehat{\mathsf K}^{T,\beta,\varepsilon,\sigma_m}_{[s,t]}(x,A_j)
\,=\,
\sum_{j\ge1}
\mu^{T,\beta,\varepsilon,x}_{[s,t]}(A_j)
\]
The last equality holds by letting $m\to\infty$ in~\eqref{EqDiagonalSandwichPieces} along $m\ge j$ for each fixed $j\ge1$. Substituting the above in~\eqref{EqLimitUnionAlongSigma} yields countably additivity for $\mu^{T, \beta, \varepsilon, x}_{ [s, t]}$.
\end{proof}

\subsection{The limiting sub-probability kernel \texorpdfstring{$\boldsymbol{\widehat{\mathsf K}^{T,\beta,\varepsilon}_{[s,t]}}$}{Lg}}\label{SecLimitingSubProbKernel}

Proposition~\ref{PropExistLimitSubprobKernel} establishes the existence of a limiting function 
$\widehat{\mu}^{T,\beta,\varepsilon}_{[s,t]}$; however, it does not yield the Borel measurability of the map 
$x \mapsto \widehat{\mu}^{T,\beta,\varepsilon}_{[s,t]}(x,A)$ for fixed $A\in\mathcal B(E_\varepsilon)$. 
In this section, we use the canonical sequence of global dyadic partitions~\eqref{DefGlobalDyadics} to construct a limiting sub-probability kernel that assigns mass only to trajectories that remain in $E_\varepsilon$ throughout the entire interval $[s,t]$. Moreover, it satisfies the Chapman-Kolmogorov property on dyadic times. We begin with the nesting and concatenation properties of the global dyadic partitions used in the main text, along with a few illustrative examples. \vspace{.2cm}

\noindent\textit{Nesting.}
For each $k\in\{0,\dots,2^m\}$, we have $k2^{-m}T=(2k)2^{-(m+1)}T\in\mathcal D_{m+1}$, so $\mathcal D_m\subseteq \mathcal D_{m+1}$. Hence, for every $0\le s<t\le T$, we have 
\[
\pi_m[s,t]
\,:=\,\{s,t\}\cup(\mathcal D_m\cap(s,t))
\subseteq
\{s,t\}\cup(\mathcal D_{m+1}\cap(s,t))
\,=:\,\pi_{m+1}[s,t].
\]

\vspace{.1cm}

\noindent\textit{Concatenation.} For the concatenation property, notice that if $0\le r<s<t\le T$ and $r,s,t\in\mathcal D_m$, then
\begin{align*}
\pi_m[r,t]
&\,:=\,\{r,t\}\cup\big(\mathcal D_m\cap(r,t)\big)\\
&\,=\,\{r,t\}\cup\Big(\mathcal D_m\cap\big((r,s)\cup\{s\}\cup(s,t)\big)\Big)\\
&\,=\,\{r,t\}\cup\Big(\big(\mathcal D_m\cap(r,s)\big)\cup\big(\mathcal D_m\cap\{s\}\big)\cup\big(\mathcal D_m\cap(s,t)\big)\Big)\\
&\,=\,\{r,t\}\cup\Big(\big(\mathcal D_m\cap(r,s)\big)\cup\{s\}\cup\big(\mathcal D_m\cap(s,t)\big)\Big)
\qquad(\text{since } s\in\mathcal D_m)\\[2mm]
&\,=\,\{r,s,t\}\cup(\mathcal D_m\cap(r,s))\cup(\mathcal D_m\cap(s,t))
\qquad(\text{rearranging unions})\\[2mm]
&\,=\,\big(\{r,s\}\cup(\mathcal D_m\cap(r,s))\big)
\ \cup\
\big(\{s,t\}\cup(\mathcal D_m\cap(s,t))\big)
\qquad(\textup{since } r<s<t \textup{ and } r,s,t\in\mathcal D_m )\\
&\,=\,\pi_m[r,s]\ \cup\ \pi_m[s,t]
\,=\,\pi_m[r,s]\ \vee\ \pi_m[s,t].
\end{align*}

\vspace{.1cm}

\noindent\textit{Examples.} If $0\le r<s<t\le T$ and $r,s,t\in\mathcal D_m$, then first few global grids are
\[
\mathcal D_1\,=\,\Big\{0,\ \tfrac{T}{2},\ T\Big\},
\qquad
\mathcal D_2\,=\,\Big\{0,\ \tfrac{T}{4},\ \tfrac{T}{2},\ \tfrac{3T}{4},\ T\Big\},
\qquad
\mathcal D_3\,=\,\Big\{0,\ \tfrac{T}{8},\ \tfrac{2T}{8},\ \tfrac{3T}{8},\ \tfrac{4T}{8},\ \tfrac{5T}{8},\ \tfrac{6T}{8},\ \tfrac{7T}{8},\ T\Big\}.
\]
Accordingly,
\begin{align}\label{IndPartPi1}
\pi_1[s,t]
&\,=\,\{s,t\}\cup\Big(\{0, \tfrac{T}{2},T\}\cap(s,t)\Big)
\,=\,\{s,t\}\cup\Big(\{ \tfrac{T}{2}\}\cap(s,t)\Big)
\,=\,
\begin{cases}
\{s,\tfrac{T}{2},t\}, & \text{if } s<\tfrac{T}{2}<t,\\[2mm]
\{s,t\}, & \text{otherwise.}
\end{cases} \\
\pi_2[s,t]
&\,=\,\{s,t\}\cup\Big(\{\tfrac{T}{4},\tfrac{T}{2},\tfrac{3T}{4}\}\cap(s,t)\Big)
\,=\,
\{s,t\}\cup
\Big(\mathbf 1_{\{s<\frac{T}{4}<t\}}\big\{\tfrac{T}{4}\big\}
\ \cup\
\mathbf 1_{\{s<\frac{T}{2}<t\}}\big\{\tfrac{T}{2}\big\}
\ \cup\
\mathbf 1_{\{s<\frac{3T}{4}<t\}}\big\{\tfrac{3T}{4}\big\}\Big) \,, \nonumber \\
\pi_3[s,t]
&\,=\,\{s,t\}\cup\Big(\{\tfrac{T}{8},\tfrac{2T}{8},\tfrac{3T}{8},\tfrac{4T}{8},\tfrac{5T}{8},\tfrac{6T}{8},\tfrac{7T}{8}\}\cap(s,t)\Big)
\,=\,
\{s,t\}\cup\bigcup_{k=1}^{7}\mathbf 1_{\{s<\frac{kT}{8}<t\}}\Big\{\tfrac{kT}{8}\Big\}. \nonumber
\end{align}

\begin{proposition} \label{PropExistLimitSubprobKernelGlobalDyadic}
Fix $T,\beta,\varepsilon>0$ and $0\le s<t\le T$. The map 
$\widehat{\mathsf K}^{T,\beta,\varepsilon}_{[s,t]}$ given by~\eqref{EqDefGlobalDyadicLimitKernel_Inf} is well defined and satisfies the following:
\begin{enumerate}[(i)]
\item $\widehat{\mathsf K}^{T,\beta,\varepsilon}_{[s,t]}:\;
E_\varepsilon \times \mathcal B(E_\varepsilon)\to [0,1]$ is a sub-probability kernel on $E_\varepsilon$.

\item If $0\le r<s<t\le T$ with $r,s,t\in\mathbb D_T$, then for every $x\in E_\varepsilon$ and every $A\in\mathcal B(E_\varepsilon)$,
\begin{align}\label{EqCKGlobalDyadicLimitKernel}
\widehat{\mathsf K}^{T,\beta,\varepsilon}_{[r,t]}(x,A)
\,=\,
\int_{E_\varepsilon}
\widehat{\mathsf K}^{T,\beta,\varepsilon}_{[r,s]}(x,dy)\,
\widehat{\mathsf K}^{T,\beta,\varepsilon}_{[s,t]}(y,A)\, .
\end{align}
\end{enumerate}
\end{proposition}

\begin{remark}[Local dyadic refinement case]
For $0\le s<t\le T$, we may consider the local dyadic partitions
\begin{align}
\pi_m(s,t)\,=\,\{s+k2^{-m}(t-s):\ k=0,1,\dots,2^m\}, \qquad m\ge1. \nonumber
\end{align}
The family $\{\pi_m(s,t)\}_{m\ge1}$ is nested with $\pi_{m+1}(s,t)\succeq\pi_m(s,t)$ and $|\pi_m(s,t)|=(t-s)2^{-m}\downarrow0$. By the same argument as in Proposition~\ref{PropExistLimitSubprobKernel}, the limiting object defined by
\begin{align}
\widehat{\mathsf K}^{T,\beta,\varepsilon}_{(s,t)}(x,A)
\,:=\,
\inf_{m\ge1}\,
\widehat{\mathsf K}^{T,\beta,\varepsilon,\pi_m(s,t)}_{[s,t]}(x,A)
\,=\,
\lim_{m\to\infty}
\widehat{\mathsf K}^{T,\beta,\varepsilon,\pi_m(s,t)}_{[s,t]}(x,A), \nonumber
\end{align}
for $x\in E_\varepsilon$ and $A\in\mathcal B(E_\varepsilon)$, defines a sub-probability kernel on $E_\varepsilon$. However, this local refinement is not stable under concatenation of time intervals: if $r<s<t$, the partitions $\pi_m(r,s)$ and $\pi_m(s,t)$ do not combine to form $\pi_m(r,t)$. Consequently, the associated limiting kernel does not satisfy a Chapman--Kolmogorov identity. For this reason, we work with the global dyadic refinement throughout.
\end{remark}

\begin{proof}[Proof of Proposition~\ref{PropExistLimitSubprobKernelGlobalDyadic}]
Fix $x\in E_\varepsilon$ and $A\in\mathcal B(E_\varepsilon)$.  By construction,
$\pi_{m+1}[s,t]\succeq \pi_m[s,t]$ for all $m\ge1$ and $|\pi_m[s,t]|\downarrow0$ as $m\to\infty$.
Therefore, by Lemma~\ref{LemMonotoneOnEeps}(i) the sequence
$\big(\widehat{\mathsf K}^{T,\beta,\varepsilon,\pi_m[s,t]}_{[s,t]}(x,A)\big)_{m\ge1}$ is non-increasing and takes values in $[0,1]$. Hence the infimum and limit in~\eqref{EqDefGlobalDyadicLimitKernel_Inf} exist in $[0,1]$ and are equal. Consequently, the map
$\widehat{\mathsf K}^{T,\beta,\varepsilon}_{[s,t]}$ is well-defined. \vspace{.3cm}

\noindent Part (i). \textit{Sub-probability measure.} Fix $x\in E_\varepsilon$ and define the set function
$\mu^{T,\beta,\varepsilon,x}_{[s,t]}:\mathcal B(E_\varepsilon)\to[0,1]$ by
\begin{align}\label{DefmuProof2}
\mu^{T,\beta,\varepsilon,x}_{[s,t]}(A)
\,:=\,
\widehat{\mathsf K}^{T,\beta,\varepsilon}_{[s,t]}(x,A),
\qquad A\in\mathcal B(E_\varepsilon).
\end{align}
For each $m\ge1$, the map
$A\mapsto \widehat{\mathsf K}^{T,\beta,\varepsilon,\pi_m[s,t]}_{[s,t]}(x,A)$
is a sub-probability measure by Lemma~\ref{LemMonotoneOnEeps}. In particular,
\[
\widehat{\mathsf K}^{T,\beta,\varepsilon,\pi_m[s,t]}_{[s,t]}(x,\varnothing)\,=\,0,
\qquad
0\le
\widehat{\mathsf K}^{T,\beta,\varepsilon,\pi_m[s,t]}_{[s,t]}(x,A)
\le
\widehat{\mathsf K}^{T,\beta,\varepsilon,\pi_m[s,t]}_{[s,t]}(x,E_\varepsilon)
\le 1,
\]
for every $A\in\mathcal B(E_\varepsilon)$. Letting $m\to\infty$ and using~\eqref{EqDefGlobalDyadicLimitKernel_Inf} yields
\[
\mu^{T,\beta,\varepsilon,x}_{[s,t]}(\varnothing)\,=\,0,
\qquad
0\le \mu^{T,\beta,\varepsilon,x}_{[s,t]}(A)\le \mu^{T,\beta,\varepsilon,x}_{[s,t]}(E_\varepsilon)\le 1.
\]
It remains to prove countable additivity. Let $(A_j)_{j\ge1}\subset\mathcal B(E_\varepsilon)$ be
pairwise disjoint and set $A:=\bigcup_{j\ge1}A_j$. Fix $N\in\mathbb N$. Then, using~\eqref{EqDefGlobalDyadicLimitKernel_Inf} and~\eqref{DefmuProof2} we have the first equality below.
\begin{align}
\mu^{T,\beta,\varepsilon,x}_{[s,t]}(A)
\,=\,&
\lim_{m\to\infty}
\widehat{\mathsf K}^{T,\beta,\varepsilon,\pi_m[s,t]}_{[s,t]}(x,A) \nonumber \\
\,=\,&
\lim_{m\to\infty}\,
\sum_{j \ge 1}\,\widehat{\mathsf K}^{T,\beta,\varepsilon,\pi_m[s,t]}_{[s,t]}(x,A_j) \nonumber \\
\,\ge\,&
\lim_{m\to\infty}\,
\sum_{j = 1}^N\,\widehat{\mathsf K}^{T,\beta,\varepsilon,\pi_m[s,t]}_{[s,t]}(x,A_j) 
\,=\,
\sum_{j = 1}^N\,
\lim_{m\to\infty}\, \widehat{\mathsf K}^{T,\beta, \varepsilon, \pi_m[s,t]}_{[s,t]}(x,A_j) 
\,=\,
\sum_{j = 1}^N\,
\mu^{T,\beta,\varepsilon,x}_{[s,t]}(A_j) \nonumber
\end{align}
The second equality uses that the map $\widehat{\mathsf K}^{T,\beta,\varepsilon,\pi_m[s,t]}_{[s,t]}(x,\cdot)$ is a sub-probability measure for each $m\ge1$. Letting $N\to\infty$ and using that the partial sums are non-decreasing and bounded above by
$\mu^{T,\beta,\varepsilon,x}_{[s,t]}(A)\le 1$, we obtain
\begin{align}\label{EqGlobal_mu_ge_sum}
\mu^{T,\beta,\varepsilon,x}_{[s,t]}(A)
\ge
\sum_{j\ge1}\mu^{T,\beta,\varepsilon,x}_{[s,t]}(A_j).
\end{align}

For the reverse inequality, fix $N\ge1$ and set $B_N:=\bigcup_{j\ge N+1}A_j\in\mathcal B(E_\varepsilon)$. 
For each $m\ge1$, since $\widehat{\mathsf K}^{T,\beta,\varepsilon,\pi_m[s,t]}_{[s,t]}(x,\cdot)$ is a finite measure, we have the first equality below.
\begin{align}
    \widehat{\mathsf K}^{T,\beta,\varepsilon,\pi_m[s,t]}_{[s,t]}(x,A)
\,=\,&
\sum_{j=1}^N\,\widehat{\mathsf K}^{T,\beta,\varepsilon,\pi_m[s,t]}_{[s,t]}(x,A_j)
\,+\,
\widehat{\mathsf K}^{T,\beta,\varepsilon,\pi_m[s,t]}_{[s,t]}(x,B_N) \nonumber \\
\,\le\,&
\sum_{j=1}^N\,\widehat{\mathsf K}^{T,\beta,\varepsilon,\pi_m[s,t]}_{[s,t]}(x,A_j)
\,+\,
\widehat{\mathsf K}^{T,\beta,\varepsilon,\pi_1[s,t]}_{[s,t]}(x,B_N) \nonumber
\end{align}
The inequality follows from the monotonicity under refinement in Lemma~\ref{LemMonotoneOnEeps}(i). Letting $m\to\infty$ and using~\eqref{EqDefGlobalDyadicLimitKernel_Inf}, together with the finiteness of the sum over $j=1,\dots,N$, yields
\begin{align}
\mu^{T,\beta,\varepsilon,x}_{[s,t]}(A)
\,\le\,
\sum_{j=1}^N \, \lim_{m\to\infty}\,
\widehat{\mathsf K}^{T,\beta,\varepsilon,\pi_m[s,t]}_{[s,t]}(x,A_j)
\,+\,
\widehat{\mathsf K}^{T,\beta,\varepsilon,\pi_1[s,t]}_{[s,t]}(x,B_N) \nonumber 
\,=\,
\sum_{j=1}^N \mu^{T,\beta,\varepsilon,x}_{[s,t]}(A_j)
\,+\,
\widehat{\mathsf K}^{T,\beta,\varepsilon,\pi_1[s,t]}_{[s,t]}(x,B_N) \,. \nonumber 
\end{align}
Since $\widehat{\mathsf K}^{T,\beta,\varepsilon,\pi_1[s,t]}_{[s,t]}(x,\cdot)$ is a finite measure and $B_N\downarrow\varnothing$, continuity from above implies $\widehat{\mathsf K}^{T,\beta,\varepsilon,\pi_1[s,t]}_{[s,t]}(x,B_N)\downarrow 0$ as $N\to\infty$. Hence, letting $N\to\infty$ and using the definition of the series as the limit of partial sums yields
\begin{align}
\mu^{T,\beta,\varepsilon,x}_{[s,t]}(A)
\le
\sum_{j\ge1}\mu^{T,\beta,\varepsilon,x}_{[s,t]}(A_j)\,, \nonumber
\end{align}
which together with~\eqref{EqGlobal_mu_ge_sum} yields countably additivity for $\mu^{T, \beta, \varepsilon, x}_{ [s,t]}$. \vspace{.2cm}

\noindent \textit{Borel measurability in the initial point.} Fix $A\in\mathcal B(E_\varepsilon)$. For each $m\ge1$, the map
$x\mapsto \widehat{\mathsf K}^{T,\beta,\varepsilon,\pi_m[s,t]}_{[s,t]}(x,A)$
is Borel measurable on $E_\varepsilon$ since
$\widehat{\mathsf K}^{T,\beta,\varepsilon,\pi_m[s,t]}_{[s,t]}$ is a sub-probability kernel by Lemma~\ref{LemMonotoneOnEeps}.
By \eqref{EqDefGlobalDyadicLimitKernel_Inf},
\[
x\longmapsto \widehat{\mathsf K}^{T,\beta,\varepsilon}_{[s,t]}(x,A)
\,=\,
\inf_{m\ge1}\widehat{\mathsf K}^{T,\beta,\varepsilon,\pi_m[s,t]}_{[s,t]}(x,A),
\]
and the infimum is taken over a countable family of Borel measurable functions. Therefore,
$x\mapsto \widehat{\mathsf K}^{T,\beta,\varepsilon}_{[s,t]}(x,A)$ is Borel measurable on $E_\varepsilon$. Thus, $\widehat{\mathsf K}^{T,\beta,\varepsilon}_{[s,t]}$ is a sub-probability kernel on $E_\varepsilon$. \vspace{.3cm}

\noindent Part (ii). Fix $x\in E_\varepsilon$ and $A\in\mathcal B(E_\varepsilon)$. Fix $0\le r<s<t\le T$ with $r,s,t\in\mathbb D_T$. Then there exists $M\ge1$ such that $r,s,t\in\mathcal D_M$. For every $m\ge M$, we have $\pi_m[r,t]=\pi_m[r,s]\ \vee\ \pi_m[s,t]$, and hence by~\eqref{EqCK_PartitionLevel},
\begin{align}\label{EqCK_Levelm}
\widehat{\mathsf K}^{T,\beta,\varepsilon,\pi_m[r,t]}_{[r,t]}(x,A)
\,=\,
\int_{E_\varepsilon}\, \widehat{\mathsf K}^{T,\beta,\varepsilon,\pi_m[r,s]}_{[r,s]}(x,dy)\,
\widehat{\mathsf K}^{T,\beta,\varepsilon,\pi_m[s,t]}_{[s,t]}(y,A) \,.
\end{align}
For $m\ge M$, define
\begin{align}\label{DefMuF}
\mu_m(\,\cdot\,)
\,:=\,\widehat{\mathsf K}^{T,\beta,\varepsilon,\pi_m[r,s]}_{[r,s]}(x,\,\cdot\,),
\qquad
f_m(y)
\,:=\,\widehat{\mathsf K}^{T,\beta,\varepsilon,\pi_m[s,t]}_{[s,t]}(y,A),
\qquad y\in E_\varepsilon.
\end{align}
Since $\pi_{m+1}[\cdot,\cdot]\succeq \pi_m[\cdot,\cdot]$, Lemma~\ref{LemMonotoneOnEeps}(i)
implies that for every $B\in\mathcal B(E_\varepsilon)$ and every $y\in E_\varepsilon$,
\begin{align} 
\mu_m(B)\downarrow \mu(B)
\,:=\, \widehat{\mathsf K}^{T,\beta,\varepsilon}_{[r,s]}(x,B)\,,
\quad 
f_m(y)\downarrow f(y)
\, :=\, \widehat{\mathsf K}^{T,\beta,\varepsilon}_{[s,t]}(y,A)\,, \quad \textup{and} \quad 
\widehat{\mathsf K}^{T,\beta,\varepsilon,\pi_m[r,t]}_{[r,t]}(x,A)
\downarrow
\widehat{\mathsf K}^{T,\beta,\varepsilon}_{[r,t]}(x,A)\, , \nonumber
\end{align}
where the equalities follow from~\eqref{EqDefGlobalDyadicLimitKernel_Inf}. Since $0\le f\le f_m\le 1$, we have
\begin{align} \label{fmLimit}
0 \,\le\, \int_{E_\varepsilon} f_m\,d\mu_m \,-\, \int_{E_\varepsilon} f\,d\mu
\,=\,&
\int_{E_\varepsilon} (f_m-f)\,d\mu_m
\,+\, \Big(\int_{E_\varepsilon} f\,d\mu_m \,-\, \int_{E_\varepsilon} f\,d\mu\Big) \nonumber \\
\,\le\,&
\underbrace{\int_{E_\varepsilon} (f_m-f)\,d\mu_M}_{\longrightarrow 0}
\,+\, \underbrace{\Big(\int_{E_\varepsilon} f\,d\mu_m \,-\, \int_{E_\varepsilon} f\,d\mu\Big)}_{\longrightarrow 0} \,\longrightarrow \, 0 \, ,
\end{align}
where in the second inequality we used $\mu_m\le \mu_M$ as measures since $\mu_m(B)\le \mu_M(B)$ for every $B\in\mathcal B(E_\varepsilon)$. Using~\eqref{EqCK_Levelm} and~\eqref{DefMuF} we obtain the first equality below.
\begin{align}
\widehat{\mathsf K}^{T,\beta,\varepsilon,\pi_m[r,t]}_{[r,t]}(x,A)
\,=\,
\int_{E_\varepsilon} f_m(y)\,\mu_m(dy) 
\, \overset{\eqref{fmLimit}}{\longrightarrow} \,\int_{E_\varepsilon} f(y)\,\mu(dy)
\,=\,
\int_{E_\varepsilon} \widehat{\mathsf K}^{T,\beta,\varepsilon}_{[s,t]}(y,A)\,\widehat{\mathsf K}^{T,\beta,\varepsilon}_{[r,s]}(x,dy) \nonumber
\end{align}
Since $\widehat{\mathsf K}^{T,\beta,\varepsilon,\pi_m[r,t]}_{[r,t]}(x,A) \downarrow \widehat{\mathsf K}^{T,\beta,\varepsilon}_{[r,t]}(x,A)$, the above yields~\eqref{EqCKGlobalDyadicLimitKernel}. Thus, it remains to show that the two terms in~\eqref{fmLimit} vanish as $m\to\infty$. We verify this by considering each term separately. \vspace{.2cm}

\noindent \textit{First term}. Since $0\le f_m-f\le 1$, $f_m-f\to0$ pointwise on
$E_\varepsilon$, and $\mu_M$ is a finite measure, the bounded convergence theorem yields that $\int_{E_\varepsilon}(f_m-f)\,d\mu_M \longrightarrow 0$. \vspace{.2cm}

\noindent \textit{Second term}. For the second term, we claim that
\begin{align} \label{LimitZeroTermII}
\int_{E_\varepsilon} f\,d\mu_m
\downarrow
\int_{E_\varepsilon} f\,d\mu\,.
\end{align}
Indeed, this is immediate when $f=\mathbf 1_B$ is an indicator of a Borel set
$B$, because then
\[
\int_{E_\varepsilon} \mathbf 1_B\,d\mu_m=\mu_m(B)\downarrow \mu(B)
\,=\,
\int_{E_\varepsilon} \mathbf 1_B\,d\mu.
\]
By linearity, the same holds for nonnegative simple functions. Now let
$g$ be any bounded nonnegative Borel function on $E_\varepsilon$. Choose
nonnegative simple functions $g_n\uparrow g$ pointwise. Then, for each fixed $n$, since $g_n$ is a nonnegative simple function, we have
\begin{align} \label{TestFunLimit}
    \int_{E_\varepsilon} g_n\,d\mu_m \downarrow \int_{E_\varepsilon} g_n\,d\mu
\qquad\text{as }m\to\infty\,.
\end{align}
Since $0\le g_n\le g$ and all the measures $\mu_m,\mu$ are sub-probability measures such that $\mu_m(B)\downarrow \mu(B)$ for every $B\in\mathcal B(E_\varepsilon)$, we have $\mu_m\le \mu_M$ and $\mu\le \mu_M$ as measures. Hence,
\begin{align} \label{TestFunLimitMuMIneq}
0\le \int_{E_\varepsilon}(g-g_n)\,d\mu_m \le \int_{E_\varepsilon}(g-g_n)\,d\mu_M,
\qquad
0\le \int_{E_\varepsilon}(g-g_n)\,d\mu \le \int_{E_\varepsilon}(g-g_n)\,d\mu_M\, .
\end{align}
Since $0\le g-g_n\le g$, $g-g_n\to0$ pointwise on $E_\varepsilon$, and $\mu_M$ is a finite measure, the bounded convergence theorem yields
\begin{align} \label{TestFunLimitMuM}
    \int_{E_\varepsilon}(g-g_n)\,d\mu_M \longrightarrow 0
\qquad\text{as } n\to\infty \, .
\end{align}
Hence, for every $m\ge M$, using~\eqref{TestFunLimitMuMIneq} we obtain the second inequality below.
\begin{align*}
\left|\int_{E_\varepsilon} g\,d\mu_m-\int_{E_\varepsilon} g\,d\mu\right|
&\le
\left|\int_{E_\varepsilon} g_n\,d\mu_m-\int_{E_\varepsilon} g_n\,d\mu\right|
+\int_{E_\varepsilon}(g-g_n)\,d\mu_m
+\int_{E_\varepsilon}(g-g_n)\,d\mu \\
&\le
\left|\int_{E_\varepsilon} g_n\,d\mu_m-\int_{E_\varepsilon} g_n\,d\mu\right|
+2\int_{E_\varepsilon}(g-g_n)\,d\mu_M
\end{align*}
Passing first $m\to\infty$ (and using~\eqref{TestFunLimit}) and then $n\to\infty$ (and using~\eqref{TestFunLimitMuM}), we obtain~\eqref{LimitZeroTermII} for $g$.
\end{proof}

\subsection{The transition probability kernel \texorpdfstring{$\boldsymbol{\mathsf K}^{T,\beta,\varepsilon}_{[s,t]}$}{Lg}}\label{SecTransitionProbKernel}

\begin{proof}[Proof of Proposition~\ref{PropCKFullProbKernel}]
\noindent Part (i).
For $x\in E_\varepsilon$, the definition~\eqref{EqDefMeshLimitFull} can be written as
\[
\mathsf K^{T,\beta,\varepsilon}_{[s,t]}(x,\cdot)
=
\widehat{\mathsf K}^{T,\beta,\varepsilon}_{[s,t]}(x,\cdot)
+
\bigl(1-\widehat{\mathsf K}^{T,\beta,\varepsilon}_{[s,t]}(x,E_\varepsilon)\bigr)\delta_\Delta(\cdot),
\]
where $\delta_\Delta$ denotes the Dirac mass at $\Delta$ and the measure
$\widehat{\mathsf K}^{T,\beta,\varepsilon}_{[s,t]}(x,\cdot)$ is extended by zero on
$\{\Delta\}$.
Since $\widehat{\mathsf K}^{T,\beta,\varepsilon}_{[s,t]}(x,\cdot)$ is a
sub-probability measure on $E_\varepsilon$, the right-hand side defines a
probability measure on $\overline E_\varepsilon$. If $x=\Delta$, then $\mathsf K^{T,\beta,\varepsilon}_{[s,t]}(\Delta,\cdot)=\delta_\Delta(\cdot)$, which is also a probability measure.

Next, for each fixed $B\in\mathcal B(\overline E_\varepsilon)$, consider the map
\begin{align*}
x\longmapsto \mathsf K^{T,\beta,\varepsilon}_{[s,t]}(x,B)
\,:=\,
\mathbf 1_{E_\varepsilon}(x)\Big[
\widehat{\mathsf K}^{T,\beta,\varepsilon}_{[s,t]}(x,B\cap E_\varepsilon)
\, + \,
\mathbf 1_{\{\Delta\in B\}}
\bigl(1-\widehat{\mathsf K}^{T,\beta,\varepsilon}_{[s,t]}(x,E_\varepsilon)\bigr)
\Big]
\,+ \, \mathbf 1_{\{x=\Delta\}}\mathbf 1_{\{\Delta\in B\}}\, .
\end{align*}
The maps $x\mapsto \mathbf 1_{E_\varepsilon}(x)$ and $x\mapsto \mathbf 1_{\{x=\Delta\}}$
are Borel measurable on $\overline E_\varepsilon$, and by the measurability of
the sub-probability kernel
$\widehat{\mathsf K}^{T,\beta,\varepsilon}_{[s,t]}$,
the maps $x\longmapsto \widehat{\mathsf K}^{T,\beta,\varepsilon}_{[s,t]}(x,B\cap E_\varepsilon)$ and $x\longmapsto \widehat{\mathsf K}^{T,\beta,\varepsilon}_{[s,t]}(x,E_\varepsilon)$ are Borel measurable on $E_\varepsilon$. Extending these by zero at $\Delta$,
they become Borel measurable on $\overline E_\varepsilon$. It follows that
$x\mapsto \mathsf K^{T,\beta,\varepsilon}_{[s,t]}(x,B)$ is Borel measurable. \vspace{.3cm}

\noindent Part (ii). Fix $x\in\overline E_\varepsilon$ and
$A\in\mathcal B(\overline E_\varepsilon)$. Fix $0\le r<s<t\le T$ with $r,s,t\in\mathbb D_T$. \vspace{.1cm}

\noindent\textit{Case 1: $x=\Delta$.}
By the definition~\eqref{EqDefMeshLimitFull}, the cemetery state is absorbing, that is, $\mathsf K^{T,\beta,\varepsilon}_{[r,s]}(\Delta,\cdot)=\delta_\Delta(\cdot)$ and $\mathsf K^{T,\beta,\varepsilon}_{[r,t]}(\Delta,\cdot)=\delta_\Delta(\cdot)$. Hence,
\begin{align*}
\int_{\overline E_\varepsilon}
\mathsf K^{T,\beta,\varepsilon}_{[r,s]}(\Delta,dy)\,
\mathsf K^{T,\beta,\varepsilon}_{[s,t]}(y,A)
=
\int_{\overline E_\varepsilon}
\delta_\Delta(dy)\,
\mathsf K^{T,\beta,\varepsilon}_{[s,t]}(y,A) 
=
\mathsf K^{T,\beta,\varepsilon}_{[s,t]}(\Delta,A) 
=
\mathbf 1_{\{\Delta\in A\}} 
=
\mathsf K^{T,\beta,\varepsilon}_{[r,t]}(\Delta,A).
\end{align*}
Thus~\eqref{EqCKExtendedProbKernel} holds when $x=\Delta$. \vspace{.1cm}

\noindent\textit{Case 2: $x\in E_\varepsilon$.}
Write $A_0:=A\cap E_\varepsilon$. Then $A_0\in\mathcal B(E_\varepsilon)$ and, by~\eqref{EqDefMeshLimitFull},
\begin{align}\label{EqDecomposeKrt}
\mathsf K^{T,\beta,\varepsilon}_{[r,t]}(x,A)
=
\widehat{\mathsf K}^{T,\beta,\varepsilon}_{[r,t]}(x,A_0)
+
\mathbf 1_{\{\Delta\in A\}}
\Big(1-\widehat{\mathsf K}^{T,\beta,\varepsilon}_{[r,t]}(x,E_\varepsilon)\Big).
\end{align}
On the other hand since $\overline E_\varepsilon=E_\varepsilon\cup\{\Delta\}$ is a disjoint union, we have
\begin{align*}
\int_{\overline E_\varepsilon}
\mathsf K^{T,\beta,\varepsilon}_{[r,s]}(x,dy)\,
\mathsf K^{T,\beta,\varepsilon}_{[s,t]}(y,A) 
\,=\,&
\int_{E_\varepsilon}
\mathsf K^{T,\beta,\varepsilon}_{[r,s]}(x,dy)\,
\mathsf K^{T,\beta,\varepsilon}_{[s,t]}(y,A)
\,+\,
\mathsf K^{T,\beta,\varepsilon}_{[r,s]}(x,\{\Delta\})\,
\mathsf K^{T,\beta,\varepsilon}_{[s,t]}(\Delta,A) \nonumber \\
\,=\,&
\int_{E_\varepsilon}
\widehat{\mathsf K}^{T,\beta,\varepsilon}_{[r,s]}(x,dy)\,
\widehat{\mathsf K}^{T,\beta,\varepsilon}_{[s,t]}(y,A_0)
\,+\,\mathbf 1_{\{\Delta\in A\}}
\int_{E_\varepsilon}
\widehat{\mathsf K}^{T,\beta,\varepsilon}_{[r,s]}(x,dy)\,
\Bigl(1-\widehat{\mathsf K}^{T,\beta,\varepsilon}_{[s,t]}(y,E_\varepsilon)\Bigr) \nonumber \\
\,+\, &\mathbf 1_{\{\Delta\in A\}}
\Bigl(1-\widehat{\mathsf K}^{T,\beta,\varepsilon}_{[r,s]}(x,E_\varepsilon)\Bigr)\, \nonumber \\
\,=\,&
\int_{E_\varepsilon}
\widehat{\mathsf K}^{T,\beta,\varepsilon}_{[r,s]}(x,dy)\,
\widehat{\mathsf K}^{T,\beta,\varepsilon}_{[s,t]}(y,A_0)
\,+\,
\mathbf 1_{\{\Delta\in A\}}
\bigg[
1-
\int_{E_\varepsilon}
\widehat{\mathsf K}^{T,\beta,\varepsilon}_{[r,s]}(x,dy)\,
\widehat{\mathsf K}^{T,\beta,\varepsilon}_{[s,t]}(y,E_\varepsilon)
\bigg]
\end{align*}
where the second equality uses~\eqref{EqDecomposeKrt} with $r=s$ and $x=y$, and the following identities from~\eqref{EqDefMeshLimitFull}
\[
\mathsf K^{T,\beta,\varepsilon}_{[r,s]}(x,dy)
\,=\,
\widehat{\mathsf K}^{T,\beta,\varepsilon}_{[r,s]}(x,dy)
\quad\text{on }E_\varepsilon,
\quad
\mathsf K^{T,\beta,\varepsilon}_{[r,s]}(x,\{\Delta\})
\,=\,
1-\widehat{\mathsf K}^{T,\beta,\varepsilon}_{[r,s]}(x,E_\varepsilon),
\quad
\mathsf K^{T,\beta,\varepsilon}_{[s,t]}(\Delta,A)
\,=\,
\mathbf 1_{\{\Delta\in A\}} \,.
\]
Since $r,s,t\in\mathbb D_T$, the Chapman-Kolmogorov identity~\eqref{EqCKGlobalDyadicLimitKernel} applied to the sets $A_0$ and $E_\varepsilon$, yields the first equality below.
\begin{align*}
\int_{\overline E_\varepsilon}
\mathsf K^{T,\beta,\varepsilon}_{[r,s]}(x,dy)\,
\mathsf K^{T,\beta,\varepsilon}_{[s,t]}(y,A)
\,=\,
\widehat{\mathsf K}^{T,\beta,\varepsilon}_{[r,t]}(x,A_0)
\,+\,
\mathbf 1_{\{\Delta\in A\}}
\Bigl(1-\widehat{\mathsf K}^{T,\beta,\varepsilon}_{[r,t]}(x,E_\varepsilon)\Bigr)
\,\overset{\eqref{EqDecomposeKrt}}{=}\,
\mathsf K^{T,\beta,\varepsilon}_{[r,t]}(x,A)
\end{align*}
Thus~\eqref{EqCKExtendedProbKernel} also holds when $x\in E_\varepsilon$. Combining the two cases completes the proof of Part~(ii).
\end{proof}

\section{Detailed construction of the Markov process} \label{DetailedConstructionOfTheMarkovprocess}

The goal of the present section is to give a detailed account of the
Markovian framework introduced in
Section~\ref{TheAssociatedMarkovProcess} and to justify the various
properties stated there. In
Section~\ref{TheDyadicSkeleton}, we construct the dyadic skeleton
associated with the family of transition probability kernels
$\{\mathsf K^{T,\beta,\varepsilon}_{[s,t]}\}_{\substack{s,t\in\mathbb D_T\\ s<t}}$
via the Kolmogorov extension theorem, establish the corresponding
Markov property, and study the associated dyadic killing time together
with the absorbing behavior of the cemetery state. In
Section~\ref{TheStepFunctionInterpolations}, we study the associated
step-function interpolations, prove their basic pathwise and
measurability properties, construct the corresponding interpolated
path-space laws on the Skorokhod space
$D([0,T];\overline E_\varepsilon)$, and establish some preliminary
tightness properties of the associated one-time marginals.

\subsection{The dyadic skeleton}\label{TheDyadicSkeleton}
Recall that the canonical filtered measurable space associated with the dyadic skeleton $X=\{X_t\}_{t\in\mathbb D_T}$ is given by
\[
\big(\Omega^{\mathbb D_T}_\varepsilon,
\mathcal F^{\mathbb D_T}_\varepsilon,
\{\mathcal F_t^{\mathbb D_T,\varepsilon}\}_{t\in\mathbb D_T}\big)\,.
\]
Notice that the Markov property~\eqref{EqMarkovPropertyDyadic} implies that for every $A\in\mathcal B(\overline E_\varepsilon)$,
\begin{align}\label{EqMarkovPropertyIndicator}
\mathbb P_x^{T,\beta,\varepsilon}\!\left[
X_t\in A \,\big|\,\mathcal F_s^{\mathbb D_T,\varepsilon}
\right]
=
\mathsf K^{T,\beta,\varepsilon}_{[s,t]}(X_s,A)
\qquad
\mathbb P_x^{T,\beta,\varepsilon}\textup{-a.s.}
\end{align}

\begin{proof}[Proof of Proposition~\ref{PropDyadicSkeletonProcess}]
For each finite ordered set of dyadic times $0=t_0<t_1<\cdots<t_n\le T$ with $t_0,t_1,\dots,t_n\in\mathbb D_T$, define a probability measure $\nu^{x}_{t_1,\dots,t_n}$ on $\mathcal B(\overline E_\varepsilon)^{\otimes n}$ by
\begin{align}\label{EqDefNuDyadic}
\nu^{x}_{t_1,\dots,t_n}(A_1\times\cdots\times A_n)
\,:=\,
\int_{A_1}\mathsf K^{T,\beta,\varepsilon}_{[0,t_1]}(x,dy_1)
\int_{A_2}\mathsf K^{T,\beta,\varepsilon}_{[t_1,t_2]}(y_1,dy_2)
\cdots
\int_{A_n}\mathsf K^{T,\beta,\varepsilon}_{[t_{n-1},t_n]}(y_{n-1},dy_n)
\end{align}
for $A_1,\dots,A_n\in\mathcal B(\overline E_\varepsilon)$. Because each $\mathsf K^{T,\beta,\varepsilon}_{[t_{j-1},t_j]}$ is a probability kernel by Proposition~\ref{PropCKFullProbKernel}(i), it follows by repeated integration that~\eqref{EqDefNuDyadic} defines a probability measure on $\overline E_\varepsilon^n$. It remains to check consistency of this family of finite-dimensional distributions. For this, let $0=t_0<t_1< \cdots < t_n \le T$ with $t_0,t_1,\dots,t_n \in \mathbb D_T$, and $A_1,\dots,A_n \in\mathcal B(\overline E_\varepsilon)$. \vspace{.2cm}

\noindent \textit{Removing the last point.} Let $A_n=\overline E_\varepsilon$, then using~\eqref{EqDefNuDyadic}, we obtain
\begin{align*}
\nu^{x}_{t_1,\dots,t_n}
\bigl(A_1\times\cdots\times A_{n-1}\times \overline E_\varepsilon\bigr)
=&
\int_{A_1}\mathsf K^{T,\beta,\varepsilon}_{[0,t_1]}(x,dy_1)
\cdots
\int_{A_{n-1}}\mathsf K^{T,\beta,\varepsilon}_{[t_{n-2},t_{n-1}]}(y_{n-2},dy_{n-1})
\int_{\overline E_\varepsilon}\mathsf K^{T,\beta,\varepsilon}_{[t_{n-1},t_n]}(y_{n-1},dy_n) \nonumber \\
=\,&
\nu^{x}_{t_1,\dots,t_{n-1}}(A_1\times\cdots\times A_{n-1})
\end{align*}
where the second equality uses that, for every $y_{n-1} \in \overline E_\varepsilon$, 
\[
\int_{\overline E_\varepsilon}\mathsf K^{T,\beta,\varepsilon}_{[t_{n-1},t_n]}(y_{n-1},dy_n)
\,=\,
\mathsf K^{T,\beta,\varepsilon}_{[t_{n-1},t_n]}(y_{n-1},\overline E_\varepsilon)
\,=\,1\, .
\]

\noindent \textit{Removing an intermediate point.} Let $A_k = \overline E_\varepsilon$ for some $1\le k\le n-1$, then
\begin{align*}
\nu^x_{t_1,\dots,t_n}
\bigl(
A_1\times\cdots\times A_{k-1}\times & \overline E_\varepsilon\times A_{k+1}\times\cdots\times A_n
\bigr) \nonumber \\
\,=\,&
\int_{A_1}\mathsf K^{T,\beta,\varepsilon}_{[0,t_1]}(x,dy_1)\cdots
\int_{A_{k-1}}\mathsf K^{T,\beta,\varepsilon}_{[t_{k-2},t_{k-1}]}(y_{k-2},dy_{k-1})
\\
\,\times \, &
\underbrace{\int_{\overline E_\varepsilon}\mathsf K^{T,\beta,\varepsilon}_{[t_{k-1},t_k]}(y_{k-1},dy_k) \,
\int_{A_{k+1}}\mathsf K^{T,\beta,\varepsilon}_{[t_k,t_{k+1}]}(y_k,dy_{k+1})}
\cdots
\int_{A_n}\mathsf K^{T,\beta,\varepsilon}_{[t_{n-1},t_n]}(y_{n-1},dy_n)\,.
\end{align*}
Combining the two successive integrals in the underbraced we obtain  the first equality below.
\begin{align*}
\int_{\overline E_\varepsilon}\mathsf K^{T,\beta,\varepsilon}_{[t_{k-1},t_k]}(y_{k-1},dy_k) \,
\int_{A_{k+1}}\mathsf K^{T,\beta,\varepsilon}_{[t_k,t_{k+1}]}(y_k,dy_{k+1})
\,=\,&
\int_{\overline E_\varepsilon}\mathsf K^{T,\beta,\varepsilon}_{[t_{k-1},t_k]}(y_{k-1},dy_k) \,
\mathsf K^{T,\beta,\varepsilon}_{[t_k,t_{k+1}]}(y_k,A_{k+1}) \nonumber \\
\,=\,&
\int_{A_{k+1}}\mathsf K^{T,\beta,\varepsilon}_{[t_{k-1},t_{k+1}]}(y_{k-1},dy_{k+1})
\end{align*}
The second equality uses the Chapman-Kolmogorov identity~\eqref{EqCKExtendedProbKernel}. Thus, we have
\begin{align}
&\nu^x_{t_1,\dots,t_n}
\bigl(
A_1\times\cdots\times A_{k-1}\times \overline E_\varepsilon\times A_{k+1}\times\cdots\times A_n
\bigr)
=
\nu^x_{t_1,\dots,t_{k-1},t_{k+1},\dots,t_n}
\bigl(
A_1\times\cdots\times A_{k-1}\times A_{k+1}\times\cdots\times A_n
\bigr)\,.
\nonumber
\end{align}

\noindent \textit{Removing arbitrary intermediate points.} Suppose there exists $1\le i_1<\cdots<i_k\le n$ such that $A_{i_j} = \overline E_\varepsilon$ for $j \notin \{i_1,\cdots, i_k\}$. Define
\[
B_j\,:=\,
\begin{cases}
A_j, & \text{if } j\in\{i_1,\dots,i_k\},\\
\overline E_\varepsilon, & \text{otherwise.}
\end{cases}
\]
Then
\begin{align*}
\nu^x_{t_1,\dots,t_n}(B_1\times\cdots\times B_n)
&=
\int_{B_1}\mathsf K^{T,\beta,\varepsilon}_{[0,t_1]}(x,dy_1)
\cdots
\int_{B_n}\mathsf K^{T,\beta,\varepsilon}_{[t_{n-1},t_n]}(y_{n-1},dy_n).
\end{align*}
For every $j\notin\{i_1,\dots,i_k\}$, since $B_j=\overline E_\varepsilon$ and
$\mathsf K^{T,\beta,\varepsilon}_{[t_{j-1},t_j]}(y_{j-1},\overline E_\varepsilon)=1$,
the corresponding integrations contribute only total mass one. Repeated application of the Chapman-Kolmogorov identity~\eqref{EqCKExtendedProbKernel} then yields
\begin{align*}
\nu^x_{t_1,\dots,t_n}(B_1\times\cdots\times B_n)
&=
\int_{A_{i_1}}\mathsf K^{T,\beta,\varepsilon}_{[0,t_{i_1}]}(x,dy_{i_1})
\int_{A_{i_2}}\mathsf K^{T,\beta,\varepsilon}_{[t_{i_1},t_{i_2}]}(y_{i_1},dy_{i_2})
\cdots
\int_{A_{i_k}}\mathsf K^{T,\beta,\varepsilon}_{[t_{i_{k-1}},t_{i_k}]}(y_{i_{k-1}},dy_{i_k}) \\
&=
\nu^x_{t_{i_1},\dots,t_{i_k}}
\bigl(A_{i_1}\times\cdots\times A_{i_k}\bigr).
\end{align*}
Hence the family $\{\nu^x_{t_1,\dots,t_n}\}$ is consistent. Therefore, by the Kolmogorov extension theorem, there exists a probability measure
$\mathbb P_{x}^{T,\beta,\varepsilon}$ on
$\bigl(\Omega^{\mathbb D_T}_\varepsilon,\mathcal F^{\mathbb D_T}_\varepsilon\bigr)$
whose finite-dimensional distributions are given by~\eqref{EqDefNuDyadic}. In particular, under
$\mathbb P_{x}^{T,\beta,\varepsilon}$, the coordinate process
$\{X_t\}_{t\in\mathbb D_T}$ has initial distribution $\delta_x$. Moreover, by the product structure in~\eqref{EqDefNuDyadic} and the Chapman--Kolmogorov identity~\eqref{EqCKExtendedProbKernel}, the process is a time--inhomogeneous Markov process with transition kernels
$\{\mathsf K^{T,\beta,\varepsilon}_{[s,t]}\}_{\substack{s,t\in\mathbb D_T\\ s<t}}$.
\end{proof}

Define the dyadic killing time by
\begin{align}\label{EqDefDyadicKillingTime}
\tau_\varepsilon^{\mathbb D_T}
(\omega)
\,:=\,
\inf\big\{t\in\mathbb D_T:\ X_t(\omega)=\Delta\big\},
\qquad \omega\in\Omega^{\mathbb D_T}_\varepsilon,
\end{align}
with the convention $\inf\varnothing:=\infty$. Part~(iii) below may equivalently be written as $\{\tau_\varepsilon^{\mathbb D_T}>t\}=\{X_t\in E_\varepsilon\}$ almost surely under $\mathbb P_x^{T,\beta,\varepsilon}$. 

\begin{proposition} \label{PropDyadicKillingTime}
Fix $T,\beta,\varepsilon>0$ and $x\in \overline E_\varepsilon$. The following hold.
\begin{enumerate}[(i)]
\item The random time $\tau_\varepsilon^{\mathbb D_T}$ is a stopping time with respect to the dyadic filtration
$\{\mathcal F_t^{\mathbb D_T,\varepsilon}\}_{t\in\mathbb D_T}$.

\item The cemetery state is absorbing along the dyadic skeleton: whenever $s,t\in\mathbb D_T$ with $0\le s\le t\le T$,
\begin{align}
\mathbb P_x^{T,\beta,\varepsilon}\bigl[X_t=\Delta \,\big|\, \mathcal F_s^{\mathbb D_T,\varepsilon}\bigr]
=1
\quad
\mathbb P_x^{T,\beta,\varepsilon}\textup{-a.s. on }\{X_s=\Delta\}. \nonumber
\end{align}

\item At dyadic times, hitting the cemetery state is equivalent to having exited:
\[
\{\tau_\varepsilon^{\mathbb D_T}\le t\}
=
\{X_t=\Delta\}
\quad
\mathbb P_x^{T,\beta,\varepsilon}\textup{-a.s.}, 
\qquad t\in\mathbb D_T.
\]
\end{enumerate}
\end{proposition}

\begin{proof}
\noindent Part (i). Fix $t\in\mathbb D_T$. Since $\mathbb D_T\cap[0,t]$ is countable, we have the equality below.
\[
\{\tau_\varepsilon^{\mathbb D_T}\le t\}
\,=\,
\bigcup_{r\in\mathbb D_T,\; r\le t}\, \{X_r=\Delta\} \, \in \, \mathcal F_t^{\mathbb D_T,\varepsilon}
\]
Where we have used that for each $r\in\mathbb D_T$ with $r\le t$, the event $\{X_r=\Delta\} \in \mathcal F_r^{\mathbb D_T,\varepsilon} \subset \mathcal F_t^{\mathbb D_T,\varepsilon}$. Hence, $\tau_\varepsilon^{\mathbb D_T}$ is a stopping time with respect to $\{\mathcal F_t^{\mathbb D_T,\varepsilon}\}_{t\in\mathbb D_T}$. \vspace{.2cm}

\noindent Part (ii). Fix $s,t\in\mathbb D_T$ with $0\le s\le t\le T$. By the Markov property~\eqref{EqMarkovPropertyIndicator} with $A:=\{\Delta\}\in\mathcal B(\overline E_\varepsilon)$,
\begin{align}
\mathbb P_x^{T,\beta,\varepsilon}\big[X_t=\Delta\,\big|\,\mathcal F_s^{\mathbb D_T,\varepsilon}\big]
=
\mathsf K_{[s,t]}^{T,\beta,\varepsilon}(X_s,\{\Delta\})
\qquad
\mathbb P_x^{T,\beta,\varepsilon}\textup{-a.s.} \nonumber
\end{align}
On the event $\{X_s=\Delta\}$, the right side is equal to $1$ by~\eqref{EqDefMeshLimitFull}. \vspace{.2cm}

\noindent Part (iii). Fix $t\in\mathbb D_T$. If $\tau_\varepsilon^{\mathbb D_T}\le t$, then there exists $r\in\mathbb D_T$ with
$r\le t$ and $X_r=\Delta$. By Part~(ii), we have
\[
\mathbb P_x^{T,\beta,\varepsilon}\big[
X_u=\Delta \text{ for all } u\in\mathbb D_T,\ u\ge \tau_\varepsilon^{\mathbb D_T}
\big]
=1.
\]
That is, once the process reaches $\Delta$, it remains there at all later dyadic times. Hence $X_t=\Delta$ on the event $\{\tau_\varepsilon^{\mathbb D_T}\le t\}$, $\mathbb P_x^{T,\beta,\varepsilon}$-a.s., and therefore
\[
\{\tau_\varepsilon^{\mathbb D_T}\le t\}\subseteq \{X_t=\Delta\}
\quad
\mathbb P_x^{T,\beta,\varepsilon}\textup{-a.s.}
\]
Conversely, if $X_t=\Delta$, then by the definition of the infimum in~\eqref{EqDefDyadicKillingTime}, namely the first dyadic time at which the process hits $\Delta$, it follows that $\tau_\varepsilon^{\mathbb D_T}\le t$.
Hence $\{X_t=\Delta\}\subseteq \{\tau_\varepsilon^{\mathbb D_T}\le t\}$ pointwise.
\end{proof}

\subsection{The step-function interpolations}\label{TheStepFunctionInterpolations}

For each $m\ge1$, define the left-endpoint projection
$\vartheta_m:[0,T]\to\mathcal D_m$ by
\[
\vartheta_m(t)
\,:=\,
\max\big(\mathcal D_m\cap[0,t]\big)
=
\frac{T}{2^m}
\Big\lfloor
\frac{2^m t}{T}
\Big\rfloor,
\qquad t\in[0,T].
\]
Using this projection, the step-function interpolation
$X^{(m)}=\{X_t^{(m)}\}_{t\in[0,T]}$ is given by
\begin{align}\label{EqDefDyadicStepInterpolation}
X_t^{(m)}(\omega)
=
X_{\vartheta_m(t)}(\omega),
\qquad t\in[0,T],\ \omega\in\Omega_\varepsilon^{\mathbb D_T}.
\end{align}
The following proposition summarizes basic regularity properties of $X^{(m)}$.

\begin{proposition} \label{PropDyadicStepInterpolation}
Fix $T,\beta,\varepsilon>0$, $m \ge 1$, and $x\in\overline E_\varepsilon$. Let  $X^{(m)}=\{X_t^{(m)}\}_{t\in[0,T]}$ be the step-function process defined in~\eqref{EqDefDyadicStepInterpolation}. The following hold.
\begin{enumerate}[(i)]
\item For every $\omega\in\Omega_\varepsilon^{\mathbb D_T}$, the path
$t\mapsto X_t^{(m)}(\omega)$ is piecewise constant on the intervals $\big[k2^{-m}T,(k+1)2^{-m}T\big)$ for $k=0,1,\dots,2^m-1$
and is c\`adl\`ag on $[0,T]$ with values in $\overline E_\varepsilon$.

\item The function $X^{(m)}_{\cdot}:\Omega_\varepsilon^{\mathbb D_T}\to D([0,T];\overline E_\varepsilon)$, defined by $X^{(m)}_{\cdot}(\omega):=\big(t\mapsto X_t^{(m)}(\omega)\big)$ for $\omega\in\Omega_\varepsilon^{\mathbb D_T}$, is $\mathcal F^{\mathbb D_T}_\varepsilon/\mathcal B(D([0,T];\overline E_\varepsilon))$-measurable.

\item For every $t\in\mathcal D_m$, we have $X_t^{(m)} = X_t$. In particular, the finite-dimensional distributions of $X^{(m)}$ at times belonging to $\mathcal D_m$ agree with those of the dyadic skeleton $\{X_t\}_{t\in\mathbb D_T}$ under $\mathbb P_x^{T,\beta,\varepsilon}$.
\end{enumerate}
\end{proposition}

\begin{proof}
\noindent Part (i). For each $\omega\in\Omega_\varepsilon^{\mathbb D_T}$ and each $k\in\{0,1,\dots,2^m-1\}$, we have $\vartheta_m(t)=k2^{-m}T$ for all $t\in \big[k2^{-m}T,(k+1)2^{-m}T\big)$ and therefore
\[
X_t^{(m)}(\omega)
\,=\,
X_{k2^{-m}T}(\omega)
\qquad\text{for all }t\in \big[k2^{-m}T,(k+1)2^{-m}T \big) \, .
\]
Thus the path $t\mapsto X_t^{(m)}(\omega)$ is piecewise constant on the above dyadic intervals. In particular, it is right-continuous on each interval and has finite left limits at each grid point. Since also $X_T^{(m)}(\omega)=X_T(\omega)$, the path is c\`adl\`ag on $[0,T]$. \vspace{.2cm}

\noindent Part (ii). The given map $X^{(m)}_{\cdot}:\Omega_\varepsilon^{\mathbb D_T}
\to D([0,T];\overline E_\varepsilon)$ is well-defined since for each $\omega \in \Omega_\varepsilon^{\mathbb D_T}$, $X^{(m)}_{\cdot}(\omega)(t):=X_t^{(m)}(\omega)$ for $t\in[0,T]$. Hence, $X^{(m)}_{\cdot}(\omega) \in D([0,T];\overline E_\varepsilon)$ by part~(i).

For each $t\in[0,T]$, let $e_t : D([0,T];\overline E_\varepsilon)\to \overline E_\varepsilon$ be the evaluation (coordinate) map defined by $e_t(p):=p(t)$. Then, for every
$\omega\in\Omega_\varepsilon^{\mathbb D_T}$,
\[
(e_t\circ X^{(m)}_{\cdot})(\omega)
=
e_t\big(X^{(m)}_{\cdot}(\omega)\big)
=
\big(X^{(m)}_{\cdot}(\omega)\big)(t)
=
X_t^{(m)}(\omega)
=
X_{\vartheta_m(t)}(\omega).
\]
Since $\vartheta_m(t)\in\mathcal D_m\subset\mathbb D_T$, the map
$\omega\mapsto X_{\vartheta_m(t)}(\omega)$ is a coordinate map on
$\Omega_\varepsilon^{\mathbb D_T}
=\overline E_\varepsilon^{\mathbb D_T}$ and is therefore
$\mathcal F_\varepsilon^{\mathbb D_T}/\mathcal B(\overline E_\varepsilon)$--measurable.
Thus $e_t\circ X^{(m)}_{\cdot}$ is
$\mathcal F_\varepsilon^{\mathbb D_T}/\mathcal B(\overline E_\varepsilon)$--measurable
for every $t\in[0,T]$.

Finally, since $\overline E_\varepsilon$ is Polish, the space
$D([0,T];\overline E_\varepsilon)$ is equipped with its Borel
$\sigma$--algebra
$\mathcal B(D([0,T];\overline E_\varepsilon))$ for the Skorohod topology, which is generated by
the coordinate maps $\{e_t: t\in[0,T]\}$. Thus,
$X^{(m)}_{\cdot}$ is $\mathcal F_\varepsilon^{\mathbb D_T}/
\mathcal B(D([0,T];\overline E_\varepsilon))
\textup{--measurable}$. \vspace{.3cm}

\noindent Part (iii). If $t\in\mathcal D_m$, then by definition of $\vartheta_m$, we have $\vartheta_m(t)=t$. Hence, $X_t^{(m)}=X_{\vartheta_m(t)}=X_t$. Thus, for $0=t_0<t_1<\cdots<t_n\le T$ with $t_0,t_1,\dots,t_n\in\mathcal D_m$, and let $A_0,\dots,A_n\in\mathcal B(\overline E_\varepsilon)$, we have $X_{t_j}^{(m)}=X_{t_j}$ for all $j=0,1,\dots,n$. Therefore,
\begin{align*}
\mathbb P_x^{T,\beta,\varepsilon}\bigl(
X_{t_0}^{(m)}\in A_0,\dots,X_{t_n}^{(m)}\in A_n
\bigr)
\,=\,&
\mathbb P_x^{T,\beta,\varepsilon}\bigl(
X_{t_0}\in A_0,\dots,X_{t_n}\in A_n
\bigr) \nonumber \\
\,=\,&
\mathbf 1_{A_0}(x)
\int_{A_1}\mathsf K_{[t_0,t_1]}^{T,\beta,\varepsilon}(x,dy_1)
\int_{A_2}\mathsf K_{[t_1,t_2]}^{T,\beta,\varepsilon}(y_1,dy_2)
\cdots
\int_{A_n}\mathsf K_{[t_{n-1},t_n]}^{T,\beta,\varepsilon}(y_{n-1},dy_n),
\end{align*}
where the second equality follows from iterating the Markov property~\eqref{EqMarkovPropertyDyadic}.
\end{proof}

The following lemma shows that, under $\mathbb P_{x}^{T,\beta,\varepsilon}$, the finite-dimensional distributions of the step-function interpolated process
$\{X_t^{(m)}\}_{t\in [0,T]}$ are completely determined by the finite-dimensional
distributions of the dyadic skeleton $\{X_r\}_{r\in\mathcal D_m}$.
\begin{lemma}  \label{LemKernelRepresentationInterpolatedTimes}
Fix $T,\beta,\varepsilon>0$, $m\ge1$, and $x\in\overline E_\varepsilon$. For every $0\le t_1<\cdots<t_n\le T$, under $\mathbb P_{x}^{T,\beta,\varepsilon}$, the joint distribution of the vector $\big(X_{t_1}^{(m)}, \cdots , X_{t_n}^{(m)} \big)$ is determined by the joint distribution of $\big(X_{r_1}, \dots ,X_{r_\ell}\big)$, where $r_1<\cdots<r_\ell$ are the distinct values among $\vartheta_m(t_1),\dots,\vartheta_m(t_n)$.
\end{lemma}

\begin{proof}
Fix $A_1,\dots,A_n\in\mathcal B(\overline E_\varepsilon)$. Notice that $\vartheta_m(t_1),\dots,\vartheta_m(t_n) \in \mathcal D_m$, and the sequence
$\big(\vartheta_m(t_1),\dots,\vartheta_m(t_n)\big)$ is nondecreasing because $\vartheta_m$ is nondecreasing and $t_1<\cdots<t_n$. If the values $\vartheta_m(t_1),\dots,\vartheta_m(t_n)$ are all distinct, then the result holds trivially.

Now suppose repetitions occur. Let $r_1<\cdots<r_\ell$ be the distinct values appearing in $\big(\vartheta_m(t_1),\dots,\vartheta_m(t_n)\big)$, and for each
$k=1,\dots,\ell$ define $I_k  :=  \{j\in\{1,\dots,n\}: \vartheta_m(t_j) =r_k\}$ and $B_k := \bigcap_{j\in I_k} A_j$. Then,
\[
\{X_{\vartheta_m(t_1)}\in A_1,\dots,X_{\vartheta_m(t_n)}\in A_n\}
=
\bigcap_{k=1}^{\ell}\;\bigcap_{j\in I_k}\{X_{r_k}\in A_j\}
=
\bigcap_{k=1}^{\ell}
\Big\{X_{r_k}\in \bigcap_{j\in I_k} A_j\Big\}
=
\bigcap_{k=1}^{\ell}\{X_{r_k}\in B_k\}.
\]
Thus, we have
\[
\mathbb P_x^{T,\beta,\varepsilon}
\big(
X_{\vartheta_m(t_1)}\in A_1,\dots,X_{\vartheta_m(t_n)}\in A_n
\big)
=
\mathbb P_x^{T,\beta,\varepsilon}
\bigg(
\bigcap_{k=1}^{\ell}\{X_{r_k}\in B_k\}
\bigg)
=
\mathbb P_x^{T,\beta,\varepsilon}
\big(
X_{r_1}\in B_1,\dots,X_{r_\ell}\in B_\ell
\big).
\]
Since for every $s \in\mathcal D_m$, we have $X_s^{(m)} =X_{\vartheta_m(s)}= X_s$, the above completes the proof.
\end{proof}

For every $m \ge 1$ and $x\in\overline E_\varepsilon$, define a probability measure (interpolated path-space law) $\mathbb P_{x}^{T,\beta,\varepsilon,(m)}$ on $\big(D([0,T];\overline E_\varepsilon),\,
\mathcal B(D([0,T];\overline E_\varepsilon))\big)$ by the pushforward
\begin{align} \label{InterpolatedPathSpaceLaw}
\mathbb P_{x}^{T,\beta,\varepsilon,(m)}
\,:=\,
\mathbb P_{x}^{T,\beta,\varepsilon}
\circ
\big(X^{(m)}_{\cdot}\big)^{-1}\,.
\end{align}
Equivalently, for every
$A\in\mathcal B(D([0,T];\overline E_\varepsilon))$, we have $\mathbb P_{x}^{T,\beta,\varepsilon,(m)}(A)
=
\mathbb P_{x}^{T,\beta,\varepsilon}
\big(
X^{(m)}_{\cdot}\in A
\big)$. Let $Y=\{Y_t\}_{t \in [0,T]}$ denote the canonical process on $D([0,T];\overline E_\varepsilon)$, that is,
\[
Y_t(p):=p(t),
\qquad p\in D([0,T];\overline E_\varepsilon),\quad t\in[0,T].
\]
\begin{proposition} \label{PropCanonicalProcessInterpolatedLaw}
Fix $T,\beta,\varepsilon>0$, $m\ge1$, and $x\in\overline E_\varepsilon$. Under
$\mathbb P_{x}^{T,\beta,\varepsilon,(m)}$,
the canonical process $Y=\{Y_t\}_{t\in[0,T]}$
has the same finite-dimensional distributions as the interpolated process
$X^{(m)}=\{X_t^{(m)}\}_{t\in[0,T]}$
under $\mathbb P_{x}^{T,\beta,\varepsilon}$.
\end{proposition}

\begin{proof}
Fix $n\in\mathbb N$, times $0\le t_1,\dots,t_n\le T$, and sets
$A_1,\dots,A_n\in\mathcal B(\overline E_\varepsilon)$. By definition of the
pushforward measure $\mathbb P_{x}^{T,\beta,\varepsilon,(m)} = \mathbb P_{x}^{T,\beta,\varepsilon} \circ \big(X^{(m)}_{\cdot}\big)^{-1}$,
we have the first equality below.
\begin{align}
\mathbb P_{x}^{T,\beta,\varepsilon,(m)}
\big(
Y_{t_1}\in A_1,\dots,Y_{t_n}\in A_n
\big)
\,=\,&
\mathbb P_{x}^{T,\beta,\varepsilon}
\Big(
X^{(m)}_{\cdot}\in
\big\{
p\in D([0,T];\overline E_\varepsilon):
Y_{t_1}(p)\in A_1,\dots,Y_{t_n}(p)\in A_n
\big\}
\Big) \nonumber \\
\,=\,&
\mathbb P_{x}^{T,\beta,\varepsilon}
\Big(
Y_{t_1}\big(X^{(m)}_{\cdot}\big)\in A_1,\dots,
Y_{t_n}\big(X^{(m)}_{\cdot}\big)\in A_n
\Big) \nonumber 
\end{align}
By definition $Y_t(p)=p(t)$ for every $t\in[0,T]$ and every
$p\in D([0,T];\overline E_\varepsilon)$. Hence, for each $j=1,\dots,n$ and each
$\omega\in\Omega_\varepsilon^{\mathbb D_T}$, we have the first equality below.
\[
Y_{t_j}\big(X^{(m)}_{\cdot}(\omega)\big)
=
\big(X^{(m)}_{\cdot}(\omega)\big)(t_j)
=
X_{t_j}^{(m)}(\omega)
\]
The second equality uses the definition of the path-valued map
$X^{(m)}_{\cdot}$. Thus, we have
\[
\mathbb P_{x}^{T,\beta,\varepsilon,(m)}
\big(
Y_{t_1}\in A_1,\dots,Y_{t_n}\in A_n
\big)
\,=\,
\mathbb P_{x}^{T,\beta,\varepsilon}
\big(
X_{t_1}^{(m)}\in A_1,\dots,X_{t_n}^{(m)}\in A_n
\big).
\]
The above completes the proof.
\end{proof}

The following corollary is an immediate consequence of Proposition~\ref{PropDyadicStepInterpolation}(iii) and Proposition~\ref{PropCanonicalProcessInterpolatedLaw}.

\begin{corollary} \label{CorAgreementOnDyadicGridInterpolatedLaw}
Fix $T,\beta,\varepsilon>0$, $m\ge1$, and $x\in\overline E_\varepsilon$. The finite-dimensional distributions of the canonical process $Y=\{Y_t\}_{t\in[0,T]}$ under
$\mathbb P_{x}^{T,\beta,\varepsilon,(m)}$ agree on $\mathcal D_m$ with those of
the dyadic skeleton $\{X_t\}_{t\in\mathbb D_T}$ under
$\mathbb P_{x}^{T,\beta,\varepsilon}$.
\end{corollary}

Observe that for every $m\ge1$, $x\in\overline E_\varepsilon$, $0\le t_1<\cdots<t_n\le T$, and
$A_1,\dots,A_n\in\mathcal B(\overline E_\varepsilon)$, since $X_{t_j}^{(m)}:=X_{\vartheta_m(t_j)}$ for all $j$, Proposition~\ref{PropCanonicalProcessInterpolatedLaw} yields
\begin{align}
\mathbb P_{x}^{T,\beta,\varepsilon,(m)}
\big(
Y_{t_1}\in A_1,\dots,Y_{t_n}\in A_n
\big)
\,=\,
\mathbb P_{x}^{T,\beta,\varepsilon}
\big(
X_{\vartheta_m(t_1)}\in A_1,\dots,X_{\vartheta_m(t_n)}\in A_n
\big)\,.
\end{align}

\begin{proposition}[Tightness of the one-time marginals]
\label{PropTightnessOneTimeMarginals}
Fix $T,\beta,\varepsilon>0$, $x\in\overline E_\varepsilon$, and $t\in[0,T]$.
Then the family $\{ \mathbb P_{x}^{T,\beta,\varepsilon,(m)}\circ Y_t^{-1}\}_{m\ge1}$ is tight on $\overline E_\varepsilon$.
\end{proposition}

The proof of the above proposition relies on the following two lemmas, which we establish first.

\begin{lemma}[Uniform one-time tail bound at infinity]
\label{LemUniformOneTimeTailInfinity}
Fix $T,\beta,\varepsilon>0$, $x\in\overline E_\varepsilon$, and $t\in[0,T]$.
Then
\[
\lim_{R\to\infty}\;
\sup_{m\ge1}
\mathbb P_{x}^{T,\beta,\varepsilon,(m)}
\bigl(|Y_t|>R\bigr)
=
0.
\]
\end{lemma}

\begin{proof}
Fix $m\ge1$. Since $\{|Y_t|>R\} =\{p\in D([0,T];\overline E_\varepsilon):|Y_t(p)|>R\}$ and the canonical process $Y$ is defined by $Y_t(p)=p(t)$, we have the first two equalities below.
\begin{align*}
\mathbb P_{x}^{T,\beta,\varepsilon,(m)}
\bigl(|Y_t|>R\bigr)
\,=\,&
\mathbb P_{x}^{T,\beta,\varepsilon,(m)}
\Bigl(
\bigl\{p\in D([0,T];\overline E_\varepsilon): |Y_t(p)|>R\bigr\}
\Bigr) \nonumber \\
\,=\,&
\mathbb P_{x}^{T,\beta,\varepsilon,(m)}
\Bigl(
\bigl\{p\in D([0,T];\overline E_\varepsilon): |p(t)|>R\bigr\}
\Bigr) \nonumber \\
\,=\,&
\mathbb P_{x}^{T,\beta,\varepsilon}
\Bigl(
X^{(m)}_\cdot \in
\bigl\{p\in D([0,T];\overline E_\varepsilon): |p(t)|>R\bigr\}
\Bigr) 
\,=\,
\mathbb P_{x}^{T,\beta,\varepsilon}
\bigl(
|X_t^{(m)}|>R
\bigr)
\end{align*}
The third equality uses that 
$ \mathbb P_{x}^{T,\beta,\varepsilon,(m)} = \mathbb P_{x}^{T,\beta,\varepsilon} \circ (X^{(m)}_\cdot)^{-1}$, so for any measurable $ A \subset D([0,T];\overline E_\varepsilon)$, we have $\mathbb P_{x}^{T,\beta,\varepsilon,(m)}(A) = \mathbb P_{x}^{T,\beta,\varepsilon}\big(X^{(m)}_\cdot\in A\big)$. The fourth equality uses that $X^{(m)}_\cdot(\omega)\in\{p:|p(t)|>R\}
$ iff $ |X_t^{(m)}(\omega)|>R$. Next using $X_t^{(m)} := X_{\vartheta_m(t)}$, we have the first equality below.
\begin{align*}
\mathbb P_{x}^{T,\beta,\varepsilon,(m)}
\bigl(|Y_t|>R\bigr)
\,=\,&
\mathbb P_{x}^{T,\beta,\varepsilon}
\bigl(|X_{\vartheta_m(t)}|>R\bigr)
\nonumber \\
\,=\,&
\mathbb P_{x}^{T,\beta,\varepsilon}
\Bigl(
X_{\vartheta_m(t)}
\in
\{y\in\overline E_\varepsilon:|y|>R\}
\Bigr)
\,=\,
\mathsf K^{T,\beta,\varepsilon}_{[0,\vartheta_m(t)]}
\Bigl(
x,\,
\{y\in\overline E_\varepsilon:|y|>R\}
\Bigr)
\end{align*}
Since $\vartheta_m(t) := \max(\mathcal D_m\cap[0,t])$, we have
$\vartheta_m(t)\in\mathcal D_m \subset \mathbb D_T $. Also, under $\mathbb P_{x}^{T,\beta,\varepsilon}$ the dyadic–skeleton process
$\{X_s\}_{s\in\mathbb D_T}$ has transition kernel
$\mathsf K^{T,\beta,\varepsilon}_{[s,t]}$ by Proposition~\ref{PropDyadicSkeletonProcess}. Thus, the one--time distribution of
$X_{\vartheta_m(t)}$ is given by  $\mathsf K^{T,\beta,\varepsilon}_{[0,\vartheta_m(t)]}(x,\cdot)$, which implies the last equality above.

Next, since $\vartheta_m(t)\in[0,t]\cap\mathcal D_m \subset[0,t] \cap \mathbb D_T$ for every $m\ge1$, it follows that
\begin{align*}
\sup_{m\ge1}
\mathbb P_{x}^{T,\beta,\varepsilon,(m)}
\bigl(|Y_t|>R\bigr)
=
\sup_{m\ge1}
\mathsf K^{T,\beta,\varepsilon}_{[0,\vartheta_m(t)]}
\Bigl(
x,\,
\{y\in\overline E_\varepsilon:|y|>R\}
\Bigr)
\le
\sup_{s\in[0,t]\cap\mathbb D_T}
\mathsf K^{T,\beta,\varepsilon}_{[0,s]}
\Bigl(
x,\,
\{y\in\overline E_\varepsilon:|y|>R\}
\Bigr).
\end{align*}
Next we distinguish the cases $x=\Delta$ and $x\in E_\varepsilon$.
If $x=\Delta$, then by the absorbing property of the cemetery state, $\mathsf K^{T,\beta,\varepsilon}_{[0,s]}(\Delta,\cdot)=\delta_\Delta
$ for every $s \in\mathbb D_T$, and hence
\begin{align*}
\mathsf K^{T,\beta,\varepsilon}_{[0,s]}
\Bigl(
\Delta,\,
\{y\in\overline E_\varepsilon:|y|>R\}
\Bigr)
=
\delta_\Delta\Bigl(
\{y\in\overline E_\varepsilon:|y|>R\}
\Bigr)
=
0
\end{align*}
for every $s\in[0,t]\cap\mathbb D_T$. Therefore,
\begin{align*}
\sup_{m\ge1}
\mathbb P_{\Delta}^{T,\beta,\varepsilon,(m)}
\bigl(|Y_t|>R\bigr)=0.
\end{align*}

It remains to consider the case $x\in E_\varepsilon$. In this case,
since $\Delta\notin\{y\in\overline E_\varepsilon:|y|>R\}$, we have
\begin{align*}
\mathsf K^{T,\beta,\varepsilon}_{[0,s]}
\Bigl(
x,\,
\{y\in\overline E_\varepsilon:|y|>R\}
\Bigr)
=
\mathsf K^{T,\beta,\varepsilon}_{[0,s]}
\Bigl(
x,\,
\{y\in E_\varepsilon:|y|>R\}
\Bigr)
\end{align*}
for every $s\in[0,t]\cap\mathbb D_T$. Hence, for $x\in E_\varepsilon$,
\begin{align*}
\sup_{m\ge1}
\mathbb P_{x}^{T,\beta,\varepsilon,(m)}
\bigl(|Y_t|>R\bigr)
\le
\sup_{s\in[0,t]\cap\mathbb D_T}
\mathsf K^{T,\beta,\varepsilon}_{[0,s]}
\Bigl(
x,\,
\{y\in E_\varepsilon:|y|>R\}
\Bigr).
\end{align*}
Therefore, to prove the lemma, it therefore remains to show that
\begin{align*}
\lim_{R\to\infty}
\sup_{s\in[0,t]\cap\mathbb D_T}
\mathsf K^{T,\beta,\varepsilon}_{[0,s]}
\Bigl(
x,\,
\{y\in E_\varepsilon:|y|>R\}
\Bigr)
=0.
\end{align*}
We now separate the case $s=0$ from the case $s>0$. For $s=0$, since the initial distribution is $\delta_x$, we have
\begin{align*}
\mathsf K^{T,\beta,\varepsilon}_{[0,0]}
\Bigl(
x,\,
\{y\in E_\varepsilon:|y|>R\}
\Bigr)
=
\delta_x\bigl(\{y\in E_\varepsilon:|y|>R\}\bigr)
=
\mathbf 1_{\{|x|>R\}}.
\end{align*}
Hence,
\begin{align*}
\sup_{s\in[0,t]\cap\mathbb D_T}
\mathsf K^{T,\beta,\varepsilon}_{[0,s]}
\Bigl(
x,\,
\{y\in E_\varepsilon:|y|>R\}
\Bigr)
=
\max\bigg\{
\mathbf 1_{\{|x|>R\}},
\;
\sup_{s\in(0,t]\cap\mathbb D_T}
\mathsf K^{T,\beta,\varepsilon}_{[0,s]}
\Bigl(
x,\,
\{y\in E_\varepsilon:|y|>R\}
\Bigr)
\bigg\}.
\end{align*}
Since $x$ is fixed, we have $\mathbf 1_{\{|x|>R\}}\longrightarrow 0$ as $R\to\infty$. Next, fix $s\in(0,t]\cap\mathbb D_T$. Since $x\in E_\varepsilon$ is fixed and $A:=\{y\in E_\varepsilon:\ |y|>R\}\subset E_\varepsilon$ does not contain $\Delta$, it follows from from~\eqref{EqDefMeshLimitFull} that
\begin{align*}
\mathsf K^{T,\beta,\varepsilon}_{[0,s]}
\Bigl(
x,\,
\{y\in E_\varepsilon:\ |y|>R\}
\Bigr)
&=
\widehat{\mathsf K}^{T,\beta,\varepsilon}_{[0,s]}
\Bigl(
x,\,
\{y\in E_\varepsilon:\ |y|>R\}
\Bigr).
\end{align*}
Therefore,
\begin{align*}
\sup_{s\in(0,t]\cap\mathbb D_T}
\mathsf K^{T,\beta,\varepsilon}_{[0,s]}
\Bigl(
x,\,
\{y\in E_\varepsilon:\ |y|>R\}
\Bigr)
&=
\sup_{s\in(0,t]\cap\mathbb D_T}
\widehat{\mathsf K}^{T,\beta,\varepsilon}_{[0,s]}
\Bigl(
x,\,
\{y\in E_\varepsilon:\ |y|>R\}
\Bigr).
\end{align*}
Hence it remains to show that
\begin{align}\label{WTS}
\lim_{R\to\infty}
\sup_{s\in(0,t]\cap\mathbb D_T}
\widehat{\mathsf K}^{T,\beta,\varepsilon}_{[0,s]}
\Bigl(
x,\,
\{y\in E_\varepsilon:\ |y|>R\}
\Bigr)
=0.
\end{align}
Fix $R>\varepsilon$. By~\eqref{EqDefGlobalDyadicLimitKernel_Inf},
for each $s\in(0,t]\cap\mathbb D_T$ we have
\begin{align*}
\widehat{\mathsf K}^{T,\beta,\varepsilon}_{[0,s]}
\Bigl(
x,\,
\{y\in E_\varepsilon:\ |y|>R\}
\Bigr)
&=
\inf_{m\ge1}
\widehat{\mathsf K}^{T,\beta,\varepsilon,\pi_m[0,s]}_{[0,s]}
\Bigl(
x,\,
\{y\in E_\varepsilon:\ |y|>R\}
\Bigr) \\
&\le
\widehat{\mathsf K}^{T,\beta,\varepsilon,\pi_1[0,s]}_{[0,s]}
\Bigl(
x,\,
\{y\in E_\varepsilon:\ |y|>R\}
\Bigr)\,,
\end{align*}
where $\pi_m[s,t]$ denotes the partition defined in~\eqref{DefGlobalDyadics}. The inequality follows from monotonicity under refinement in Lemma~\ref{LemMonotoneOnEeps}(i). Therefore,
\begin{align*}
\sup_{s\in(0,t]\cap\mathbb D_T}
\widehat{\mathsf K}^{T,\beta,\varepsilon}_{[0,s]}
\Bigl(
x,\,
\{y\in E_\varepsilon:\ |y|>R\}
\Bigr)
\le
\sup_{s\in(0,t]\cap\mathbb D_T}
\widehat{\mathsf K}^{T,\beta,\varepsilon,\pi_1[0,s]}_{[0,s]}
\Bigl(
x,\,
\{y\in E_\varepsilon:\ |y|>R\}
\Bigr).
\end{align*}

\noindent \textit{Case 1.} For $0<s\le \frac{T}{2}$, recall from~\eqref{IndPartPi1} that $\pi_1[0,s]=\{0,s\}$. Hence, using~\eqref{FirstTrans3dSub} we obtain the first equality below.
\begin{align*}
\widehat{\mathsf K}^{T,\beta,\varepsilon,\pi_1[0,s]}_{[0,s]}
\Bigl(
x,\,
\{y\in E_\varepsilon:\ |y|>R\}
\Bigr)
\,=\,
\int_{E_\varepsilon}
\mathbf 1_{\{|y|>R\}}\,
p^{\,T,\beta}_{0,s}(x,y)\,dy 
\,\le\,
\int_{E_\varepsilon}
\mathbf 1_{\{|y|>R\}}\,
\bigl(1+Q_{T-s}^\beta(y)\bigr)\,
P_s^\beta(x,y)\,dy  \nonumber 
\end{align*}
The inequality uses~\eqref{FirstTrans3d} together with $1+Q_t^\beta(x)\ge 1$. Next, fix $L:=\beta\sqrt T$. Then $\beta\sqrt{T-s}\le L$ since $0<s\le \frac T2$, so by~\eqref{QUpperBoundCL},
\begin{align*}
Q_{T-s}^\beta(y)
\le
C_L\Bigl(1+\frac{\sqrt{T-s}}{|y|}\Bigr)e^{-4\pi\beta|y|}
\le
C_L\Bigl(1+\frac{\sqrt T}{|y|}\Bigr)e^{-4\pi\beta|y|},
\qquad y\in \R^3\setminus\{0\}.
\end{align*}
Substituting this into the previous bound, we get the first inequality below.
\begin{align*}
\int_{E_\varepsilon}
\mathbf 1_{\{|y|>R\}}\,
p^{\,T,\beta}_{0,s}(x,y)\,dy
&\le
\int_{E_\varepsilon}
\mathbf 1_{\{|y|>R\}}\,
\Bigl[
1+
C_L\Bigl(1+\frac{\sqrt T}{|y|}\Bigr)e^{-4\pi\beta|y|}
\Bigr]
P_s^\beta(x,y)\,dy 
\le
I_1(R,s)+I_2(R,s)\,.
\end{align*}
The last inequality holds by applying the upper bound from~\eqref{PtBetaBounds} to
$P_s^\beta(x,y)$, where we bound the two integrals $I_1(R,s)$ and $I_2(R,s)$ as follows. \vspace{.2cm}

\noindent \textit{First integral.} Using the notation in~\eqref{DefFreeHeat3d}, we can rewrite $P(s;x,y)=P_s(x,y)=P_s(|x-y|)$. Therefore the first integral becomes
\begin{align*}
I_1(R,s)
&\,:=\,
\int_{E_\varepsilon}
\mathbf 1_{\{|y|>R\}}\,
\Bigl[
1+
C_L\Bigl(1+\frac{\sqrt T}{|y|}\Bigr)e^{-4\pi\beta|y|}
\Bigr]
P_s(x,y)\,dy \\
&=
\int_{E_\varepsilon\cap\{|y|>R\}}
\Bigl[
1+
C_L\Bigl(1+\frac{\sqrt T}{|y|}\Bigr)e^{-4\pi\beta|y|}
\Bigr]
P_s(|x-y|)\,dy \\
&\le
\Bigl(1+C_L+C_L\frac{\sqrt T}{R}\Bigr)
\int_{E_\varepsilon\cap\{|y|>R\}}
P_s(|x-y|)\,dy 
\le
\Bigl(1+C_L+C_L\frac{\sqrt T}{R}\Bigr)
\int_{\{|y|>R\}}
P_s(|x-y|)\,dy.
\end{align*}
The first inequality holds since $|y|>R$ on the domain of integration, we have $\frac{1}{|y|}\le \frac{1}{R} $ and $ e^{-4\pi\beta|y|}\le 1$. Moreover, since $E_\varepsilon\cap\{|y|>R\}\subset\{|y|>R\}$, we have the last inequality. Next, applying the change of variables $z=y-x,$ to obtain the first equality below.
\begin{align*}
\int_{\{|y|>R\}} P_s(|x-y|)\,dy
=
\int_{\{|z+x|>R\}} P_s(|z|)\,dz 
\le
\int_{\{|z|>R-|x|\}} P_s(|z|)\,dz
\le
\int_{\{|z|>R-|x|\}} P_{T/2}(|z|)\,dz
\end{align*}
The first inequality holds since if $|z+x|>R$, then by the triangle inequality $|z| \ge |\,|z+x|-|x|\,| > R-|x|$. Hence, $\{|z+x|>R\}\subset \{|z|>R-|x|\}$. The last inequality uses that $0<s\le \frac{T}{2}$ and the Gaussian tail is increasing in the time parameter. Consequently,
\begin{align*}
I_1(R,s)
\le
\Bigl(1+C_L+C_L\frac{\sqrt T}{R}\Bigr)
\int_{\{|z|>R-|x|\}} P_{T/2}(|z|)\,dz.
\end{align*}
The right-hand side is independent of $s\in(0,T/2]$ and tends to zero as
$R\to\infty$. Therefore,
\begin{align*}
\lim_{R\to\infty}
\sup_{0<s\le T/2}
I_1(R,s)
=
0.
\end{align*}

\noindent\textit{Second integral.}
We now decompose $I_2(R,s)$ by expanding the bracket inside the integral.
\begin{align*}
I_2(R,s)
&\,:=\,
C_L\, s^{-1/2}\,\frac{e^{-\frac{|x|^2}{4s}}}{|x|}
\int_{E_\varepsilon}
\mathbf 1_{\{|y|>R\}}\,
\Bigl[
1+
C_L\Bigl(1+\frac{\sqrt T}{|y|}\Bigr)e^{-4\pi\beta|y|}
\Bigr]
\frac{e^{-\frac{|y|^2}{4s}}}{|y|}\,dy \\
&=
C_L\, s^{-1/2}\,\frac{e^{-\frac{|x|^2}{4s}}}{|x|}
\int_{E_\varepsilon}
\mathbf 1_{\{|y|>R\}}\,
\frac{e^{-\frac{|y|^2}{4s}}}{|y|}\,dy 
+
C_L^2\, s^{-1/2}\,\frac{e^{-\frac{|x|^2}{4s}}}{|x|}
\int_{E_\varepsilon}
\mathbf 1_{\{|y|>R\}}\,
\Bigl(1+\frac{\sqrt T}{|y|}\Bigr)e^{-4\pi\beta|y|}
\frac{e^{-\frac{|y|^2}{4s}}}{|y|}\,dy 
\end{align*}
For notational convenience, we denote the two terms on the right-hand side above by $I^{(1)}_2(R,s)$ and $I^{(2)}_2(R,s)$, respectively. Then,
\begin{align*}
I^{(1)}_2(R,s)
&\,:=\,
C_L\, s^{-1/2}\,\frac{e^{-\frac{|x|^2}{4s}}}{|x|}
\int_{E_\varepsilon}
\mathbf 1_{\{|y|>R\}}\,
\frac{e^{-\frac{|y|^2}{4s}}}{|y|}\,dy \\
&=
C_L\, s^{-1/2}\,\frac{e^{-\frac{|x|^2}{4s}}}{|x|}
\int_{E_\varepsilon\cap\{|y|>R\}}
\frac{e^{-\frac{|y|^2}{4s}}}{|y|}\,dy \\
&=
C_L\, s^{-1/2}\,\frac{e^{-\frac{|x|^2}{4s}}}{|x|}
\int_{\{|y|>R\}}
\frac{e^{-\frac{|y|^2}{4s}}}{|y|}\,dy 
=
4\pi C_L\, s^{-1/2}\,\frac{e^{-\frac{|x|^2}{4s}}}{|x|}
\int_R^\infty
r\,e^{-\frac{r^2}{4s}}\,dr
\end{align*}
The third equality holds since $R>\varepsilon$, we have $E_\varepsilon\cap\{|y|>R\}=\{|y|>R\}$ and the last equality uses passing to spherical coordinates. Next, using the change of variables
$u=\frac{r^2}{4s}$, so that $du=\frac{r}{2s}dr$, we obtain $\int_R^\infty r e^{-\frac{r^2}{4s}} dr = 2s\int_{\frac{R^2}{4s}}^\infty e^{-u} du = 2s e^{-\frac{R^2}{4s}}$. Substituting this into the above display yields the first equality below.
\begin{align*}
I^{(1)}_2(R,s)
=
8\pi C_L\, s^{1/2}\,
\frac{e^{-\frac{|x|^2}{4s}}}{|x|}
\,e^{-\frac{R^2}{4s}}
=
8\pi C_L\, \frac{\sqrt{s}}{|x|}
e^{
-\frac{|x|^2+R^2}{4s}}
\le
8\pi C_L\, \frac{\sqrt{T/2}}{|x|}
e^{
-\frac{|x|^2+R^2}{2T}}
\end{align*}
Since $0<s\le \frac{T}{2}$, the inequality uses that $s\le T/2$ and $s\mapsto e^{-a/s}$ is increasing for $a>0$. Consequently, the right--hand side is independent of $s\in(0,T/2]$ and since $x$ is fixed, the right--hand side tends to $0$ as $R\to\infty$,
\begin{align}\label{WTS1}
\lim_{R\to\infty}
\sup_{0<s\le T/2}
I^{(1)}_2(R,s)
=
0.
\end{align}

For $I^{(2)}_2(R,s)$, again using $E_\varepsilon\cap\{|y|>R\}=\{|y|>R\}$, we have
\begin{align*}
I^{(2)}_2(R,s)
&=
C_L^2\, s^{-1/2}\,\frac{e^{-\frac{|x|^2}{4s}}}{|x|}
\int_{\{|y|>R\}}
\Bigl(1+\frac{\sqrt T}{|y|}\Bigr)
e^{-4\pi\beta|y|}
\frac{e^{-\frac{|y|^2}{4s}}}{|y|}\,dy \\
&=
C_L^2\, s^{-1/2}\,\frac{e^{-\frac{|x|^2}{4s}}}{|x|}
\int_{\{|y|>R\}}
e^{-4\pi\beta|y|}
\frac{e^{-\frac{|y|^2}{4s}}}{|y|}\,dy 
+
C_L^2\sqrt T\, s^{-1/2}\,\frac{e^{-\frac{|x|^2}{4s}}}{|x|}
\int_{\{|y|>R\}}
e^{-4\pi\beta|y|}
\frac{e^{-\frac{|y|^2}{4s}}}{|y|^2}\,dy 
\end{align*}
Since $e^{-4\pi\beta|y|}\le 1$ for all $y\in\R^3$, the first term above is bounded by $C_L^2 I^{(1)}_2(R,s)$, which has already been treated in~\eqref{WTS1}. Hence it remains to consider the second term. Since $|y|>R$ on the domain of integration, we have $\frac{1}{|y|^2}\le \frac{1}{R|y|}$. Using also that $e^{-4\pi\beta|y|}\le 1$, we obtain
\begin{align*}
C_L^2\sqrt T\, s^{-1/2}\,\frac{e^{-\frac{|x|^2}{4s}}}{|x|}
\int_{\{|y|>R\}}
e^{-4\pi\beta|y|}
\frac{e^{-\frac{|y|^2}{4s}}}{|y|^2}\,dy 
\le
\frac{C_L^2\sqrt T}{R}\,
s^{-1/2}\,\frac{e^{-\frac{|x|^2}{4s}}}{|x|}
\int_{\{|y|>R\}}
\frac{e^{-\frac{|y|^2}{4s}}}{|y|}\,dy 
=
\frac{C_L^2\sqrt T}{R\,C_L}\, I^{(1)}_2(R,s).
\end{align*}
Therefore both terms in the decomposition of $I^{(2)}_2(R,s)$ are controlled by
a constant multiple of $I^{(1)}_2(R,s)$. Hence, for $0<s\le T/2$,
\begin{align*}
I^{(2)}_2(R,s)
\le
C_L^2 I^{(1)}_2(R,s)
+
\frac{C_L^2\sqrt T}{R\,C_L}\, I^{(1)}_2(R,s) 
=
\left(
C_L^2+\frac{C_L\sqrt T}{R}
\right)
I^{(1)}_2(R,s).
\end{align*}
It follows from~\eqref{WTS1} that
\begin{align*}
\lim_{R\to\infty}\sup_{0<s\le T/2} I_2(R,s)=0.
\end{align*}
Thus,~\eqref{WTS} holds for the case $0\le s <\frac{T}{2}$. \vspace{.2cm}

\noindent \textit{Case 2.} For $\frac{T}{2}<s\le T$, recall from~\eqref{IndPartPi1} that $\pi_1[0,s]=\big\{0,\frac{T}{2},s\big\}$. Hence, using~\eqref{FirstTrans3dSub} we obtain
\begin{align*}
\widehat{\mathsf K}^{T,\beta,\varepsilon,\pi_1[0,s]}_{[0,s]}
\Bigl(
x,\,
\{y\in E_\varepsilon:\ |y|>R\}
\Bigr) 
=
\int_{E_\varepsilon}\int_{E_\varepsilon}
\mathbf 1_{\{|y_2|>R\}}\,
p^{\,T,\beta}_{0,T/2}(x,y_1)\,
p^{\,T,\beta}_{T/2,s}(y_1,y_2)\,
dy_1\,dy_2.
\end{align*}
Using~\eqref{FirstTrans3d} the right hand side above can be written as
\begin{align*}
\int_{E_\varepsilon}\int_{E_\varepsilon}
\mathbf 1_{\{|y_2|>R\}}\,
&\frac{1+Q^{\beta}_{T/2}(y_1)}{1+Q^{\beta}_{T}(x)}
\,P^{\beta}_{T/2}(x,y_1)\,
\frac{1+Q^{\beta}_{T-s}(y_2)}{1+Q^{\beta}_{T/2}(y_1)}
\,P^{\beta}_{\,s-T/2}(y_1,y_2)\,
dy_1\,dy_2  \nonumber \\
&=
\frac{1}{1+Q^{\beta}_{T}(x)}
\int_{E_\varepsilon}\int_{E_\varepsilon}
\mathbf 1_{\{|y_2|>R\}}\,
\bigl(1+Q^{\beta}_{T-s}(y_2)\bigr)\,
P^{\beta}_{T/2}(x,y_1)\,
P^{\beta}_{\,s-T/2}(y_1,y_2)\,
dy_1\,dy_2 \nonumber
\end{align*}
Thus, using $1+Q_t^\beta(x)\ge 1$, we obtain the first inequality below.
\begin{align*}
\widehat{\mathsf K}^{T,\beta,\varepsilon,\pi_1[0,s]}_{[0,s]}
\Bigl(
x,\,
\{y\in E_\varepsilon:\ |y|>R\}
\Bigr) 
\le&
\int_{E_\varepsilon}\int_{E_\varepsilon}
\mathbf 1_{\{|y_2|>R\}}\,
\bigl(1+Q^{\beta}_{T-s}(y_2)\bigr)\,
P^{\beta}_{T/2}(x,y_1)\,
P^{\beta}_{\,s-T/2}(y_1,y_2)\,
dy_1\,dy_2 \nonumber \\
=&
\int_{E_\varepsilon}
\mathbf 1_{\{|y_2|>R\}}\,
\bigl(1+Q^{\beta}_{T-s}(y_2)\bigr)
\left(
\int_{E_\varepsilon}
P^{\beta}_{T/2}(x,y_1)\,
P^{\beta}_{\,s-T/2}(y_1,y_2)\,
dy_1
\right)
dy_2
\end{align*}
The first equality uses Tonelli's theorem to interchange
the order of integration (since the integrand is nonnegative). Next, since the integrand in the inner integral is nonnegative, enlarging the
domain of integration from $E_\varepsilon$ to $\R^3$ gives
\begin{align*}
\int_{E_\varepsilon}
P^{\beta}_{T/2}(x,y_1)\,
P^{\beta}_{\,s-T/2}(y_1,y_2)\,
dy_1
\le
\int_{\R^3}
P^{\beta}_{T/2}(x,y_1)\,
P^{\beta}_{\,s-T/2}(y_1,y_2)\,
dy_1 
=
P_s^\beta(x,y_2),
\end{align*}
where the last equality follows from the semigroup property~\eqref{EqSemigroupKernelReliable} of the
kernel $P^\beta$. Substituting this into the previous bound and renaming $y_2$ as $y$, we obtain the upper bound
\begin{align*}
\widehat{\mathsf K}^{T,\beta,\varepsilon,\pi_1[0,s]}_{[0,s]}
\Bigl(
x,\,
\{y\in E_\varepsilon:\ |y|>R\}
\Bigr) 
\le
\int_{E_\varepsilon}
\mathbf 1_{\{|y|>R\}}\,
\bigl(1+Q^{\beta}_{T-s}(y)\bigr)\,
P_s^\beta(x,y)\,
dy 
=
J_1(R,s)+J_2(R,s),
\end{align*}
for $\frac{T}{2}<s\le T$, where $J_1(R,s)$ and $J_2(R,s)$ denote the two terms obtained by using~\eqref{PtBetaBounds} and splitting the integral on the right-hand side exactly as in Case~1. From this point onward, the estimates proceed exactly as in Case~1, with
$J_1(R,s)$ and $J_2(R,s)$ in place of $I_1(R,s)$ and $I_2(R,s)$. The only
difference is that here $\frac{T}{2}<s\le T$, so whenever monotonicity in the
time parameter is used, we bound by the value at time $T$ instead of $T/2$.
Consequently,
\[
\lim_{R\to\infty}\sup_{T/2<s\le T}J_1(R,s)=0,
\qquad
\lim_{R\to\infty}\sup_{T/2<s\le T}J_2(R,s)=0.
\]
Hence,~\eqref{WTS} also holds in the regime $\frac{T}{2}<s\le T$, and the proof is complete.
\end{proof}

\begin{lemma}[Uniform one-time boundary-approach bound]
\label{LemUniformOneTimeBoundaryApproach}
Fix $T,\beta,\varepsilon>0$, $x\in\overline E_\varepsilon$, and $t\in[0,T]$.
Then
\[
\lim_{\delta\downarrow0}\;
\sup_{m\ge1}
\mathbb P_{x}^{T,\beta,\varepsilon,(m)}
\bigl(\varepsilon<|Y_t|\le \varepsilon+\delta\bigr)
=
0.
\]
\end{lemma}

\begin{proof} 
Fix $m \ge 1$ and $\delta\in(0,1)$. Set $A_\delta := \big\{ y\in\overline E_\varepsilon: \varepsilon<|y|\le \varepsilon+\delta \big\}$. Recall that under
$\mathbb P_{x}^{T,\beta,\varepsilon,(m)}$ the canonical process
$Y=\{Y_t\}_{t\in[0,T]}$ has the same finite--dimensional distributions as the
interpolated process $X^{(m)}=\{X_t^{(m)}\}_{t\in[0,T]}$ under
$\mathbb P_{x}^{T,\beta,\varepsilon}$ by Proposition~\ref{PropCanonicalProcessInterpolatedLaw}.
In particular, taking the one--dimensional distribution at time $t$, we have $\mathbb P_{x}^{T,\beta,\varepsilon,(m)}(Y_t\in A) = \mathbb P_{x}^{T,\beta,\varepsilon}(X_t^{(m)}\in A) $ for every $A\in\mathcal B(\overline E_\varepsilon)$. This implies the second equality below.
\begin{align*}
\mathbb P_{x}^{T,\beta,\varepsilon,(m)}
\bigl(\varepsilon<|Y_t|\le \varepsilon+\delta\bigr)
=
\mathbb P_{x}^{T,\beta,\varepsilon,(m)}
\bigl(
Y_t\in A_\delta
\bigr) =
\mathbb P_{x}^{T,\beta,\varepsilon}
\bigl(
X_{\vartheta_m(t)}\in A_\delta
\bigr) =
\mathsf K^{T,\beta,\varepsilon}_{[0,\vartheta_m(t)]}
\bigl(
x,\,
A_\delta
\bigr).
\end{align*}
The last equality follows from Proposition~\ref{PropDyadicSkeletonProcess}. Indeed, since $X_0=x$ almost surely and
$\vartheta_m(t)\in\mathbb D_T$, the distribution of $X_{\vartheta_m(t)}$ under
$\mathbb P_{x}^{T,\beta,\varepsilon}$ is given by
$\mathsf K^{T,\beta,\varepsilon}_{[0,\vartheta_m(t)]}(x,\cdot)$. Therefore,
\begin{align*}
\sup_{m\ge1}
\mathbb P_{x}^{T,\beta,\varepsilon,(m)}
\bigl(\varepsilon<|Y_t|\le \varepsilon+\delta\bigr)
=
\sup_{m\ge1}
\mathsf K^{T,\beta,\varepsilon}_{[0,\vartheta_m(t)]}
\bigl(
x,\,
A_\delta
\bigr) 
\le
\sup_{s\in[0,t]\cap\mathbb D_T}
\mathsf K^{T,\beta,\varepsilon}_{[0,s]}
\bigl(
x,\,
A_\delta
\bigr).
\end{align*}
The inequality uses that $\vartheta_m(t)\in[0,t]\cap\mathcal D_m\subset[0,t]\cap\mathbb D_T$
for every $m\ge1$. Thus it remains to show that
\begin{align*}
\lim_{\delta\downarrow0}\;
\sup_{s\in[0,t]\cap\mathbb D_T}
\mathsf K^{T,\beta,\varepsilon}_{[0,s]}
\bigl(
x,\,
A_\delta
\bigr)
=0.
\end{align*}
We now separate the case $x=\Delta$ from the case $x\in E_\varepsilon$. The case $x=\Delta$ is trivial since by the definition of the cemetery state extension, $\mathsf K^{T,\beta,\varepsilon}_{[0,s]}(\Delta,\cdot)=\delta_\Delta(\cdot)$
for $ s\in[0,T]$. Since $\Delta\notin A_\delta$, it follows that
\begin{align*}
\mathsf K^{T,\beta,\varepsilon}_{[0,s]}(\Delta,A_\delta)
=
\delta_\Delta(A_\delta)
=
0,
\qquad s\in[0,T],
\end{align*}
and therefore the desired limit is immediate in this case. It remains to consider the case $x\in E_\varepsilon$. We now separate the case
$s=0$ from the case $s>0$. For $s=0$, since the initial distribution is $\delta_x$, we have $\mathsf K^{T,\beta,\varepsilon}_{[0,0]}(x,A_\delta)
=
\delta_x(A_\delta)
=
\mathbf 1_{\{\varepsilon<|x|\le \varepsilon+\delta\}}$. Hence,
\begin{align*}
\sup_{s\in[0,t]\cap\mathbb D_T}
\mathsf K^{T,\beta,\varepsilon}_{[0,s]}(x,A_\delta)
=
\max\Big\{
\mathbf 1_{\{\varepsilon<|x|\le \varepsilon+\delta\}},
\;
\sup_{s\in(0,t]\cap\mathbb D_T}
\mathsf K^{T,\beta,\varepsilon}_{[0,s]}(x,A_\delta)
\Big\}.
\end{align*}
Since $x\in E_\varepsilon$ is fixed, we have $\mathbf 1_{\{\varepsilon<|x|\le \varepsilon+\delta\}} \to 0$ as $\delta\downarrow0$. Next, fix $s\in(0,t]\cap\mathbb D_T$. Since $A_\delta\subset E_\varepsilon $ and $ \Delta\notin A_\delta$, it follows from~\eqref{EqDefMeshLimitFull} that $\mathsf K^{T,\beta,\varepsilon}_{[0,s]}(x,A_\delta) = \widehat{\mathsf K}^{T,\beta,\varepsilon}_{[0,s]}(x,A_\delta)$. Therefore,
\begin{align*}
\sup_{s\in(0,t]\cap\mathbb D_T}
\mathsf K^{T,\beta,\varepsilon}_{[0,s]}(x,A_\delta)
=
\sup_{s\in(0,t]\cap\mathbb D_T}
\widehat{\mathsf K}^{T,\beta,\varepsilon}_{[0,s]}(x,A_\delta).
\end{align*}
Hence it remains to show that
\begin{align*}
\lim_{\delta\downarrow0}\;
\sup_{s\in(0,t]\cap\mathbb D_T}
\widehat{\mathsf K}^{T,\beta,\varepsilon}_{[0,s]}(x,A_\delta)
=0.
\end{align*}
By~\eqref{EqDefGlobalDyadicLimitKernel_Inf}, for each
$s\in(0,t]\cap\mathbb D_T$ we have
\begin{align*}
\widehat{\mathsf K}^{T,\beta,\varepsilon}_{[0,s]}(x,A_\delta)
=
\inf_{m\ge1}
\widehat{\mathsf K}^{T,\beta,\varepsilon,\pi_m[0,s]}_{[0,s]}(x,A_\delta) 
\le
\widehat{\mathsf K}^{T,\beta,\varepsilon,\pi_1[0,s]}_{[0,s]}(x,A_\delta),
\end{align*}
where $\pi_m[0,s]$ denotes the partition defined in~\eqref{DefGlobalDyadics}.
The inequality follows from monotonicity under refinement in
Lemma~\ref{LemMonotoneOnEeps}. Therefore,
\begin{align*}
\sup_{s\in(0,t]\cap\mathbb D_T}
\widehat{\mathsf K}^{T,\beta,\varepsilon}_{[0,s]}(x,A_\delta)
\le
\sup_{s\in(0,t]\cap\mathbb D_T}
\widehat{\mathsf K}^{T,\beta,\varepsilon,\pi_1[0,s]}_{[0,s]}(x,A_\delta).
\end{align*}

We now treat separately the two cases corresponding to the form of the coarsest partition $\pi_1[0,s]$.

\noindent\textit{Case 1.} Let $0<s\le T/2$. Then
$\pi_1[0,s]=\{0,s\}$ by~\eqref{IndPartPi1}. Hence, using~\eqref{FirstTrans3dSub},
\begin{align*}
\widehat{\mathsf K}^{T,\beta,\varepsilon,\pi_1[0,s]}_{[0,s]}(x,A_\delta)
=
\int_{E_\varepsilon}
\mathbf 1_{\{\varepsilon<|y|\le \varepsilon+\delta\}}
p^{\,T,\beta}_{0,s}(x,y)\,dy 
\le
\int_{E_\varepsilon}
\mathbf 1_{\{\varepsilon<|y|\le \varepsilon+\delta\}}
\bigl(1+Q_{T-s}^\beta(y)\bigr)P_s^\beta(x,y)\,dy.
\end{align*}
The inequality follows from~\eqref{FirstTrans3d} and that $1+Q_{T}^\beta(x) \ge 1$. Set $L:=\beta\sqrt T$. Using~\eqref{QUpperBoundCL} we obtain
\begin{align*}
\widehat{\mathsf K}^{T,\beta,\varepsilon,\pi_1[0,s]}_{[0,s]}(x,A_\delta)
\le
\int_{E_\varepsilon}
\mathbf 1_{\{\varepsilon<|y|\le \varepsilon+\delta\}}
\Bigl[
1+C_L\Bigl(1+\frac{\sqrt T}{|y|}\Bigr)e^{-4\pi\beta|y|}
\Bigr]
P_s^\beta(x,y)\,dy
\le
K_1(\delta,s)+K_2(\delta,s),
\end{align*}
where the final inequality uses the upper bound from~\eqref{PtBetaBounds} to $P_s^\beta(x,y)$ and the integrals $K_1(\delta,s)$, $K_2(\delta,s)$ are defined and estimated below separately. \vspace{.2cm}

\noindent\textit{First integral.} Using the notation in~\eqref{DefFreeHeat3d}, we
rewrite $P_s(x,y)=P_s(|x-y|)$. Also, since $\varepsilon<|y|\le \varepsilon+\delta$ on the domain of integration, we have $\frac1{|y|}\le \frac1\varepsilon$ and $e^{-4\pi\beta|y|}\le 1$. Therefore,
\begin{align*}
K_1(\delta,s)
\le
\Bigl(1+C_L+C_L\frac{\sqrt T}{\varepsilon}\Bigr)
\int_{\{\varepsilon<|y|\le \varepsilon+\delta\}}
P_s(|x-y|)\,dy
=
\Bigl(1+C_L+C_L\frac{\sqrt T}{\varepsilon}\Bigr)
\int_{\{\varepsilon<|z+x|\le \varepsilon+\delta\}}P_s(|z|)\,dz,
\end{align*}
where the equality uses the change of variables $z=y-x$. We next bound the Gaussian term uniformly in $s\in(0,T/2]$. Since
$x\in E_\varepsilon$ is fixed, we have $|x|>\varepsilon$. Set $\eta_x:=\frac{|x|-\varepsilon}{2}>0$. Now let $0<\delta\le \eta_x$, then $\varepsilon+\delta \le |x|$. Moreover, if $\varepsilon<|y|\le \varepsilon+\delta$, then
\[
|x-y|
\ge ||x|-|y||
\ge |x|-(\varepsilon+\delta)
\ge \frac{|x|-\varepsilon}{2}
=
\eta_x.
\]
Hence, on the domain of integration in $K_1(\delta,s)$, we have
$|x-y|\ge \eta_x$. Next, recall~\eqref{DefFreeHeat3d},
\[
P_s(|x-y|)
=
(4\pi s)^{-\frac{3}{2}}e^{-\frac{|x-y|^2}{4s}}
=
\pi^{-\frac{3}{2}}|x-y|^{-3}
\Bigl(\frac{|x-y|^2}{4s}\Bigr)^{\frac{3}{2}}
e^{-\frac{|x-y|^2}{4s}}
\le
C\,|x-y|^{-3}
\le
C\,\eta_x^{-3},
\]
for all $0<s\le T/2$ and all $y$ such that
$\varepsilon<|y|\le \varepsilon+\delta$; where the inequality follows using the
elementary bound $u^{\frac{3}{2}}e^{-u}\le C$ for $u>0$\footnote{This follows since the
function $f(u)=u^{\frac{3}{2}}e^{-u}$ is continuous on $(0,\infty)$ and satisfies
$f(u)\to0$ as $u\downarrow0$ and $f(u)\to0$ as $u\to\infty$. Therefore $f$
attains a finite maximum on $(0,\infty)$, which yileds $u^{\frac{3}{2}}e^{-u}\le \max_{u>0} u^{\frac{3}{2}}e^{-u}=C$ for all $u>0$.}, with $u=\frac{|x-y|^2}{4s}$. Therefore,
\begin{align*}
K_1(\delta,s)
&\le
\Bigl(1+C_L+C_L\frac{\sqrt T}{\varepsilon}\Bigr)
\int_{\{\varepsilon<|y|\le \varepsilon+\delta\}}
P_s(|x-y|)\,dy \\
&\le
C_{x,T,L,\varepsilon}
\big|
\{y\in\R^3:\ \varepsilon<|y|\le \varepsilon+\delta\}
\big|
=
4\pi C_{x,T,L,\varepsilon}
\int_\varepsilon^{\varepsilon+\delta}r^2\,dr,
\end{align*}
where the equality follows since the volume of the thin annulus by passing to spherical coordinates is
$\big| \{y\in\R^3:\ \varepsilon<|y|\le \varepsilon+\delta\}
\big| = 4\pi\int_\varepsilon^{\varepsilon+\delta}r^2\,dr$. The right-hand side is independent of $s\in(0,T/2]$ and tends to $0$ as
$\delta\downarrow0$. Therefore,
\begin{align*}
\lim_{\delta\downarrow0}\sup_{0<s\le T/2}K_1(\delta,s)=0.
\end{align*}

\noindent\textit{Second integral.} By definition,
\begin{align*}
K_2(\delta,s)
&\,:=\,
C_L\,s^{-1/2}\,\frac{e^{-\frac{|x|^2}{4s}}}{|x|}
\int_{E_\varepsilon}
\mathbf 1_{\{\varepsilon<|y|\le \varepsilon+\delta\}}
\Bigl[
1+C_L\Bigl(1+\frac{\sqrt T}{|y|}\Bigr)e^{-4\pi\beta|y|}
\Bigr]
\frac{e^{-\frac{|y|^2}{4s}}}{|y|}\,dy \\
&=
C_L\,s^{-1/2}\,\frac{e^{-\frac{|x|^2}{4s}}}{|x|}
\int_{\{\varepsilon<|y|\le \varepsilon+\delta\}}
\Bigl[
1+C_L\Bigl(1+\frac{\sqrt T}{|y|}\Bigr)e^{-4\pi\beta|y|}
\Bigr]
\frac{e^{-\frac{|y|^2}{4s}}}{|y|}\,dy.
\end{align*}
Since $\varepsilon<|y|\le \varepsilon+\delta$ on the domain of integration, we have $\frac1{|y|}\le \frac1\varepsilon$ and $e^{-4\pi\beta|y|}\le 1$. Therefore,
\begin{align*}
K_2(\delta,s)
&\le
C_L\,s^{-1/2}\,\frac{e^{-\frac{|x|^2}{4s}}}{|x|}
\Bigl(1+C_L+C_L\frac{\sqrt T}{\varepsilon}\Bigr)
\int_{\{\varepsilon<|y|\le \varepsilon+\delta\}}
\frac{e^{-\frac{|y|^2}{4s}}}{|y|}\,dy \\
&\le
C_{L,T,\varepsilon}\,
s^{-1/2}\,\frac{e^{-\frac{|x|^2}{4s}}}{|x|}
\int_{\{\varepsilon<|y|\le \varepsilon+\delta\}}
e^{-\frac{|y|^2}{4s}}\,dy
=
C_{L,T,\varepsilon}\,
s^{-1/2}\,\frac{e^{-\frac{|x|^2}{4s}}}{|x|}
\int_\varepsilon^{\varepsilon+\delta}
r^2e^{-\frac{r^2}{4s}}\,dr.
\end{align*}
The equality holds by passing to spherical coordinates. Now, since $r\ge\varepsilon$ on the interval of integration, we have $e^{-\frac{r^2}{4s}} \le e^{-\frac{\varepsilon^2}{4s}}$. Also, for $\varepsilon<r\le\varepsilon+\delta$ and $0<\delta\le1$, we have $r^2\le (\varepsilon+1)^2$. Therefore, $\int_\varepsilon^{\varepsilon+\delta}
r^2e^{-\frac{r^2}{4s}}\,dr \le (\varepsilon+1)^2\,\delta\,e^{-\frac{\varepsilon^2}{4s}}$. Substituting this into the previous display yields the first inequality below.
\begin{align*}
K_2(\delta,s)
\le
\delta\, C_{x,L,T,\varepsilon}\,
s^{-1/2}
e^{-\frac{|x|^2+\varepsilon^2}{4s}} 
\le
\delta\,C_{x,L,T,\varepsilon}\,
\sup_{0<u\le T/2}
u^{-1/2}e^{-\frac{|x|^2+\varepsilon^2}{4u}}
\le
\delta\, C_{x,L,T,\varepsilon}
\end{align*}
The final inequality holds by enlarging the constant, since the function
$s\mapsto s^{-1/2}e^{-\frac{|x|^2+\varepsilon^2}{4s}}$ is bounded on
$(0,T/2]$\footnote{
Define $f(s):=s^{-1/2}e^{-\frac{|x|^2+\varepsilon^2}{4s}}$ for $0<s\le T/2$. Then $f$ is continuous on $(0,T/2]$, and moreover $f(s)\to0$ as $s\downarrow0$.
Hence $f$ extends continuously to $[0,T/2]$ by setting $f(0):=0$.
Therefore $f$ attains a finite maximum on the compact interval $[0,T/2]$, and consequently $\sup_{0<s\le T/2} s^{-1/2}e^{-\frac{|x|^2+\varepsilon^2}{4s}}
<\infty$.
}. The right-hand side is independent of $s\in(0,T/2]$ and tends to $0$ as
$\delta\downarrow0$. Therefore,
\begin{align*}
\lim_{\delta\downarrow0}\sup_{0<s\le T/2}K_2(\delta,s)=0.
\end{align*}
Consequently, recalling that $\widehat{\mathsf K}^{T,\beta,\varepsilon,\pi_1[0,s]}_{[0,s]}(x,A_\delta)
\le K_1(\delta,s)+K_2(\delta,s)$ for $ 0<s\le T/2$, we conclude that
\begin{align*}
\lim_{\delta\downarrow0}
\sup_{0<s\le T/2}
\widehat{\mathsf K}^{T,\beta,\varepsilon,\pi_1[0,s]}_{[0,s]}(x,A_\delta)
=0.
\end{align*}

\noindent\textit{Case 2.} Let $T/2<s\le T$. Then $\pi_1[0,s]=\big\{0,\frac{T}{2},s\big\}$
by~\eqref{IndPartPi1}. Hence, using~\eqref{FirstTrans3dSub}, we obtain the first equality below.
\begin{align*}
\widehat{\mathsf K}^{T,\beta,\varepsilon,\pi_1[0,s]}_{[0,s]}(x,A_\delta)
&=
\int_{E_\varepsilon}\int_{E_\varepsilon}
\mathbf 1_{\{\varepsilon<|y_2|\le \varepsilon+\delta\}}
p^{\,T,\beta}_{0,T/2}(x,y_1)\,
p^{\,T,\beta}_{T/2,s}(y_1,y_2)\,
dy_1\,dy_2 \\
&=
\int_{E_\varepsilon}\int_{E_\varepsilon}
\mathbf 1_{\{\varepsilon<|y_2|\le \varepsilon+\delta\}}
\frac{1+Q^\beta_{T/2}(y_1)}{1+Q^\beta_T(x)}
P^\beta_{T/2}(x,y_1)
\frac{1+Q^\beta_{T-s}(y_2)}{1+Q^\beta_{T/2}(y_1)}
P^\beta_{s-T/2}(y_1,y_2)\,
dy_1\,dy_2 \\
&\le
\int_{E_\varepsilon}\int_{E_\varepsilon}
\mathbf 1_{\{\varepsilon<|y_2|\le \varepsilon+\delta\}}
\bigl(1+Q^\beta_{T-s}(y_2)\bigr)
P^\beta_{T/2}(x,y_1)\,
P^\beta_{s-T/2}(y_1,y_2)\,
dy_1\,dy_2 \\
&=
\int_{E_\varepsilon}
\mathbf 1_{\{\varepsilon<|y_2|\le \varepsilon+\delta\}}
\bigl(1+Q^\beta_{T-s}(y_2)\bigr)
\left(
\int_{E_\varepsilon}
P^\beta_{T/2}(x,y_1)\,
P^\beta_{s-T/2}(y_1,y_2)\,dy_1
\right)
dy_2
\end{align*}
The second equality uses~\eqref{FirstTrans3d}, the inequality uses $1+Q_t^\beta(x)\ge1$, and since the integrand is nonnegative, Tonelli's theorem yields the final equality. Next, since the integrand in the inner integral is nonnegative, enlarging the
domain of integration from $E_\varepsilon$ to $\R^3$ gives
\begin{align*}
\int_{E_\varepsilon}
P^\beta_{T/2}(x,y_1)\,
P^\beta_{s-T/2}(y_1,y_2)\,dy_1
\le
\int_{\R^3}
P^\beta_{T/2}(x,y_1)\,
P^\beta_{s-T/2}(y_1,y_2)\,dy_1 
=
P_s^\beta(x,y_2),
\end{align*}
where the last equality follows from the semigroup property of
$P^\beta$. Substituting this into the previous bound and renaming $y_2$ as
$y$, we obtain
\begin{align*}
\widehat{\mathsf K}^{T,\beta,\varepsilon,\pi_1[0,s]}_{[0,s]}(x,A_\delta)
&\le
\int_{E_\varepsilon}
\mathbf 1_{\{\varepsilon<|y|\le \varepsilon+\delta\}}
\bigl(1+Q^\beta_{T-s}(y)\bigr)
P^\beta_s(x,y)\,dy.
\end{align*}
This is exactly the same upper bound as in Case~1. Repeating the estimates
from Case~1, we obtain
\begin{align*}
\lim_{\delta\downarrow0}
\sup_{T/2<s\le T}
\widehat{\mathsf K}^{T,\beta,\varepsilon,\pi_1[0,s]}_{[0,s]}(x,A_\delta)
=0.
\end{align*}
Combining Cases~1 and~2, completes the proof.
\end{proof}

\begin{proof}[Proof of Proposition~\ref{PropTightnessOneTimeMarginals}]
Fix $\eta>0$. By Lemma~\ref{LemUniformOneTimeBoundaryApproach}, there
exists $\delta=\delta(\eta)>0$ such that
\begin{align*}
\sup_{m\ge1}
\mathbb P_{x}^{T,\beta,\varepsilon,(m)}
\bigl(\varepsilon<|Y_t|\le \varepsilon+\delta\bigr)
<
\frac{\eta}{2}.
\end{align*}
Next, by Lemma~\ref{LemUniformOneTimeTailInfinity}, there exists
$R=R(\eta)<\infty$ such that
\begin{align*}
\sup_{m\ge1}
\mathbb P_{x}^{T,\beta,\varepsilon,(m)}
\bigl(|Y_t|>R\bigr)
<
\frac{\eta}{2}.
\end{align*}

Define $K_{\delta,R} := \big\{ y\in E_\varepsilon:\ \varepsilon+\delta\le |y|\le R \big\} \cup \{\Delta\}$. We claim that $K_{\delta,R}$ is compact in $\overline E_\varepsilon$.
Indeed, the set $\big\{
y\in E_\varepsilon:\ \varepsilon+\delta\le |y|\le R
\big\}$ is closed and bounded in $\mathbb R^3$, hence compact in the Euclidean
topology. Since on subsets bounded away from the boundary
$\{|y|=\varepsilon\}$ the topology induced by $\rho$ agrees with the Euclidean
topology, this set is also compact as a subset of $\overline E_\varepsilon$.
Moreover, $\{\Delta\}$ is compact, so $K_{\delta,R}$ is compact.

Next, observe that $\overline E_\varepsilon\setminus K_{\delta,R} = \big\{ y\in E_\varepsilon:\ \varepsilon<|y|<\varepsilon+\delta
\big\} \cup \big\{ y\in E_\varepsilon:\ |y|>R
\big\}$. Hence, for every $m\ge1$,
\begin{align*}
\mathbb P_{x}^{T,\beta,\varepsilon,(m)}
\bigl(Y_t\notin K_{\delta,R}\bigr)
&\le
\mathbb P_{x}^{T,\beta,\varepsilon,(m)}
\bigl(\varepsilon<|Y_t|<\varepsilon+\delta\bigr)
+
\mathbb P_{x}^{T,\beta,\varepsilon,(m)}
\bigl(|Y_t|>R\bigr) \\
&\le
\mathbb P_{x}^{T,\beta,\varepsilon,(m)}
\bigl(\varepsilon<|Y_t|\le \varepsilon+\delta\bigr)
+
\mathbb P_{x}^{T,\beta,\varepsilon,(m)}
\bigl(|Y_t|>R\bigr).
\end{align*}
Taking the supremum over $m\ge1$ and using the two bounds above, we obtain
\begin{align*}
\sup_{m\ge1}
\mathbb P_{x}^{T,\beta,\varepsilon,(m)}
\bigl(Y_t\notin K_{\delta,R}\bigr)
<
\eta.
\end{align*}

Since $\eta>0$ was arbitrary, it follows that for every $\eta>0$ there exists a
compact set $K_{\delta(\eta),R(\eta)}\subset\overline E_\varepsilon$ such that the above holds. Equivalently, the family of probability measures $\big\{ \mathbb P_{x}^{T,\beta,\varepsilon,(m)}\circ Y_t^{-1} \big\}_{m\ge1}$ is tight on $\overline E_\varepsilon$.
\end{proof}

\section{Special functions}\label{SpecialFunctions}
For $x\in\mathbb{R}$, recall that the error function and its complementary function are defined, respectively, by
\begin{align} \label{DefErrorFunction}
\erf(x)
\,:=\,
\frac{2}{\sqrt{\pi}}\int_{0}^{x} e^{-u^{2}}\,du,
\qquad
\erfc(x)
\,:=\,
1-\erf(x)
\,=\,
\frac{2}{\sqrt{\pi}}\int_{x}^{\infty} e^{-u^{2}}\,du \, .
\end{align}
The next three subsections collect analytic properties of the
functions associated with the three-dimensional attractive
point interaction heat kernel. In
Section~\ref{SubsectPFunct}, we establish integral identities and
basic estimates for the kernel $P_t^\beta(x,y)$. In
Section~\ref{SubsectQFunct}, we derive explicit representations and
quantitative bounds for the function $Q_t^\beta(x)$. Finally, in
Section~\ref{SubsectDoobpFunct}, we study the corresponding
Doob-transformed transition densities
$p_{s,t}^{T,\beta}(x,y)$ and prove relevant quantitative estimates.

\subsection{The function \texorpdfstring{$\boldsymbol{P_{ t}^{\beta}(x,y)}$}{Lg}}\label{SubsectPFunct}

Recall the three-dimensional attractive point interaction heat kernel \(P_t^\beta(x,y)\) defined in~\eqref{DefPointKer3dBeta}, which is strictly positive and symmetric in \((x,y)\). The map \((t,x,y)\mapsto P_t^\beta(x,y)\) is jointly continuous on \((0,\infty)\times(\R^3\setminus\{0\})\times(\R^3\setminus\{0\})\), and for each fixed \(y\neq0\), the function \(x\mapsto P_t^\beta(x,y)\) solves the heat equation away from the origin,
\[
\partial_t P_t^\beta(x,y)
=
\frac{1}{2}\,\Delta_x P_t^\beta(x,y),
\qquad x\in\R^3\setminus\{0\},\; t>0,
\]
with the singular interaction encoded at \(x=0\); see~\cite[Cor.~7]{Fleischmann}. Moreover, for each \(\beta>0\) and each \(T>0\), by~\cite[Lem.~8]{Fleischmann} there exists a constant \(C=C(\beta,T)>0\) such that for all \(0<t\le T\) and all \(x,y\in\R^3\setminus\{0\}\),
\begin{align}\label{PtBetaBounds}
P_t(x,y)
\;\le\;
P_t^\beta(x,y)
\;\le\;
P_t(x,y)
+
C\,t^{-1/2}\,
\frac{e^{-\frac{|x|^2}{4t}}}{|x|}
\frac{e^{-\frac{|y|^2}{4t}}}{|y|}\,,
\end{align}
where \(P_t(x,y)\) denotes the three-dimensional free heat kernel defined in~\eqref{DefFreeHeat3d}.

\begin{lemma}\label{GaussianIdentity}
Fix $t>0$, and $x\in\R^3\setminus\{0\}$. Then,
\begin{align*}
\int_{\R^3}\frac{1}{|y|}\,P_t(|x|+|y|)\,dy
\, = \,
\frac{1}{\sqrt{\pi t}}\,e^{-\frac{|x|^2}{4t}}
\, - \,
\frac{|x|}{2t} \,\erfc\!\Big(\frac{|x|}{2\sqrt t}\Big)\,.
\end{align*}

\end{lemma}

\begin{proof}
Using the definition of $P_t(r)$ from~\eqref{DefFreeHeat3d}, we get the first equality below.
\begin{align*}
\int_{\R^3}\frac{1}{|y|}\,P_t(|x|+|y|)\,dy
=
\frac{1}{(4\pi t)^{\frac{3}{2}}}\, \int_{\R^3}\frac{1}{|y|}
e^{-\frac{(|x|+|y|)^2}{4t}}\,dy 
=
\frac{1}{2t \sqrt{\pi t}} \, \int_0^\infty r\,e^{-\frac{(|x|+r)^2}{4t}}\,dr
\end{align*}
The second equality uses spherical coordinates in $y$ (so $r=|y|$ and $dy=4\pi r^2dr$). Using the change of variables $s=|x|+r$, we compute
\begin{align*}
\int_0^\infty r\,e^{-\frac{(|x|+r)^2}{4t}}\,dr
\, = \, \int_{|x|}^\infty s\,e^{-\frac{s^2}{4t}}\,ds
-|x|\int_{|x|}^\infty e^{-\frac{s^2}{4t}}\,ds
 \, = \, 2t\,e^{-\frac{|x|^2}{4t}}
-|x|\sqrt{\pi t}\,\erfc\!\Big(\frac{|x|}{2\sqrt t}\Big),
\end{align*}
where the last equality uses the definition of $\erfc$ in~\eqref{DefErrorFunction}. Substituting this back into the last display completes the proof.
\end{proof}

\begin{lemma}
Let $t, \beta>0$ and $x \in \R^3 \setminus \{0\}$. Then,
\begin{align}
\int_0^\infty e^{4\pi\beta u}\,e^{-\frac{(|x|+u)^2}{4t}}\,du
\,=\,&
\sqrt{\pi t}\,e^{\,-4\pi\beta |x|+16\pi^2\beta^2 t}\,
\erfc\!\Big(\frac{|x|}{2\sqrt t}-4\pi\beta\sqrt t\Big)\, , \label{EqLaplaceGaussian} \\
\int_0^\infty e^{4\pi\beta u}\,u\,e^{-\frac{(|x|+u)^2}{4t}}\,du
\,=\,&
\sqrt{\pi t}\,(8\pi\beta t-|x|)\,e^{-4\pi\beta|x|+16\pi^2\beta^2t}\,
\erfc\!\Big(\frac{|x|}{2\sqrt t}-4\pi\beta \sqrt{t}\Big)
\,+\,
2t\,e^{-\frac{|x|^2}{4t}}\, . \label{EqLaplaceGaussian2}
\end{align}
\end{lemma}

\begin{proof}
First observe the following
\begin{align}\label{ExpPowerIdentity}
4\pi\beta u-\frac{(|x|+u)^2}{4t}
\,= \,&
-\frac1{4t}\Big(u^2+(2|x|-16\pi\beta t)u+|x|^2\Big) \nonumber \\
\,= \,&
-\frac1{4t}\Big(\big(u+|x|-8\pi\beta t\big)^2-\big(|x|-8\pi\beta t\big)^2+|x|^2\Big) \nonumber \\
\,= \,&
-\frac{(u+|x|-8\pi\beta t)^2}{4t}
 \,- \, 4\pi\beta |x|+16\pi^2\beta^2 t\,,
\end{align}
where the second equality uses $u^2+(2|x|-16\pi\beta t)u
= \big(u+|x|-8\pi\beta t\big)^2-\big(|x|-8\pi\beta t\big)^2$. Therefore, we have the first equality below.
\begin{align}
\int_0^\infty e^{4\pi\beta u}\,e^{-\frac{(|x|+u)^2}{4t}}\,du
\,= \,&
e^{-4\pi\beta |x|+16\pi^2\beta^2 t}\,
\int_0^\infty \, e^{-\frac{(u+|x|-8 \pi \beta t)^2}{4t}}\,du \nonumber \\
\,= \,&
2\sqrt{t}\, e^{-4\pi\beta |x|+16\pi^2\beta^2 t}\,
\int_{\frac{|x|}{2\sqrt t} - 4\pi\beta\sqrt t}^\infty \, e^{-w^2}\,dw \nonumber
\end{align}
The second equality uses the following change of variables $w =  \frac{u+|x|-8\pi\beta t}{2\sqrt t}$. Thus, using the definition of $\erfc$ in~\eqref{DefErrorFunction} we obtain~\eqref{EqLaplaceGaussian}. \vspace{.3cm}

For~\eqref{EqLaplaceGaussian2} we use~\eqref{ExpPowerIdentity} to obtain the equality below.
\begin{align*}
\int_0^\infty e^{4\pi\beta u}\,u\,e^{-\frac{(|x|+u)^2}{4t}}\,du 
\,=\,&
e^{-4\pi\beta|x|+16\pi^2\beta^2t}\,
\int_0^\infty u\,e^{-\frac{(u+|x|-8\pi\beta t)^2}{4t}}\,du \, 
\end{align*}

\noindent Next, with the change of variables $z= \frac{u+|x|-8\pi\beta t}{2\sqrt t}$ so that
$u=2\sqrt t\,z-|x|+8\pi\beta t$ and $du=2\sqrt t\,dz$, we obtain the first equality below.
\begin{align*}
\int_0^\infty u\,e^{-\frac{(u+|x|-8\pi\beta t)^2}{4t}}\,du
&=
2\sqrt t\,
\int_{\frac{|x|}{2\sqrt t}-4\pi\beta \sqrt{t}}^\infty\, 
\big(2\sqrt t\,z-|x|+8\pi\beta t\big)\,e^{-z^2}\,dz \\
&=
2\sqrt t\,(8\pi\beta t-|x|)
\int_{\frac{|x|}{2\sqrt t}-4\pi\beta \sqrt{t}}^\infty \, e^{-z^2}\,dz
\,+\,
4t\int_{\frac{|x|}{2\sqrt t}-4\pi\beta \sqrt{t}}^\infty \,z\,e^{-z^2}\,dz \\
&=
\sqrt{\pi t}\,(8\pi\beta t-|x|)\,
\erfc\!\Big(\frac{|x|}{2\sqrt t}-4\pi\beta \sqrt{t}\Big)
\,+\,
2t\,e^{-\frac{(|x|-8\pi\beta t)^2}{4t}}\,
\end{align*}
The last equality uses~\eqref{DefErrorFunction}. Substituting this into the previous display and using the elementary identity $-4\pi\beta|x|+16\pi^2\beta^2 t -\frac{(|x|-8\pi\beta t)^2}{4t} = -\frac{|x|^2}{4t}$ yields~\eqref{EqLaplaceGaussian2}.
\end{proof}

\begin{lemma}\label{LemIntPyPoint3d}
Fix $t,\beta>0$, and $x\in\R^3\setminus\{0\}$. For the function $P_t^\beta(x,y)$ defined in~\eqref{DefPointKer3dBeta},
\begin{align}
\int_{\R^3} P_t^\beta(x,y)\,dy
\,=\,
1
\,+\,
\frac{1}{4\pi\beta\,|x|}\bigg[
e^{\,-4\pi\beta |x|+16\pi^2\beta^2 t}\,
\erfc\!\Big(\frac{|x|}{2\sqrt t}-4\pi\beta\sqrt t\Big)
\,-\,
\erfc\!\Big(\frac{|x|}{2\sqrt t}\Big)
\bigg]\,.
\end{align}

\end{lemma}

\begin{proof}
Since $P_t(x,y)$ is the three-dimensional Gaussian heat kernel, integration over $y \in \R^3$ is one in the first term of~\eqref{DefPointKer3dBeta}.
\begin{align}\label{Eq:DecomposeInt}
\int_{\R^3}P_t^\beta(x,y)\,dy
\,=\,
1 \,+\, \frac{2t}{|x|}\, \int_{\R^3}\, \frac{1}{|y|}\,P_t(|x|+|y|)\,dy
\,+\, \frac{8\pi\beta t}{|x|}\, \int_{\R^3}\, \frac{1}{|y|} \, \int_0^\infty e^{4\pi\beta u}\,P_t(u+|x|+|y|)\,du\,dy\,
\end{align}
Using Lemma~\ref{GaussianIdentity}, we get the following.
\begin{align}\label{Eq:I3Split1}
\frac{2t}{|x|}\, \int_{\R^3}\, \frac{1}{|y|}\,P_t(|x|+|y|)\,dy
\,=\,
\frac{2\sqrt t}{\sqrt\pi\,|x|}\,e^{-\frac{|x|^2}{4t}}
-\erfc\!\Big(\frac{|x|}{2\sqrt t}\Big)
\end{align}
Again using Lemma~\ref{GaussianIdentity} with $|x|$ replaced by $|x|+u$, we obtain the second equality below.
\begin{align}\label{Eq:I3Split}
\int_{\R^3}\, \frac{1}{|y|} \, \int_0^\infty e^{4\pi\beta u}\,& P_t(u+|x|+|y|)\,du\,dy \nonumber \\
&=\int_0^\infty e^{4\pi\beta u}
\Bigg(\int_{\R^3}\frac{1}{|y|}P_t(u+|x|+|y|)\,dy\Bigg)\,du \nonumber \\
&=
\frac{1}{\sqrt{\pi t}}\,
\int_0^\infty e^{4\pi\beta u}\,e^{-\frac{(|x|+u)^2}{4t}}\,du
\,-\,
\frac{1}{2t}
\int_0^\infty e^{4\pi\beta u}\,(|x|+u)\,
\erfc\!\Big(\frac{|x|+u}{2\sqrt t}\Big)\,du\,  
\end{align}
For the second integral in~\eqref{Eq:I3Split}, observe that $(|x|+u)e^{4\pi\beta u}=\frac{d}{du}\Big(\frac{|x|+u}{4\pi\beta}e^{4\pi\beta u}-\frac{1}{(4\pi\beta)^2}e^{4\pi\beta u}\Big)$, therefore we have the first equality below.
\begin{align}
\int_0^\infty \, e^{4\pi\beta u}\,(|x|+u)\,&
\erfc\!\Big(\frac{|x|+u}{2\sqrt t}\Big)\,du\, \nonumber \\
\,=\,&
\int_0^\infty 
\erfc\!\Big(\frac{|x|+u}{2\sqrt t}\Big)\,d\Big(\frac{|x|+u}{4\pi\beta}e^{4\pi\beta u} \, -\, \frac{1}{(4\pi\beta)^2}e^{4\pi\beta u}\Big)\, \nonumber \\
\,=\,&
\erfc\!\Big(\frac{|x|+u}{2\sqrt t}\Big)\,\Big(\frac{|x|+u}{4\pi\beta}e^{4\pi\beta u} \, -\, \frac{1}{(4\pi\beta)^2}e^{4\pi\beta u}\Big)\bigg|_{u=0}^{u=\infty}
\nonumber \\
\,+\,&
\frac{1}{\sqrt{\pi t}}
\int_0^\infty \,
e^{-\frac{(|x|+u)^2}{4t}}\,
\Big(\frac{|x|+u}{4\pi\beta}e^{4\pi\beta u} \, -\, \frac{1}{(4\pi\beta)^2}e^{4\pi\beta u}\Big)\, du \nonumber \\
\,=\,&
\Big(-\frac{|x|}{4\pi \beta}+\frac{1}{(4\pi\beta)^2}\Big)\erfc\!\Big(\frac{|x|}{2\sqrt t}\Big)
\, +\,
\frac{1}{4\pi \beta\sqrt{\pi t}}\, \int_0^\infty \,
e^{4\pi\beta u} \, (|x|+u)\,e^{-\frac{(|x|+u)^2}{4t}}\, du \nonumber \\
\, -\,&
\frac{ 1 }{(4 \pi \beta)^2 \sqrt{\pi t}}\, \int_0^\infty \,e^{4\pi\beta u}\, 
e^{-\frac{(|x|+u)^2}{4t}}\, du \nonumber
\end{align}
The second equality uses integration by parts and the fact that $\frac{d}{du}\erfc\!\big(\frac{|x|+u}{2\sqrt t}\big)= -\frac{1}{\sqrt{\pi t}}\,e^{-\frac{(|x|+u)^2}{4t}}$. The simplification in the third equality uses the standard asymptotic $\erfc(z)\sim \frac{1}{\sqrt{\pi}\,z}e^{-z^{2}}$ as $z\to\infty$, so that $z e^{z}\erfc(z)\sim \frac{1}{\sqrt{\pi}}e^{\,z-z^{2}}\to0$ as $z\to\infty$. Substituting the above in~\eqref{Eq:I3Split} yields the following.
\begin{align}\label{Eq:I3Split2}
\int_{\R^3}\, \frac{1}{|y|} \, &\int_0^\infty e^{4\pi\beta u}\, P_t(u+|x|+|y|)\,du\,dy \nonumber \\
\,=\,&
\frac{1}{\sqrt{\pi t}}\, \Big(1\, + \, \frac{ 1 }{2t(4 \pi \beta)^2 }\Big)
\int_0^\infty e^{4\pi\beta u}\,e^{-\frac{(|x|+u)^2}{4t}}\,du 
\,+\,\frac{4\pi \beta|x|-1}{2 t (4\pi \beta )^2}\,\erfc\!\Big(\frac{|x|}{2\sqrt t}\Big) \nonumber \\
\,-\,&
\frac{1}{8 \pi \beta t \sqrt{ \pi t}}
\int_0^\infty e^{4\pi\beta u}(|x|+u)\,e^{-\frac{(|x|+u)^2}{4t}}\,du\,\nonumber \\
\,=\,&
\Big(1\, + \, \frac{ 1 }{2t(4 \pi \beta)^2 }\Big)
\, e^{-4\pi\beta |x|+16\pi^2\beta^2 t}\,
\erfc\!\Big(\frac{|x|}{2\sqrt t}-4\pi\beta\sqrt t\Big) 
\,+\,\frac{4\pi \beta|x|-1}{2 t (4\pi \beta )^2}\,\erfc\!\Big(\frac{|x|}{2\sqrt t}\Big) \nonumber \\
\,-\,&
e^{-4\pi\beta |x|+16\pi^2\beta^2 t}\,
\erfc\!\Big(\frac{|x|}{2\sqrt t}-4\pi\beta\sqrt t\Big)
\,-\,
\frac{1}{4\pi \beta \sqrt{\pi t}}\,e^{-\frac{|x|^2}{4t}}\, \nonumber \\
\,=\,&
\frac{ 1 }{2t(4 \pi \beta)^2 }\,
e^{-4\pi\beta |x|+16\pi^2\beta^2 t}\,
\erfc\!\Big(\frac{|x|}{2\sqrt t}-4\pi\beta\sqrt t\Big) 
\,+\,\frac{4\pi \beta|x|-1}{2 t (4\pi \beta )^2}\,\erfc\!\Big(\frac{|x|}{2\sqrt t}\Big)
\,-\,
\frac{1}{4\pi \beta \sqrt{\pi t}}\,e^{-\frac{|x|^2}{4t}}\,
\end{align}
The second equality uses~\eqref{EqLaplaceGaussian} for the first term. The third and fourth terms in the second equality above follow from the following
\begin{align}
\int_0^\infty e^{4\pi\beta u}(|x|+u)\,e^{-\frac{(|x|+u)^2}{4t}}\,du
\,=\,&
-2t\,\frac{d}{d|x|}
\int_0^\infty e^{4\pi\beta u}\,e^{-\frac{(|x|+u)^2}{4t}}\,du \nonumber \\
\,=\,&
-2t\,\sqrt{\pi t}\, \frac{d}{d|x|}\bigg(
e^{-4\pi\beta |x|+16\pi^2\beta^2 t}\,
\erfc\!\Big(\frac{|x|}{2\sqrt t}-4\pi\beta\sqrt t\Big)
\bigg) \nonumber \\
\,=\,&
8\pi\beta t\, \sqrt{\pi t}\, e^{-4\pi\beta |x|+16\pi^2\beta^2 t}\,
\erfc\!\Big(\frac{|x|}{2\sqrt t}-4\pi\beta\sqrt t\Big)
\,+\,
2t\,e^{-\frac{|x|^2}{4t}}\, , \nonumber
\end{align}
where the second equality above again uses~\eqref{EqLaplaceGaussian} and the second term in the third equality uses the fact that $\frac{d}{d|x|}\erfc\!\Big(\frac{|x|}{2\sqrt t}-4\pi\beta\sqrt t\Big)
= -\frac{1}{\sqrt{\pi t}}\,
e^{-\big(\frac{|x|}{2\sqrt t}-4\pi\beta\sqrt t\big)^2}$. Plugging in~\eqref{Eq:I3Split1} and~\eqref{Eq:I3Split2} in~\eqref{Eq:DecomposeInt} yields the following
\begin{align}
\int_{\R^3}P_t^\beta(x,y)\,dy
\,=\,&
1 \,+\, \frac{2\sqrt t}{\sqrt\pi\,|x|}\,e^{-\frac{|x|^2}{4t}}
-\erfc\!\Big(\frac{|x|}{2\sqrt t}\Big)  \nonumber \\
\,+\, &
\frac{ 1 }{4 \pi \beta |x| }\,
e^{-4\pi\beta |x|+16\pi^2\beta^2 t}\,
\erfc\!\Big(\frac{|x|}{2\sqrt t}-4\pi\beta\sqrt t\Big) 
\,+\,
\frac{4\pi \beta|x|-1}{4\pi \beta |x|}\,\erfc\!\Big(\frac{|x|}{2\sqrt t}\Big)
\,+\,
\frac{2 \sqrt{t} }{ \sqrt{\pi }|x|}\,e^{-\frac{|x|^2}{4t}}\, \nonumber \\
\,=\,&
1 \,+\, 
\frac{ 1 }{4 \pi \beta |x| }\,
e^{-4\pi\beta |x|+16\pi^2\beta^2 t}\,
\erfc\!\Big(\frac{|x|}{2\sqrt t}-4\pi\beta\sqrt t\Big) 
\,-\,
\frac{1}{4\pi \beta |x|}\,\erfc\!\Big(\frac{|x|}{2\sqrt t}\Big) \,, \nonumber
\end{align}
which is the desired equality.
\end{proof}

\subsection{The function \texorpdfstring{$\boldsymbol{Q_{ t}^{\beta}(x)}$}{Lg}}\label{SubsectQFunct}

Recall that $Q_t^\beta:\R^3 \to [0,\infty]$ is defined by $Q_t^\beta(x):= \int_{\R^3} P_t^\beta(x,y)dy-1$. Using Lemma~\ref{LemIntPyPoint3d}, we may therefore obtain the following explicit form
\begin{align}\label{DefH3dIst}
Q_t^\beta(x)
\, =\, 
\frac{1}{4\pi\beta\,|x|}\bigg[
e^{-4\pi\beta |x|+16\pi^2\beta^2 t}\,
\erfc\!\Big(\frac{|x|}{2\sqrt t}-4\pi\beta\sqrt t\Big)
\, - \, 
\erfc\!\Big(\frac{|x|}{2\sqrt t}\Big)
\bigg]\,,
\end{align}
where, recall that, $\erfc(z)$ is defined in~\eqref{DefErrorFunction}. Thus, the function $Q_t^\beta(x)$ can alternatively be expressed as follows
\begin{align} \label{DefH3d3rd}
Q_t^\beta(x)
\, =\, 
\frac{1}{2 \boldsymbol{\beta} \mathbf{x}}\Big[
e^{-2\boldsymbol{\beta} \mathbf{x}+\boldsymbol{\beta}^2 }\,
\erfc\!(\mathbf{x}-\boldsymbol{\beta})
\, - \, 
\erfc\!( \mathbf{x} )
\Big] \,,
\end{align}
where $\mathbf{x}:=\frac{|x|}{2\sqrt t}$ and $\boldsymbol{\beta}:=4\pi\beta\sqrt t$.

\begin{lemma}\label{LemmaDefH3d2nd}
Fix $t, \beta>0$ and $x\in\R^3\setminus\{0\}$. The function $Q_t^\beta$ admits the representation
\begin{align}\label{DefH3d2nd}
Q_t^\beta(x)
\, = \,
\frac{1}{4\pi\beta\,|x|\,\sqrt{\pi t}}
\int_{0}^{\infty} \,
\big(e^{4\pi\beta w}\,-\, 1\big)\,
e^{-\frac{(|x|+w)^2}{4t}}\,dw\,.
\end{align}
\end{lemma}

\begin{proof}
Using the definition of $Q_t^\beta$ in~\eqref{DefH3dIst} and $\erfc(z)$ as given in~\eqref{DefErrorFunction}, we obtain the first equality below.
\begin{align}
Q_t^\beta(x)
\, =\, &
\frac{1}{4\pi\beta\,|x|}\bigg[
\frac{2}{\sqrt{\pi}}\,
e^{-4\pi\beta |x|+16\pi^2\beta^2 t}\,
\int_{\frac{|x|}{2\sqrt t}-4\pi\beta\sqrt t}^{\infty} e^{-u^{2}}\,du
\, - \,
\frac{2}{\sqrt{\pi}}\,
\int_{\frac{|x|}{2 \sqrt{t}}}^{\infty} e^{-u^{2}}\,du
\bigg] \nonumber \\
\, =\, &
\frac{1}{4\pi\beta\,|x|}\bigg[
\frac{1}{\sqrt{\pi t}}\,
e^{-4\pi\beta |x|+16\pi^2\beta^2 t}\,
\int_{0}^{\infty}e^{-\frac{(|x|+w-8\pi\beta t)^2}{4t}}\,dw
\, - \,
\frac{1}{\sqrt{\pi t}}\, \int_{0}^{\infty}e^{-\frac{(|x|+w)^2}{4t}}\,dw
\bigg] \nonumber
\end{align}
Here the second equality follows by the changes of variables \(u=\frac{|x|+w-8\pi\beta t}{2\sqrt t}\) in the first integral and \(u=\frac{|x|+w}{2\sqrt t}\) in the second. Observe that
\[
\frac{(|x|+w-8\pi\beta t)^2}{4t}
=
\frac{(|x|+w)^2}{4t}-4\pi\beta(|x|+w)+16\pi^2\beta^2 t\, .
\]
Substituting this into the previous display yields~\eqref{DefH3d2nd}.
\end{proof}

\begin{lemma}\label{LemmaHUpDown}
The function $Q_T^\beta(x)$ satisfies the following bounds for $T,\beta>0$ and $x\in\R^3\setminus\{0\}$.
\begin{enumerate}[(i)]

\item $\displaystyle 1 + Q_T^\beta(x) \ge  \frac{\sqrt{T}}{|x| \sqrt{\pi}} e^{-\frac{|x|^2}{4T}} $

\item Fix $L>0$. There exists a constant $C_L>0$ such that for all $T, \beta>0$ with $\beta\sqrt T\le L$, and all $x\in\R^3\setminus\{0\}$,
\begin{align}\label{QUpperBoundCL}
Q_T^\beta(x)
\;\le\;
C_L\Big(1+\frac{\sqrt T}{|x|}\Big)\,e^{-4\pi\beta|x|}\,.
\end{align}

\item Fix $L>0$. There exists a constant $C_L>0$ such that for all $T, \beta>0$ with $\beta\sqrt T\le L$, and all $x\in\R^3\setminus\{0\}$,
\begin{align} \nonumber
\big|\nabla Q_T^\beta(x)\big|
\,\le\,&
C_L\,
\bigg[
\frac{1}{|x|}
+\frac{1+\sqrt{T}}{|x|^2}
+\frac{1}{|x|\sqrt{T}}
\bigg]\,
e^{-4\pi\beta|x|} \,.
\end{align}
\end{enumerate}
\end{lemma}

\begin{proof}
\noindent Part (i): Since $e^{a}-1\ge a\ge \frac{1}{2}a$ for all $a\ge0$, we have $e^{4\pi\beta w}-1\ge 2\pi\beta w$ for all $w\ge0$. Therefore, using the representation~\eqref{DefH3d2nd}, we obtain the inequality below.
\begin{align*}
Q_T^\beta(x)
\, \ge \,
\frac{1}{4\pi\beta\,|x|\,\sqrt{\pi T}}
\int_{0}^{\infty}
(2\pi\beta w)\,
e^{-\frac{(|x|+w)^2}{4T}}\,dw 
\,=\,
\frac{1}{2|x|\,\sqrt{\pi T}}
\int_{|x|}^{\infty} u\,e^{-\frac{u^2}{4T}}\,du
\,-\,
\frac{1}{2\sqrt{\pi T}}\int_{|x|}^{\infty} e^{-\frac{u^2}{4T}}\,du
\end{align*}
The equality uses the substitution $u=|x|+w$. The first integral above equals $2Te^{-\frac{|x|^2}{4T}}$ by the substitution $q=u^2/(4T)$, whereas the second integral equals $\sqrt{\pi T} \erfc\!\big(\frac{|x|}{2\sqrt T}\big)$ by the change of variables $v=u/(2\sqrt T)$. Substituting these identities into the previous display yields the first inequality below.
\begin{align*}
Q_T^\beta(x)
\, \ge \,
\frac{\sqrt{T}}{|x|\,\sqrt{\pi}}\,e^{-\frac{|x|^2}{4T}}
\,-\,
\frac{1}{2}\,\erfc\!\Big(\frac{|x|}{2\sqrt T}\Big)
\, \ge \,
\frac{\sqrt{T}}{|x|\,\sqrt{\pi}}\,e^{-\frac{|x|^2}{4T}}
\,-\,
1
\end{align*}
The second inequality uses that $\erfc\!(z) \le 2$ for all $z \in \R$. The above implies the desired bound. \vspace{.3cm}

\noindent Part (ii): Using the inequality $e^u-1 \le u e^u$ for $u\ge 0$ with $u=4\pi\beta w$ in~\eqref{DefH3d2nd} we obtain the inequality below.
\begin{align}\label{HUpperBoundProof}
Q_T^\beta(x)
\,\le\,&
\frac{1}{4\pi\beta\,|x|\,\sqrt{\pi T}}
\int_{0}^{\infty} 
\big(4\pi\beta w\big)\,e^{4\pi\beta w}\,
e^{-\frac{(|x|+w)^2}{4T}}\,dw  \nonumber \\
\,=\,&
\frac{1}{|x|\,\sqrt{\pi T}}
\int_{0}^{\infty} 
w\,e^{4\pi\beta w}\,
e^{-\frac{(|x|+w)^2}{4T}}\,dw 
\,=\,
\frac{e^{-4\pi\beta|x|+16\pi^2\beta^2 T}}{|x|\sqrt{\pi T}}
\int_{0}^{\infty} w\,e^{-\frac{(|x|-8\pi\beta T+w)^2}{4T}}\,dw 
\end{align}
The last equality uses~\eqref{ExpPowerIdentity}. Next, using the change of variables $u=w-(8\pi\beta T-|x|)$ we obtain the first equality below.
\begin{align}\label{HUpperBoundProof2}
\int_{0}^{\infty} w\,e^{-\frac{(|x|-8\pi\beta T+w)^2}{4T}}\,dw
\,=\,&
\int_{-(8\pi\beta T-|x|)}^{\infty}
\Big(u+(8\pi\beta T-|x|)\Big)\,e^{-\frac{u^2}{4T}}\,du \nonumber \\
\,\le\,&
\int_{-\infty}^{\infty} |u|\,e^{-\frac{u^2}{4T}}\,du
\;+\;
\big|8\pi\beta T-|x|\big|
\int_{-\infty}^{\infty} e^{-\frac{u^2}{4T}}\,du
\,=\,
4T
+
2\sqrt{\pi T}\, \big|8\pi\beta T-|x|\big| 
\end{align}
The first inequality follows from the pointwise bound $u+(8\pi\beta T-|x|) \le |u|+\big|8\pi\beta T-|x|\big|$ for all $u\in\mathbb{R}$, together with the fact that $e^{-u^2/(4T)}\ge0$, and from enlarging the domain of integration from $\big[-(8\pi\beta T-|x|),\infty\big)$ to $\mathbb{R}$ on the right-hand side. The final equality uses that the Gaussian integrals are explicitly given by
\[
\int_{-\infty}^{\infty} |u|\,e^{-\frac{u^2}{4T}}\,du
\,=\, 4T\,,
\quad \textup{ and } \quad
\int_{-\infty}^{\infty} e^{-\frac{u^2}{4T}}\,du
\,=\, 2\sqrt{\pi T}\,.
\]
Substituting~\eqref{HUpperBoundProof2} into~\eqref{HUpperBoundProof} yields the first inequality below.
\begin{align}
Q_T^\beta(x)
\;\le\; &
\frac{e^{-4\pi\beta|x|+16\pi^2\beta^2 T}}{|x|\sqrt{\pi T}}
\Big[4T\,+\,2\sqrt{\pi T}\,\big|8\pi\beta T-|x|\big|\Big] \nonumber \\
\,\le \,&
e^{-4\pi\beta|x|+16\pi^2\beta^2 T}
\bigg[
\frac{4}{\sqrt\pi}\frac{\sqrt T}{|x|}
\;+\;
2
\;+\;
16\pi\beta\,\frac{T}{|x|} \bigg] 
\,\le \,
e^{-4\pi\beta|x|}\,e^{16\pi^2L^2}
\bigg[
2
\,+\,
\Big(\frac{4}{\sqrt\pi}+16\pi L\Big)\frac{\sqrt T}{|x|}
\bigg] \nonumber
\end{align}
The second inequality uses the triangle inequality $\big|8\pi\beta T-|x|\big|\le |x|+8\pi\beta T$, and the final inequality uses the assumption $\beta\sqrt T\le L$. Finally, since $2 \le 2\big(1+\frac{\sqrt T}{|x|}\big)$ and $\frac{\sqrt T}{|x|} \le 1+\frac{\sqrt T}{|x|}$, the above estimate yields the upper bound~\eqref{QUpperBoundCL} with $C_L := e^{16\pi^2L^2}\,\big(2+\frac{4}{\sqrt\pi}+16\pi L\big)$. \vspace{.3cm}

\noindent Part (iii): Using that $\big| \nabla Q_{ T}^{\beta}(x)\big| =-\frac{\partial}{\partial a}\bar{Q}_{ T}^{\beta}(a)$ for $a=|x|$
and recalling~\eqref{DefH3d2nd}, we obtain the first identity below.
\begin{align} \label{NablaHUpper}
\big| \nabla \,Q_{ T}^{\beta}(x)\big|
\,=\,&
-\frac{\partial}{\partial a}\,
\frac{1}{4\pi\beta \, a\, \sqrt{\pi T}}
\int_{0}^{\infty}\big(e^{4\pi\beta w}\,-\, 1\big)\,
e^{-\frac{(a+w)^2}{4T}}\,dw\,
\bigg|_{a=|x|} \nonumber \\
\,=\,&
\frac{1}{|x|}\, Q_T^{\beta}(x)
\,+\,
\frac{1}{4\pi\beta |x| \sqrt{\pi T}}
\int_{0}^{\infty}\big(e^{4\pi\beta w}\, -\, 1\big)\,
\frac{|x|+w}{2T}\,
e^{-\frac{(|x|+w)^2}{4T}}\,dw \nonumber  \\ 
\,\le\,&
\frac{1}{|x|}\,Q_t^\beta(x)
\,+\, \frac{1}{|x|\sqrt{\pi T}}\, \frac{1}{2T}\, \int_0^\infty\, w(|x|+w)\, e^{4\pi\beta w}\, e^{-\frac{(|x|+w)^2}{4T}}\,dw \nonumber \\
\,=\,&
\frac{1}{|x|}\,Q_t^\beta(x)
\,+\, \frac{e^{-4\pi\beta|x|+16\pi^2\beta^2T}}{|x|\sqrt{\pi T}}\,
\frac{1}{2T} \int_0^\infty w(|x|+w)\, e^{-\frac{(|x|-8\pi\beta T+w)^2}{4T}}\,dw
\end{align}
The inequality uses the elementary bound $e^u-1\le u e^u$ for $u\ge0$ with $u=4\pi\beta w$. The final equality uses~\eqref{ExpPowerIdentity}. Next, applying the change of variables $u=w-(8\pi\beta T-|x|)$, we obtain the first equality below.
\begin{align*}
\int_0^\infty
w(|x|+w)\,
e^{-\frac{(|x|-8\pi\beta T+w)^2}{4T}}\,dw
&=
\int_{-(8\pi\beta T-|x|)}^\infty
\Big(u+(8\pi\beta T-|x|)\Big)\Big(u+8\pi\beta T\Big)\,
e^{-\frac{u^2}{4T}}\,du \\
&\le
\int_{\mathbb R}
\Big(|u|+\big|8\pi\beta T-|x|\big|\Big)\Big(|u|+8\pi\beta T\Big)\,
e^{-\frac{u^2}{4T}}\,du \\
&=
\int_{\mathbb R}
\Big(u^2+\big(\big|8\pi\beta T-|x|\big|+8\pi\beta T\big)|u|
+\big|8\pi\beta T-|x|\big|\,(8\pi\beta T)\Big)\,
e^{-\frac{u^2}{4T}}\,du 
\end{align*}
Using the explicit Gaussian integrals $\int_{\mathbb R} e^{-\frac{u^2}{4T}}\,du=2\sqrt{\pi T}$, $\int_{\mathbb R} |u|\,e^{-\frac{u^2}{4T}}\,du=4T$, and $\int_{\mathbb R} u^2\,e^{-\frac{u^2}{4T}}\,du=4\sqrt{\pi}\,T^{\frac{3}{2}}$, we obtain
\begin{align*}
\int_0^\infty
w(|x|+w)\,
e^{-\frac{(|x|-8\pi\beta T+w)^2}{4T}}\,dw
&\le
4\sqrt{\pi}\,T^{\frac{3}{2}}
\;+\;
4T\big(\big|8\pi\beta T-|x|\big|+8\pi\beta T\big)
\;+\;
2\sqrt{\pi T}\,\big|8\pi\beta T-|x|\big|\,(8\pi\beta T).
\end{align*}
Substituting this back in~\eqref{NablaHUpper} gives
\begin{align*}
\big| \nabla \,Q_{ T}^{\beta}(x)\big|
\,\le\,&
\frac{1}{|x|}\,Q_t^\beta(x)
\,+\, 
\frac{e^{-4\pi\beta|x|+16\pi^2\beta^2T}}{|x|}
\Bigg[
\frac{2}{|x|}
+\frac{2}{\sqrt{\pi}}\frac{\big|8\pi\beta T-|x|\big|+8\pi\beta T}{|x|\sqrt{T}}
+\frac{\big|8\pi\beta T-|x|\big|\,(8\pi\beta T)}{|x|T}
\Bigg] \nonumber \\
\,\le\,&
\frac{1}{|x|}\,Q_t^\beta(x)
\,+\, 
\frac{e^{16\pi^2L^2} e^{-4\pi\beta|x|}}{|x|}\,
\Bigg[
\frac{2}{|x|}
+\frac{2}{\sqrt{\pi}}\frac{16\pi L\sqrt{T}+|x|}{|x|\sqrt{T}}
+\frac{64\pi^2L^2T+8\pi L\sqrt{T}|x|}{|x|T}
\Bigg] \nonumber \\
\;\le\;&
\frac{1}{|x|}Q_t^\beta(x)
\;+\;
C_L\,e^{-4\pi\beta|x|}
\bigg[
\frac{1}{|x|^2}
+\frac{1}{|x|\sqrt{T}}
\bigg]
\end{align*}
The second inequality uses $\big|8\pi\beta T-|x|\big|\le |x|+8\pi\beta T$ and $\beta \sqrt{T}\le L$. The final inequality holds by choosing a constant $C_L>0$ large enough. Finally, applying the upper bound~\eqref{QUpperBoundCL} and, if necessary, enlarging the constant $C_L$, we obtain the desired upper bound.
\end{proof}

\begin{corollary}
Fix $L>0$. There exists a constant $C_L>0$ such that for all $T, \beta>0$ with $\beta\sqrt T\le L$, and all $x\in\R^3\setminus\{0\}$, the gradient of $Q_T^\beta(x)$ satisfies the following upper bound.
\begin{align} \nonumber
\big|\nabla Q_T^\beta(x)\big|
\,\le\,
C_L
\bigg[
\frac{1+\sqrt{T}}{|x|^2}\,\mathbf 1_{\{|x|\le \sqrt{T}\}}
\;+\;
\Big(\frac{1}{|x|}+\frac{1}{|x|\sqrt{T}}\Big)\,\mathbf 1_{\{|x|>\sqrt{T}\}}
\bigg]
\, e^{-4\pi\beta|x|}
\end{align}
\end{corollary}

\begin{proof}
By Lemma~\eqref{LemmaHUpDown}(iii), there exists a constant $C_L$  such that
\begin{align} \nonumber
\big|\nabla Q_T^\beta(x)\big|
\,\le\,&
C_L\,
\bigg[
\frac{1}{|x|}
+\frac{1+\sqrt{T}}{|x|^2}
+\frac{1}{|x|\sqrt{T}}
\bigg]\,
e^{-4\pi\beta|x|} \,.
\end{align}
In the small regime $|x|\le \sqrt{T}$ we have $\frac{1}{|x|}\le \frac{\sqrt{T}}{|x|^2}$ and
$\frac{1}{|x|\sqrt{T}}\le \frac{1}{|x|^2}$, hence
\[
\frac{1}{|x|}
+\frac{1+\sqrt{T}}{|x|^2}
+\frac{1}{|x|\sqrt{T}}
\;\le\;
\frac{\sqrt{T}}{|x|^2}
+\frac{1+\sqrt{T}}{|x|^2}
+\frac{1}{|x|^2}
\,=\,
\Big(2\sqrt{T}+2\Big)\frac{1}{|x|^2}
\;\le\;
4\,(1+\sqrt{T})\,\frac{1}{|x|^2}.
\]
In the large regime $|x|>\sqrt{T}$ we have $\frac{1}{|x|^2}\le \frac{1}{|x|\sqrt{T}}$ and
$\frac{\sqrt{T}}{|x|^2}\le \frac{1}{|x|}$, hence
\[
\frac{1}{|x|}
+\frac{1+\sqrt{T}}{|x|^2}
+\frac{1}{|x|\sqrt{T}}
\,=\,
\frac{1}{|x|}
+\Big(\frac{1}{|x|^2}+\frac{\sqrt{T}}{|x|^2}\Big)
+\frac{1}{|x|\sqrt{T}}
\;\le\;
2\,\frac{1}{|x|}
+2\,\frac{1}{|x|\sqrt{T}}\,.
\]
Combining the small and large regimes, and enlarging $C_L$ if necessary, we obtain the desired bound.
\end{proof}

\begin{lemma}\label{SubProofLemTranKernTechinal} Fix $T,\beta>0$.  For all $0\le s<t\le T$ and $x\in\R^3 \setminus \{0\}$,
\begin{align} \nonumber
\int_{\R^3}\, P_{t-s}^{\beta}(x, y)\,\big(1+Q_{T-t}^{\beta}(y)\big)\,dy\, < \, \infty\,.
\end{align}
\end{lemma}

\begin{proof}
Fix $T,\beta>0$ and $0\le s<t\le T$. Set $\tau:=t-s\in(0,T]$ and fix
$x\in\R^3\setminus\{0\}$. Using~\eqref{PtBetaBounds} we obtain the inequality below.
\begin{align}
\int_{\R^3}\, P_{t-s}^{\beta}(x, y)\,\big(1+Q_{T-t}^{\beta}(y)\big)\,dy\,
\,\le\,&
\int_{\R^3} P_{\tau}(x,y) \big(1+Q_{T-t}^{\beta}(y)\big)dy
+
C\,\tau^{-1/2} \frac{e^{-\frac{|x|^2}{4\tau}}}{|x|}\, 
\int_{\R^3}\,\frac{e^{-\frac{|y|^2}{4\tau}}}{|y|} \big(1+Q_{T-t}^{\beta}(y)\big)dy\,\nonumber 
\end{align}
Denote the two integrals on the right-hand side by $I_1$ and $I_2$, respectively. \vspace{.2cm}

\noindent \textit{First integral.} Set $L:=\beta\sqrt T$. By~\eqref{QUpperBoundCL}, applied with $T$ replaced by
$T-t$, we obtain the first inequality below.
\begin{align}
I_1 \,\le\,  \int_{\R^3} P_{\tau}(x,y) \, \Big(1+C_L\big(1+\tfrac{\sqrt{T}}{|y|}\big)\Big)\,e^{-4\pi\beta|y|}\, dy 
\,\le\,
\int_{\R^3}\, P_{\tau}(x,y) \, \Big(1+C_L\big(1+\tfrac{\sqrt{T}}{|y|}\big)\Big)\,dy \nonumber
\end{align}
The second inequality uses that $e^{-4\pi\beta|y|}\le 1$. We now split $\R^3=B(0,1)\cup B(0,1)^c$. On $B(0,1)$, using $P_\tau(x,y)\le (4\pi\tau)^{-3/2}$, we obtain the first inequality below.
\begin{align*}
\int_{|y|\le 1}\, P_{\tau}(x,y) \, \Big(1+C_L\big(1+\tfrac{\sqrt{T}}{|y|}\big)\Big)\,dy
\, \le \,& 
(4\pi\tau)^{-\frac{3}{2}}\int_{|y|\le 1} \, \Big(1+C_L\big(1+\tfrac{\sqrt{T}}{|y|}\big)\Big)\,dy \nonumber \\
\, = \,& 
4\pi(4\pi\tau)^{-\frac{3}{2}}\int_{0}^1 \, \Big(r+C_L\big(1+ \sqrt{T}\big)\Big)\,rdr 
\,<\, 
\infty \,, 
\end{align*}
where the second last equality uses the spherical coordinates in $\R^3$. On $B(0,1)^c$, we have $|y|>1$, which implies
\begin{align*}
\int_{|y| > 1}\, P_{\tau}(x,y) \, \Big(1+C_L\big(1+\tfrac{\sqrt{T}}{|y|}\big)\Big)\,dy
\, \le \, 
\Big(1+C_L\big(1+ \sqrt{T}\big)\Big)\,\int_{|y| > 1}\, P_{\tau}(x,y) \,  dy 
\, = \,
\Big(1+C_L\big(1+ \sqrt{T}\big)\Big)
\,<\, \infty\,,
\end{align*}
where we have used that $\int_{\R^3}P_\tau(x,y)\,dy=1$ and
$e^{-4\pi\beta|y|}\le 1$. Combining the two estimates gives $I_1<\infty$. \vspace{.2cm}

\noindent \textit{Second integral.} Again using~\eqref{QUpperBoundCL} with $L:=\beta\sqrt T$, we obtain the inequality below.
\[
I_2\,:=\,\int_{\R^3} \frac{e^{-\frac{|y|^2}{4\tau}}}{|y|}\,\big(1+Q_{T-t}^{\beta}(y)\big)\,dy
\le
\int_{\R^3} \frac{e^{-\frac{|y|^2}{4\tau}}}{|y|}\, \Big(1+C_L\big(1+\tfrac{\sqrt{T}}{|y|}\big)\Big)\, dy .
\]
We split again into $|y|\le 1$ and $|y|>1$. On $B(0,1)$, using $e^{-\frac{|y|^2}{4\tau}} \le 1$ and enlarging the constant, we obtain
\[
\int_{|y|\le 1} \frac{e^{-\frac{|y|^2}{4\tau}}}{|y|}\, \Big(1+C_L\big(1+\tfrac{\sqrt{T}}{|y|}\big)\Big)\, dy
\;\le\;
C\int_{|y|\le 1}\Big(\frac{1}{|y|}+\frac{1}{|y|^2}\Big)\,dy
\;<\;\infty,
\]
since in $d=3$ both $\int_{|y|\le 1}|y|^{-1}dy$ and
$\int_{|y|\le 1}|y|^{-2}dy$ are finite. On $B(0,1)^c$, using $|y|>1$, we obtain the first inequality below.
\begin{align}
\int_{|y| > 1} \frac{e^{-\frac{|y|^2}{4\tau}}}{|y|}\, \Big(1+C_L\big(1+\tfrac{\sqrt{T}}{|y|}\big)\Big)\, dy
\,\le\, 
\Big(1+C_L(1+\sqrt{T})\Big)  \int_{|y|>1}\, e^{-\frac{|y|^2}{4\tau}}
\, dy 
\, \le \, 
(1\,+\,C)\, \int_{\R^3}\, e^{-\frac{|y|^2}{4\tau}}\,dy \nonumber
\end{align}
The final inequality holds by enlarging the constant. The right-hand side above is finite because the Gaussian factor $e^{-|y|^2/(4\tau)}$ is integrable on $\R^3$.

Combining the above estimates, we obtain $I_1<\infty$ and $I_2<\infty$, and hence the desired inequality follows.
\end{proof}

\subsection{The function \texorpdfstring{$\boldsymbol{p_{s,t}^{T,\beta}(x,y)}$}{Lg}}\label{SubsectDoobpFunct}

The following lemma records quantitative boundary-avoidance properties of the one- and two-step transition densities associated with $p^{T,\beta}_{s,t}(x,y)$ given in~\eqref{FirstTrans3d},
showing that the mass entering the killing region is exponentially small when the starting point remains away from the boundary.

\begin{lemma}
Fix $T,\beta,\varepsilon>0$ and $\delta>0$. There exist constants
$C,c>0$, depending only on $(T,\beta,\varepsilon,\delta)$, such that the following hold.

\begin{enumerate} [(i)]
    \item For every $0\le s<t\le T$ and every $x\in E_\varepsilon$ satisfying
$|x|-\varepsilon\ge \delta$, we have
\begin{align}\label{BoundOneStepBadMassAwayFromBoundary}
\int_{\{|y|\le \varepsilon\}}
p^{T,\beta}_{s,t}(x,y)\,dy
\;\le\;
C\,e^{-c/(t-s)}\,.
\end{align}

\item For every $0\le r<s<t\le T$ and every $x\in E_\varepsilon$ with
$|x|-\varepsilon\ge \delta$, we have
\begin{align} \label{BoundTwoStepBadMassAwayFromBoundary}
\int_{\{0<|y|-\varepsilon<\delta\}}
\int_{\{|z|\le \varepsilon\}}
p^{T,\beta}_{r,s}(x,y)\,
p^{T,\beta}_{s,t}(y,z)\,dz\,dy
\;\le\;
C\,e^{-c/(t-r)}\,.
\end{align}
\end{enumerate}
\end{lemma}

\begin{proof}
\noindent Part (i).  Let $\tau:=t-s$. Then, by the definition~\eqref{FirstTrans3d}, we have
\begin{align*}
\int_{\{|y|\le \varepsilon\}}
p^{\,T,\beta}_{s,t}(x,y)\,dy
=
\frac{1}{1+Q^{\beta}_{T-s}(x)}
\int_{\{|y|\le \varepsilon\}}
\bigl(1+Q^{\beta}_{T-t}(y)\bigr)\,
P^{\beta}_{\tau}(x,y)\,dy 
\le
\int_{\{|y|\le \varepsilon\}}
\bigl(1+Q^{\beta}_{T-t}(y)\bigr)\,
P^{\beta}_{\tau}(x,y)\,dy,
\end{align*}
where the inequality uses that $1+Q^{\beta}_{T-s}(x) \ge 1$. Next, using~\eqref{PtBetaBounds} with $t=\tau$, we obtain
\begin{align*}
\int_{{|y|\le \varepsilon}}
d^{T,\beta}_{s,t}(x,y)\,dy
\;\le\;
\int_{\{|y|\le \varepsilon\}}
(1+Q^{\beta}_{T-t}(y))\,P_{\tau}(x,y)\,dy
+
C\,\tau^{-1/2}\,
\frac{e^{-\frac{|x|^2}{4\tau}}}{|x|}
\int_{\{|y|\le \varepsilon\}}
(1+Q^{\beta}_{T-t}(y))\,
\frac{e^{-\frac{|y|^2}{4\tau}}}{|y|}\,dy,
\end{align*}
where $C=C(T,\beta)>0$. Let us name the terms on the right side above by $I_1$ and $I_2$. \vspace{.3cm}

\noindent \textit{Estimate of \(I_1\).} For $|y|\le \varepsilon$ and $|x|-\varepsilon\ge \delta$, we have $|x-y|\;\ge\;|x|-|y|\;\ge\;\varepsilon+\delta-\varepsilon\,=\,\delta$. Hence, using the definition of the free heat kernel~\eqref{DefFreeHeat3d}, we have $P_{\tau}(x,y) := (4\pi\tau)^{-\frac{3}{2}} e^{-\frac{|x-y|^2}{4\tau}}
\le (4\pi\tau)^{-\frac{3}{2}} e^{-\frac{\delta^2}{4\tau}}$. Therefore, we have the inequality below.
\begin{align*}
I_1
\,:=\,
\int_{\{|y|\le \varepsilon\}}
(1+Q^{\beta}_{T-t}(y))\,P_{\tau}(x,y)\,dy
\le
\frac{1}{(4\pi\tau)^{\frac{3}{2}}}
e^{-\frac{\delta^2}{4\tau}}
\int_{\{|y|\le \varepsilon\}}
\bigl(1+Q^{\beta}_{T-t}(y)\bigr)\,dy .
\end{align*}
Set $L:=\beta \sqrt{T}$. Then by~\eqref{QUpperBoundCL}, there exists a constant $C'=C_L'>0$ such that for $|y|\le \varepsilon$, we have $1+Q^{\beta}_{T-t}(y) \le C'(1+|y|^{-1})$. Therefore, we have the inequality below.
\begin{align*}
I_1
\le
\frac{C' e^{-\frac{\delta^2}{4\tau}}}{(4\pi\tau)^{\frac{3}{2}}}
\int_{\{|y|\le \varepsilon\}}
\Big(1+\frac{1}{|y|}\Big)\,dy
=
\frac{C' e^{-\frac{\delta^2}{4\tau}}}{(4\pi\tau)^{\frac{3}{2}}}
\Bigl(
\frac{4\pi}{3}\varepsilon^3
+
2\pi \varepsilon^2
\Bigr)
\le
C\,\tau^{-\frac{3}{2}}\,e^{-\frac{\delta^2}{4\tau}}
\end{align*}
The final inequality follows from absorbing the $\varepsilon$-dependent factor into the constant to conclude that there exists $C>0$ such that the inequality follows. \vspace{.3cm}

\noindent \textit{Estimate of \(I_2\).} Again for $L:=\beta \sqrt{T}$ there exists a constant $C=C_L>0$ from~\eqref{QUpperBoundCL} such that
\begin{align*}
\int_{\{|y|\le \varepsilon\}}
(1+Q^{\beta}_{T-t}(y))\,
\frac{e^{-\frac{|y|^2}{4\tau}}}{|y|}\,dy
\le
C
\int_{\{|y|\le \varepsilon\}}
\Big(\frac{1}{|y|}+\frac{1}{|y|^2}\Big)
e^{-\frac{|y|^2}{4\tau}}\,dy
=
4\pi C
\int_0^\varepsilon
\bigl(r+1\bigr)e^{-\frac{r^2}{4\tau}}\,dr.
\end{align*}
Where the equality follows by passing into spherical coordinates in $\mathbb R^3$, and writing $r=|y|$ and $dy=4\pi r^2\,dr$. Since $r\mapsto (r+1)e^{-r^2/(4\tau)}$ is bounded and integrable on $[0,\varepsilon]$, the right side above is bounded by a constant  $C$ that depends only on $(T,\beta,\varepsilon)$.
Therefore, we have the first inequality below.
\begin{align*}
I_2
\,:=\,
C\,\tau^{-1/2}\,
\frac{e^{-\frac{|x|^2}{4\tau}}}{|x|}
\int_{\{|y|\le \varepsilon\}}
(1+Q^{\beta}_{T-t}(y))\,
\frac{e^{-\frac{|y|^2}{4\tau}}}{|y|}\,dy
\le
C\,\tau^{-1/2}\,
\frac{e^{-\frac{|x|^2}{4\tau}}}{|x|}
\le
C\,\tau^{-1/2}\,
e^{-\frac{(\varepsilon+\delta)^2}{4\tau}}
\end{align*}
The final inequality uses that $\frac{1}{|x|}e^{-\frac{|x|^2}{4\tau}}
\le \frac{1}{\varepsilon+\delta} e^{-\frac{(\varepsilon+\delta)^2}{4\tau}}$ since $|x|-\varepsilon\ge \delta$.

Combining the bounds for $I_1$ and $I_2$, we obtain the first inequality below.
\begin{align*}
\int_{\{|y|\le \varepsilon\}}
p^{\,T,\beta}_{s,t}(x,y)\,dy
&\le
C_1\,\tau^{-\frac{3}{2}}e^{-\frac{\delta^2}{4\tau}}
+
C_2\,\tau^{-1/2}e^{-\frac{(\varepsilon+\delta)^2}{4\tau}} \\
&\le
C\Big(\tau^{-\frac{3}{2}}+\tau^{-1/2}\Big)e^{-\frac{\delta^2}{4\tau}} 
=
C\Big[\big(\tau^{-\frac{3}{2}}+\tau^{-1/2}\big)e^{-\frac{\delta^2}{8\tau}}\Big]
e^{-\frac{\delta^2}{8\tau}}
\end{align*}
The second inequality holds because $\frac{(\varepsilon+\delta)^2}{4}\ge \frac{\delta^2}{4}$ and therefore $e^{-\frac{(\varepsilon+\delta)^2}{4\tau}} \le e^{-\frac{\delta^2}{4\tau}}$. The last equality uses the identity $e^{-\frac{\delta^2}{4\tau}} = e^{-\frac{\delta^2}{8\tau}}\,e^{-\frac{\delta^2}{8\tau}}$. Now, since $\tau=t-s\in(0,T]$, the function $u\longmapsto \bigl(u^{-\frac{3}{2}}+u^{-1/2}\bigr)e^{-\frac{\delta^2}{8u}}$ is bounded on $(0,T]$. Hence there exists a constant $C_3>0$ such that
\[
\Big(\tau^{-\frac{3}{2}}+\tau^{-1/2}\Big)e^{-\frac{\delta^2}{8\tau}}
\le
C_3.
\]
Therefore, after enlarging the constant if necessary,
\begin{align*}
\int_{\{|y|\le \varepsilon\}}
p^{\,T,\beta}_{s,t}(x,y)\,dy
&\le
C\,e^{-\frac{\delta^2}{8\tau}}
=
C\,e^{-c/(t-s)},
\end{align*}
where $c:=\delta^2/8$ and recall that $\tau:=t-s$. \vspace{.3cm}

\noindent Part (ii). Since the integrand is nonnegative, we may enlarge the domain of integration in \(y\) and then interchange the order of integration to obtain
\begin{align*}
\int_{\{0<|y|-\varepsilon<\delta\}}
\int_{|z| \le \varepsilon}
p^{T,\beta}_{r,s}(x,y)\,
p^{T,\beta}_{s,t}(y,z)\,dz\,dy
&\le
\int_{\mathbb R^3}
\int_{|z| \le \varepsilon}
p^{T,\beta}_{r,s}(x,y)\,
p^{T,\beta}_{s,t}(y,z)\,dz\,dy \\
&=
\int_{|z| \le \varepsilon}
\left(
\int_{\mathbb R^3}
p^{T,\beta}_{r,s}(x,y)\,
p^{T,\beta}_{s,t}(y,z)\,dy
\right)\,dz 
=
\int_{|z| \le \varepsilon}
p^{T,\beta}_{r,t}(x,z)\,dz ,
\end{align*}
where the last equality uses the Chapman-Kolmogorov property of the Doob-transformed kernels \(p^{T,\beta}_{r,t}\). Applying~\eqref{BoundOneStepBadMassAwayFromBoundary} with the time interval \([r,t]\) therefore yields the desired bound.
\end{proof}

The lemma below records a complementary estimate controlling large spatial excursions of the one-step transition density
$p^{T,\beta}_{s,t}(x,y)$.

\begin{lemma} \label{LemOneStepTailBoundSimple}
Fix \(T,\beta,\varepsilon>0\). Then there exist constants \(C,c>0\),
depending only on \((T,\beta,\varepsilon)\), such that for every
\(0\le s<t\le T\), every \(R\ge \varepsilon\) and every \(x\in E_\varepsilon\) with $|x| \le R$, we have
\begin{align}\label{EqOneStepTailBoundSimple}
\int_{\{|y|>R\}}
p^{T,\beta}_{s,t}(x,y)\,dy
\;\le\;
C\,
\exp\!\Bigl(
-\frac{c\,(R-|x|)^2}{t-s}
\Bigr).
\end{align}
\end{lemma}

\begin{proof}
Let $\tau:=t-s$. By the definition~\eqref{FirstTrans3d}, we have
\begin{align*}
\int_{\{|y|>R\}}
p^{\,T,\beta}_{s,t}(x,y)\,dy
=
\frac{1}{1+Q^{\beta}_{T-s}(x)}
\int_{\{|y|>R\}}
\bigl(1+Q^{\beta}_{T-t}(y)\bigr)\,
P^{\beta}_{\tau}(x,y)\,dy 
\le
\int_{\{|y|>R\}}
\bigl(1+Q^{\beta}_{T-t}(y)\bigr)\,
P^{\beta}_{\tau}(x,y)\,dy,
\end{align*}
where we used that $1+Q^{\beta}_{T-s}(x)\ge1$. Using~\eqref{PtBetaBounds}, we obtain
\begin{align*}
\int_{\{|y|>R\}}
p^{T,\beta}_{s,t}(x,y)\,dy
\;\le\;
\int_{\{|y|>R\}}
(1+Q^{\beta}_{T-t}(y))\,P_{\tau}(x,y)\,dy
+
C\,\tau^{-1/2}\,
\frac{e^{-\frac{|x|^2}{4\tau}}}{|x|}
\int_{\{|y|>R\}}
(1+Q^{\beta}_{T-t}(y))\,
\frac{e^{-\frac{|y|^2}{4\tau}}}{|y|}\,dy.
\end{align*}
We denote the two terms by $I_1$ and $I_2$.  \vspace{.3cm}

\noindent \textit{Estimate of \(I_1\).} For $L:=\beta \sqrt T$, there exists a constant \(C'>0\) by~\eqref{QUpperBoundCL} such that $1+Q^{\beta}_{T-t}(y)\le C'(1+|y|^{-1})$ for $y\in E_\varepsilon$. Hence, we have the first inequality below.
\begin{align*}
I_1
\le
C'
\int_{\{|y|>R\}}
\Bigl(1+\frac{1}{|y|}\Bigr)\,P_{\tau}(x,y)\,dy 
=
C'
\int_{\{|y|>R\}} P_{\tau}(x,y)\,dy
+
C'
\int_{\{|y|>R\}} \frac{1}{|y|}\,P_{\tau}(x,y)\,dy 
=
I_{1}^{(1)}+I_{1}^{(2)}
\end{align*}
For the first term, since \(P_\tau(x,y)=P_\tau(y-x)\), we may write
\begin{align*}
I_{1}^{(1)}
&\,:=\,
\int_{\{|y|>R\}} P_{\tau}(y-x)\,dy
=
\int_{\{|z+x|>R\}} P_{\tau}(z)\,dz 
\le
\int_{\{|z|>R-|x|\}} P_{\tau}(z)\,dz
=
\frac{1}{(4\pi\tau)^{\frac{3}{2}}}
\int_{\{|z|>R-|x|\}} e^{-\frac{|z|^2}{4\tau}}\,dz 
\end{align*}
The inequality follows because if \(|x|\le R\) and \(|z+x|>R\), then necessarily \(|z|>R-|x|\) since $|z|+|x|\ge |z+x|>R$. The final equality uses the definition of the free heat kernel~\eqref{DefFreeHeat3d}. Passing to spherical coordinates the right side may be written as
\begin{align*}
\frac{4\pi}{(4\pi\tau)^{\frac{3}{2}}}
\int_{R-|x|}^{\infty} r^2 e^{-\frac{r^2}{4\tau}}\,dr 
=
\frac{4\pi}{(4\pi)^{\frac{3}{2}}}
\int_{\frac{R-|x|}{\sqrt{\tau}}}^{\infty}
u^2 e^{-\frac{u^2}{4}}\,du,
\end{align*}
where the equality uses the change of variables \(r=\sqrt{\tau}\,u\). Observe that for every \(a\ge 0\) and every \(u\ge a\), we have $e^{-\frac{u^2}{4}} = e^{-\frac{u^2}{8}}e^{-\frac{u^2}{8}} \le e^{-\frac{a^2}{8}}e^{-\frac{u^2}{8}}$. Therefore, we have the first inequality below.
\begin{align*}
e^{\frac{a^2}{8}}\int_a^\infty u^2 e^{-\frac{u^2}{4}}\,du
\le
e^{\frac{a^2}{8}}e^{-\frac{a^2}{8}}
\int_a^\infty u^2 e^{-\frac{u^2}{8}}\,du 
\le
\int_0^\infty u^2 e^{-\frac{u^2}{8}}\,du
\end{align*}
Since the last integral is finite, it follows that the function $a \mapsto e^{\frac{a^2}{8}}\int_a^\infty u^2 e^{-\frac{u^2}{4}}\,du$ is bounded on \([0,\infty)\), and hence there exists a constant \(M>0\) such that
\begin{align} \label{IntBoundinProof}
    \int_a^\infty u^2 e^{-\frac{u^2}{4}}\,du
\le
M e^{-\frac{a^2}{8}}
\qquad\text{for all } a\ge 0.
\end{align}
Set $a:=\frac{R-|x|}{\sqrt{\tau}}\ge 0$. Then from the above computations we have the first inequality below.
\[
I_{1}^{(1)}
\le
\frac{4\pi}{(4\pi)^{\frac{3}{2}}}
\int_a^\infty u^2 e^{-\frac{u^2}{4}}\,du
\le
\frac{4\pi}{(4\pi)^{\frac{3}{2}}}\, M\, e^{-\frac{a^2}{8}}
\le
C_1\exp\!\Bigl(-\frac{a^2}{8}\Bigr)
=
C_1\exp\!\Bigl(-\frac{(R-|x|)^2}{8\tau}\Bigr)
\]
The second inequality uses~\eqref{IntBoundinProof} and the final inequality holds after absorbing the prefactor into the constant.

For the second term, we use that \(|y|>R\ge \varepsilon\). Hence, 
\begin{align*}
I_{1}^{(2)}
\,:=\,
\int_{\{|y|>R\}} \frac{1}{|y|}\,P_{\tau}(x,y)\,dy
\le
\frac{1}{\varepsilon}
\int_{\{|y|>R\}} P_{\tau}(x,y)\,dy
=
\frac{1}{\varepsilon}\,I_{1}^{(1)}
\le
C'_1\,
\exp\!\Bigl(
-\frac{(R-|x|)^2}{8\tau}
\Bigr),
\end{align*}

Combining the bounds for \(I_{1}^{(1)}\) and \(I_{1}^{(2)}\), we conclude that
\begin{align*}
I_1
\le
C\,
\exp\!\Bigl(
-\frac{(R-|x|)^2}{8\tau}
\Bigr)
=
C\,
\exp\!\Bigl(
-\frac{c\,(R-|x|)^2}{t-s}
\Bigr)
\end{align*}

\noindent \textit{Estimate of \(I_2\).} Using again~\eqref{QUpperBoundCL} with $L:=\beta \sqrt{ T}$, we obtain the first inequality below.
\begin{align*}
\int_{\{|y|>R\}}
(1+Q^{\beta}_{T-t}(y))\,
\frac{e^{-\frac{|y|^2}{4\tau}}}{|y|}\,dy
&\le
C
\int_{\{|y|>R\}}
\Bigl(\frac{1}{|y|}+\frac{1}{|y|^2}\Bigr)
e^{-\frac{|y|^2}{4\tau}}\,dy \nonumber \\
&=
4 \pi C\int_R^\infty (r+1)e^{-\frac{r^2}{4\tau}}\,dr
\le
4 \pi C \frac{R+1}{R}
\int_R^\infty r\,e^{-\frac{r^2}{4\tau}}\,dr
=
4 \pi C\frac{R+1}{R}2\tau\,e^{-\frac{R^2}{4\tau}}
\end{align*}
The first equality follows form passing to spherical coordinates. The second inequality uses that for \(r\ge R\), we have $r+1 \le (R+1)\frac{r}{R}$. Therefore, by absorbing the constants we obtain the first inequality below.
\begin{align*}
I_2
&\,:=\,
C\,\tau^{-1/2}\,
\frac{e^{-\frac{|x|^2}{4\tau}}}{|x|}
\int_{\{|y|>R\}}
(1+Q^{\beta}_{T-t}(y))\,
\frac{e^{-\frac{|y|^2}{4\tau}}}{|y|}\,dy \nonumber \\
&\le
C\,\tau^{1/2}\,
\frac{R+1}{R}
\frac{1}{|x|}
\exp\!\Bigl(
-\frac{R^2+|x|^2}{4\tau}
\Bigr)
\le
C\,\tau^{1/2}\,
\exp\!\Bigl(
-\frac{(R-|x|)^2}{4\tau}
\Bigr)
\le
C\,
\exp\!\Bigl(
-\frac{c\,(R-|x|)^2}{t-s}
\Bigr)
\end{align*}
The second last inequality uses that \(|x| > \varepsilon\) since $x \in E_{\varepsilon}$, $R  \ge \varepsilon$ (given), and \(R^2+|x|^2 \ge (R-|x|)^2\). The final inequality holds by absorbing the factor \(\tau^{1/2}\) into the constant since \(\tau\in(0,T]\). Thus, combining the bounds for \(I_1\) and \(I_2\) completes the proof.
\end{proof}

The lemma below records a moment bound for the one-step transition density
$p^{T,\beta}_{s,t}(x,y)$, providing quantitative control on its
short-time fluctuations.

\begin{lemma}\label{LemmaKolmogorov3D}
Fix $T,\beta>0$ and $x\in\R^3\setminus\{0\}$. For every $L>0$ and
$m\in\mathbb N$ with $\beta\sqrt{T}\le L$, there exists a constant
$C=C(T,\beta,L,m)>0$ such that for all $0\le s<t\le T$,
\begin{align}\label{EquationKolmogorov3D}
\int_{\R^3}p^{\,T,\beta}_{s,t}(x,y)\,|y-x|^{2m}\,dy
\;\le\;
\frac{C}{|x|}\,(t-s)^{m}.
\end{align}
Moreover, for every $\varepsilon>0$, there exists a constant
$C=C(T,\beta,L,m,\varepsilon)>0$ such that
\begin{align*}
\int_{\R^3}p^{\,T,\beta}_{s,t}(x,y)\,|y-x|^{2m}\,dy
\;\le\;
C\,(t-s)^{m}
\end{align*}
for all $0\le s<t\le T$ and all $x\in\R^3$ with $|x|\ge\varepsilon$.
\end{lemma}

\begin{proof}
Fix $0\le s<t\le T$. Since the transition density
$\mathlarger{p}_{s,t}^{T,\beta}(x,y)$ defined in~\eqref{FirstTrans3d} satisfies $\mathlarger{p}_{s,t}^{T,\beta}
=\mathlarger{p}_{0,t-s}^{T-s,\beta}$, it suffices to prove the claim for $s=0$. Using~\eqref{FirstTrans3d} with $s=0$ and $1+Q_t^\beta(x) \ge 1$ we obtain the first inequality below.
\begin{align}\label{Kolm3D}
\int_{\R^3}  \, \mathlarger{p}_{0,t}^{T,\beta}(x,y) 
\big|y -   x \big|^{2m}\, dy
\,\le\,&
\int_{\R^3}
\big|y-x\big|^{2m}\,\big(1+Q_{T-t}^\beta(y)\big)P_t^\beta(x,y)\,dy \nonumber \\
\,=\,&
\int_{\R^3}
\big|y-x\big|^{2m}P_t^\beta(x,y)\,dy \,+\,
\int_{\R^3}
\big|y-x\big|^{2m}\,Q_{T-t}^\beta(y)P_t^\beta(x,y)\,dy \nonumber \\
\,\le\,&
(1+C_L)\int_{\R^3}
\big|y-x\big|^{2m}P_t^\beta(x,y)\,dy 
\,+\,
C_L \sqrt{T} \int_{\R^3}|y-x|^{2m}\frac{1}{|y|}P_t^\beta(x,y)\,dy 
\end{align}
The last inequality uses~\eqref{QUpperBoundCL} and $\sqrt{T-t} \le \sqrt{T}$. Let us label the integrals on the right side above by $I_1$ and $I_2$, respectively. By~\eqref{PtBetaBounds} there exists a constant $C=C(\beta,T)>0$ such that the first inequality below holds.
\begin{align*}
I_1
\,:=\,&
\int_{\R^3}
\big|y-x\big|^{2m}\,P_t^\beta(x,y)\,dy \nonumber \\ 
\,\le\,&
\int_{\R^3}
\big|y-x\big|^{2m}\, P_t(x,y) \,dy
\,+\,
C\,t^{-1/2}\,\frac{e^{-\frac{|x|^2}{4t}}}{|x|}\,
\int_{\R^3}|y-x|^{2m}\,\frac{e^{-\frac{|y|^2}{4t}}}{|y|}\,dy \nonumber \\ 
\,\le\,&
(2t)^m\,(2m+1)!!
\,+\,
2^{2m}C\,t^{-1/2}\,\frac{e^{-\frac{|x|^2}{4t}}}{|x|}\,
\bigg[
\int_{\R^3}|y|^{2m-1}e^{-\frac{|y|^2}{4t}}\,dy
\;+\;
|x|^{2m}\int_{\R^3}\frac{e^{-\frac{|y|^2}{4t}}}{|y|}\,dy
\bigg] \nonumber
\end{align*}
The second equality uses that the Gaussian part is $\int_{\R^3} |y-x|^{2m} P_t(x,y) dy =(2t)^m\,(2m+1)!!$ where recall that double factorial is defined by $(2m+1)!! := 1\cdot 3\cdot 5\cdots (2m+1)$; the second term uses the inequality $|y-x|^{2m}\le 2^{2m}\big(|y|^{2m}+|x|^{2m}\big)$. Next using the following  explicit radial integrals,
\begin{align*}
\int_{\R^3}|y|^{2m-1}e^{-\frac{|y|^2}{4t}}\,dy
&=
4\pi\int_0^\infty r^{2m+1}e^{-\frac{r^2}{4t}}\,dr
=
2\pi(4t)^{m+1}\Gamma(m+1)
=
2\pi(4t)^{m+1}m!\,,\\
\int_{\R^3}\frac{e^{-\frac{|y|^2}{4t}}}{|y|}\,dy
&=
4\pi\int_0^\infty r\,e^{-\frac{r^2}{4t}}\,dr
=
8\pi t\,,
\end{align*}
we obtain the first inequality below.
\begin{align}
I_1
\le
(2t)^m\,(2m+1)!!
+
2^{2m}\,C\,t^{-1/2}\,\frac{e^{-\frac{|x|^2}{4t}}}{|x|}
\Big[2\pi(4t)^{m+1}m! \;+\; 8\pi t\,|x|^{2m}\Big] 
\le
C\,t^m
+
\frac{C}{|x|}\,t^{m}
+
C\,t^{\frac{1}{2}}\,|x|^{2m-1}e^{-\frac{|x|^2}{4t}}  \nonumber
\end{align}
Here, the second inequality follows by absorbing all numerical factors into a finite positive constant $C=C(T,\beta,L,m) $, and by using that $\sqrt{t}\le \sqrt{T}$
and $e^{-\frac{|x|^2}{4t}}\le 1$ for the middle term. Next, using the elementary
bound, for fixed $x\neq0$ and $t\in(0,T]$,
\[
|x|^{2m-1}e^{-\frac{|x|^2}{4t}}
\le
C_m\,t^{m-\frac12},
\]
which follows from $\sup_{u\ge0}u^{m-\frac12}e^{-u}<\infty$ with
$u=\frac{|x|^2}{4t}$, we obtain
\begin{align}\label{Kolm3D1}
I_1
&\le
C\,t^m
+
\frac{C}{|x|}\,t^{m}
+
C\,t^m 
\le
\frac{C}{|x|}\,t^{m},
\end{align}
where in the final inequality we enlarged the constant, using that $x\neq 0$
is fixed and $0<t\le T$.

Next, again using~\eqref{PtBetaBounds} we obtain the inequality below.
\begin{align*}
I_2
\,:=\,&
\int_{\R^3}|y-x|^{2m}\,\frac{1}{|y|}\,P_t^\beta(x,y)\,dy \nonumber \\
\,\le\,&
\int_{\R^3}|y-x|^{2m}\,\frac{1}{|y|}\,P_t(x,y)\,dy
\,+\,
C\,t^{-1/2}\,\frac{e^{-\frac{|x|^2}{4t}}}{|x|}
\int_{\R^3}|y-x|^{2m}\,\frac{e^{-\frac{|y|^2}{4t}}}{|y|^2}\,dy 
\,=:\,
I_{2}^{(1)}\;+\;I_{2}^{(2)}\,
\end{align*}

\noindent \emph{Bound on $I_{2}^{(1)}$.} We split $\R^3=A\cup B$ with $A:=\{|y|\ge |x|/2\}$ and $ B:=\{|y|<|x|/2\}$. On $A$ we have $|y|^{-1}\le 2|x|^{-1}$, which implies the first inequality below.
\begin{align*}
\int_A |y-x|^{2m}\frac{1}{|y|}P_t(x,y)\,dy
\le
\frac{2}{|x|}\int_{\R^3}|y-x|^{2m}P_t(x,y)\,dy
=
\frac{2}{|x|}(2t)^m(2m+1)!! 
\;\le\;
\frac{C}{|x|}\,t^m,
\end{align*}
where the final inequality follows by combining the numerical factors above into $C=C(T,\beta,L,m) $. On $B$ we have $|y-x|\le |y|+|x|\le \frac{3}{2}|x|$, and also
$|y-x|\ge |x|-|y|\ge |x|/2$, hence we have the inequality below.
\begin{align*}
\int_B |y-x|^{2m}\frac{1}{|y|}P_t(x,y)\,dy
\,=\,&
\frac{1}{(4\pi t)^{\frac{3}{2}}}\,\int_B |y-x|^{2m}\frac{1}{|y|}\,  e^{-\frac{|y-x|^2}{4t}}\,dy \nonumber \\
\,\le\,&
\frac{1}{(4\pi t)^{\frac{3}{2}}}\, \Big(\tfrac{3}{2}|x|\Big)^{2m}\, e^{-\frac{|x|^2}{16t}}
\int_{|y|<|x|/2}\frac{1}{|y|}\,dy \\
\,=\,&\frac{1}{(4\pi t)^{\frac{3}{2}}}\,
\Big(\tfrac{3}{2}|x|\Big)^{2m}\,e^{-\frac{|x|^2}{16t}}
\Big(4\pi\int_0^{|x|/2} r\,dr\Big) \\
\,=\,&
\frac{1}{(4\pi t)^{\frac{3}{2}}}\,\Big(\tfrac{3}{2}|x|\Big)^{2m} \, e^{-\frac{|x|^2}{16t}}
\, \frac{\pi}{2}|x|^2 
\,=\,
C\,|x|^{2m+2}\,t^{-\frac{3}{2}}\,e^{-\frac{|x|^2}{16t}}\,,
\end{align*}
where in the last step we have absorbed all numerical constants into $C$, possibly enlarging its value. By combining the above two cases we obtain the inequality below.
\begin{align}
I_{2}^{(1)}
\,:=\,
\int_{\R^3}|y-x|^{2m}\,\frac{1}{|y|}\,P_t(x,y)\,dy 
\,\le \,
\frac{C}{|x|}\,t^m
\,+\,
C\,|x|^{2m+2}\,t^{-\frac{3}{2}}\,e^{-\frac{|x|^2}{16t}}
\le
\frac{C}{|x|}\,t^m\,, \nonumber
\end{align}
where the last inequality uses the elementary estimate $|x|^{2m+3}e^{-\frac{|x|^2}{16t}} \le C_m\,t^{m+\frac32}$. \vspace{.2cm}

\noindent \emph{Bound on $I_{2}^{(2)}$.} Again using the inequality $|y-x|^{2m}\le  2^{2m} \big(|y|^{2m} + |x|^{2m} \big)$ we obtain the inequality below. 
\begin{align*}
I_{2}^{(2)}
\,:=\,&
C\,t^{-1/2}\,\frac{e^{-\frac{|x|^2}{4t}}}{|x|}
\int_{\R^3}|y-x|^{2m}\,\frac{e^{-\frac{|y|^2}{4t}}}{|y|^2}\,dy \nonumber \\
\,\le\,&
2^{2m}\,C\,t^{-1/2}\,\frac{e^{-\frac{|x|^2}{4t}}}{|x|}\,
\bigg[
\int_{\R^3}|y|^{2m-2}e^{-\frac{|y|^2}{4t}}\,dy
\;+\;
|x|^{2m}\int_{\R^3}\,\frac{e^{-\frac{|y|^2}{4t}}}{|y|^2}\,dy
\bigg] \nonumber \\
\,=\,&
2^{2m}\,C\,t^{-1/2}\,\frac{e^{-\frac{|x|^2}{4t}}}{|x|}
\, \Big[2\pi(4t)^{m+\frac12}\Gamma\!\big(m+\tfrac12\big)
\,+\,
|x|^{2m}\,4\pi\sqrt{\pi t}
\Big]
\end{align*}
The final equality uses that both the radial integrals below are explicit,
\begin{align*}
\int_{\R^3}|y|^{2m-2}e^{-\frac{|y|^2}{4t}}\,dy
\,=\,&
4\pi\int_0^\infty r^{2m}e^{-\frac{r^2}{4t}}\,dr
\,=\,
2\pi(4t)^{m+\frac12}\Gamma\!\big(m+\tfrac12\big)\,,\\
\int_{\R^3}|y|^{-2}e^{-\frac{|y|^2}{4t}}\,dy
\,=\,&
4\pi\int_0^\infty e^{-\frac{r^2}{4t}}\,dr
\,=\,
4\pi\sqrt{\pi t}\,.
\end{align*}
Absorbing the numerical factors and the quantity
$\Gamma\!\big(m+\tfrac12\big)$ into $C>0$, we obtain
\begin{align*}
I_{2}^{(2)}
\le
C\,t^{-1/2}\,\frac{e^{-\frac{|x|^2}{4t}}}{|x|}
\Big(
t^{m+\frac12}
+
|x|^{2m}\sqrt{t}
\Big) 
\le
C\,\frac{e^{-\frac{|x|^2}{4t}}}{|x|}
\Big(
t^{m}
+
|x|^{2m}
\Big)
\le
\frac{C}{|x|} \,t^m,
\end{align*}
where the final inequality uses the elementary bound $|x|^{2m}e^{-\frac{|x|^2}{4t}}\le C_m\,t^m$ and enlarging the constant if necessary. Combining the bounds of $I_{2}^{(1)}$ and $I_{2}^{(2)}$, we obtain
\begin{align*}
I_2
\,\le\,
I_{2}^{(1)}\;+\;I_{2}^{(2)}
\, \le \,
\frac{C}{|x|} \,t^m,
\end{align*}
Combining the above estimate with~\eqref{Kolm3D1} and substituting into~\eqref{Kolm3D}, we obtain~\eqref{EquationKolmogorov3D} for the case $s=0$, which completes the proof.
\end{proof}

\begin{proof}[Proof of Lemma~\ref{LemTranKern}]  \label{SubProofLemTranKern}
Recall that the function $p_{s,t}^{T,\beta}$ for $0\leq s<t\leq T$ is defined in~(\ref{FirstTrans3d}). \vspace{.2cm}

\noindent Part (i). For $x\in\R^3\setminus\{0\}$, by Lemma~\ref{SubProofLemTranKernTechinal} the numerator below is finite. Moreover,  the Lebesgue integral of $p_{s,t}^{T,\beta}(x,y)$ over $y\in \R^3$ is one since
\begin{align}\label{CheckdTwo}
\int_{\R^3} \,p_{s,t}^{T,\beta }(x,y)\,dy
\,=\,
\frac{1}{ 1+ Q_{T-s}^{\beta}(x) } \int_{\R^3}\, P_{t-s}^{\beta}(x, y)\,\big(1+Q_{T-t}^{\beta}(y)\big)\,dy
\,=\, \frac{1+ Q_{T-s}^{\beta}(x)}{ 1+ Q_{T-s}^{\beta}(x) } \,=\,1\,, 
 \end{align}
The second equality above uses that $1+Q_{r}^{\beta}(x)=\int_{\R^3}P_{r}^{\beta}(x,y)dy$ by~(\ref{DefH3d}) and the semigroup property of the family $\{P_{r}^{\beta}(x,y)\}_{r\in [0,\infty)}$.  \vspace{.2cm}

\noindent Part (ii). For $x,z \in \R^3 \setminus \{0\}$, using the semigroup property of $\{P_{t}^{\beta}\}_{t\in [0,\infty)}$,  we can compute
\begin{align}\label{CheckdOne}
\int_{\R^3} \, p_{r,s}^{T,\beta}(x,y) \, p_{s,t}^{T,\beta}(y,z) \, dy\, &\,= \,\frac{1 + Q_{T-t}^{\beta}(z)}{1 + Q_{T-r}^{\beta}(x)}\, \int_{\R^3} \, P_{s-r}^{\beta}(x,y) \, P_{t-s}^{\beta}(y,z) \, dy  \, \nonumber  \\ &\,=  \, \frac{1 + Q_{T-t}^{\beta}(z)}{1 + Q_{T-r}^{\beta}(x)}\,P_{t-r}^{\beta}(x,z) \,=: \,p_{r,t}^{T,\beta}(x,z) \, .
\end{align}
Note that the definition of  $ p_{s,t}^{T,\beta}(y,z)$ for $y=0$ is irrelevant for the above since the integration is with respect to Lebesgue measure.  
\end{proof}

\begin{appendix}

\section{Semigroup property of the attractive point interaction kernel} \label{AppendixSemigroupProperty}
In this section, we prove the semigroup property of the
three-dimensional attractive point interaction heat kernel
$P_t^\beta(x,y)$ defined in~\eqref{DefPointKer3dBeta}. We first
establish some auxiliary integral identities for the free heat kernel, which are then used in the proof of the semigroup relation~\eqref{EqSemigroupKernelReliable}; see
Appendix~\ref{AppendixProofofSemigroupProperty}.

\begin{lemma}
For the function $P_t(r)$ defined in~\eqref{DefFreeHeat3d}, the following hold.
\begin{enumerate}[(i)]
    \item For every \(s,t>0\) and every \(a,b\in\R\),
\begin{align} \label{EqSemigroupPropertyPonREquivalent}
\int_{\R}
P_s(a-r)\,P_t(r-b)\,dr
=
\frac{s+t}{4\pi st}\,P_{s+t}(a-b).
\end{align}

    \item Let $s>0$, $x\in\R^3\setminus\{0\}$, and let $F:(0,\infty)\to\R$ be measurable such that all integrals below are finite. Then
\begin{align} \label{RadialReductionIdentity}
\int_{\R^3}
P_s(x,y)\,\frac{F(|y|)}{|y|}\,dy
&=
\frac{4\pi s}{|x|}
\int_0^\infty
\Bigl(
P_s(r-|x|)-P_s(r+|x|)
\Bigr)\,F(r)\,dr.
\end{align}
\end{enumerate}
\end{lemma}

\begin{proof}

\noindent Part (i). Let $q_s(r):=(4\pi s)^{-1/2}e^{-\frac{r^2}{4s}}$ denote the one-dimensional Gaussian kernel, then $P_s(r)
=(4\pi s)^{-1}q_s(r)$ for every $s>0$ and $r\in\R$. Therefore,
\begin{align*}
\int_{\R} P_s(a-r)\,P_t(r-b)\,dr
&=
\frac{1}{(4\pi s)(4\pi t)}
\int_{\R} q_s(a-r)\,q_t(r-b)\,dr 
=
\frac{1}{(4\pi s)(4\pi t)}\,q_{s+t}(a-b),
\end{align*}
where the last equality uses the semigroup property of $q$. Rewriting the result in terms of \(P\) yields the desired identity. \vspace{.3cm}

\noindent Part (ii). Fix \(x\in\R^3\setminus\{0\}\). Using~\eqref{DefFreeHeat3d} and spherical coordinates \(y=r\omega\) with \(r>0\) and \(\omega\in S^2\), whose Jacobian is \(dy=r^2\,d\omega\,dr\), we obtain
\begin{align} \label{IntegPeqOne}
\int_{\R^3}
P_s(x,y)\,\frac{F(|y|)}{|y|}\,dy
&=
\int_0^\infty\int_{S^2}
\frac{1}{(4\pi s)^{\frac{3}{2}}}
e^{-\frac{|x-r\omega|^2}{4s}}\,
\frac{F(r)}{r}\,r^2\,d\omega\,dr  \nonumber \\
&=
\int_0^\infty
\frac{rF(r)}{(4\pi s)^{\frac{3}{2}}} \,
e^{-\frac{|x|^2+r^2}{4s}}
\int_{S^2}
e^{\frac{r|x|}{2s}\langle e,\omega\rangle}\,d\omega,
\end{align}
where the second equality uses \(|x-r\omega|^2=|x|^2+r^2-2r\langle x,\omega\rangle\) and \(e:=x/|x|\in S^2\) so that $\langle x,\omega\rangle = |x|\langle e,\omega\rangle$. By rotational invariance of the surface measure on \(S^2\), the integral
\(\int_{S^2} e^{\frac{r|x|}{2s}\langle e,\omega\rangle}\,d\omega\)
depends only on \(|e|=1\). Hence we may assume \(e=(0,0,1)\). Writing \(\omega\in S^2\) in spherical coordinates,
\[
\omega=(\sin\theta\cos\varphi,\sin\theta\sin\varphi,\cos\theta),
\qquad
0\le\theta\le\pi,\quad 0\le\varphi\le 2\pi,
\]
we have \(\langle e,\omega\rangle=\cos\theta\) and \(d\omega=\sin\theta\,d\theta\,d\varphi\). Therefore,
\begin{align*}
\int_{S^2}
e^{\frac{r|x|}{2s}\langle e,\omega\rangle}
\,d\omega
&=
\int_0^{2\pi}\int_0^\pi
e^{\frac{r|x|}{2s}\cos\theta}\sin\theta\,d\theta\,d\varphi \\
&=
2\pi\int_0^\pi e^{\frac{r|x|}{2s}\cos\theta}\sin\theta\,d\theta 
=
2\pi\int_{-1}^1 e^{\frac{r|x|}{2s}u}\,du
=
2\pi\left[\frac{e^{\frac{r|x|}{2s}u}}{\frac{r|x|}{2s}}\right]_{-1}^1 =
4\pi s\,\frac{e^{\frac{r|x|}{2s}}-e^{-\frac{r|x|}{2s}}}{r|x|} ,
\end{align*}
where the third equality uses the change of variables $u=\cos\theta$. Substituting the above into~\eqref{IntegPeqOne}, we obtain
\begin{align*}
\int_{\R^3}
P_s(x,y)\,\frac{F(|y|)}{|y|}\,dy
&=
4\pi s\,\int_0^\infty
\frac{rF(r)}{(4\pi s)^{\frac{3}{2}}}
e^{-\frac{|x|^2+r^2}{4s}}
\,\frac{e^{\frac{r|x|}{2s}}-e^{-\frac{r|x|}{2s}}}{r|x|}
\,dr \\
&=
\frac{4\pi s}{|x|}
\int_0^\infty
\frac{F(r)}{(4\pi s)^{\frac{3}{2}}}
\Bigl(
e^{-\frac{(r-|x|)^2}{4s}}
-
e^{-\frac{(r+|x|)^2}{4s}}
\Bigr)
\,dr.
\end{align*}
Thus, using the definition~\eqref{DefFreeHeat3d} of \(P_s\), we obtain the desired identity~\eqref{RadialReductionIdentity}.
\end{proof}

\begin{lemma}[Weighted Volterra decomposition identity]
\label{LemmaWeightedVolterraDecomposition}
Let \(s,t>0\), \(a,b\ge 0\), and \(\lambda>0\). Then
\begin{align}
\int_0^\infty e^{\lambda v}
\int_0^v
P_s(v+a-u)\,P_t(u+b)\,du\,dv
&=
\int_0^\infty e^{\lambda v}
\int_0^\infty
P_s(a+r)\,P_t(v+r+b)\,dr\,dv \nonumber \\
&+
\int_0^\infty e^{\lambda u}
\int_0^\infty
P_s(u+a+r)\,P_t(r+b)\,dr\,du \nonumber \\
&+
2\lambda
\int_0^\infty\int_0^\infty
e^{\lambda(u+v)}
\int_0^\infty
P_s(u+a+r)\,P_t(v+r+b)\,dr\,du\,dv. \nonumber
\end{align}
\end{lemma}

\begin{proof}
Observe that
\[
I
\,:=\,
\int_0^\infty e^{\lambda v}
\int_0^v
P_s(v+a-u)\,P_t(u+b)\,du\,dv 
=
\int_{\{(u,v):\,0<u<v<\infty\}}
e^{\lambda v}\,P_s(v+a-u)\,P_t(u+b)\,du\,dv .
\]
Consider the change of variables
\[
(u,v)\mapsto (u,r):=(u,v-u),
\qquad\text{so that}\qquad
v=u+r.
\]
Then \(u>0\) and \(r:=v-u>0\), so this transformation sends the triangular region
\(\{(u,v):0<u<v<\infty\}\) onto the first quadrant \((0,\infty)^2\). Moreover, since the Jacobian
is \(1\), we obtain
\begin{align}
I
&=
\int_0^\infty\int_0^\infty
e^{\lambda(u+r)}\,P_s(a+r)\,P_t(u+b)\,du\,dr \nonumber\\
&=
\int_0^\infty\int_0^\infty
e^{\lambda(u+v)}\,P_s(a+v)\,P_t(u+b)\,du\,dv
=
\int_0^\infty\int_0^\infty
e^{\lambda(u+v)}\,P_s(u+a)\,P_t(v+b)\,du\,dv ,
\label{EqVolterraQuadSymmetric}
\end{align}
where in the second equality we relabel \(r\) as \(v\), and in the third equality we use the symmetry of the domain \((0,\infty)^2\) and Tonelli's theorem to replace \((u,v)\) by \((v,u)\). Next define, for \(u,v\ge0\),
\begin{align} 
\label{EqDefineFWeighted}
F(u,v)
=
e^{\lambda(u+v)}
\int_0^\infty
P_s(u+a+r)\,P_t(v+r+b)\,dr .
\end{align}
Differentiating with respect to \(u\) and \(v\), we get
\begin{align*}
\partial_u F(u,v)
&=
\lambda F(u,v)
+
e^{\lambda(u+v)}
\int_0^\infty
\partial_u P_s(u+a+r)
P_t(v+r+b)\,dr, \\
\partial_v F(u,v)
&=
\lambda F(u,v)
+
e^{\lambda(u+v)}
\int_0^\infty
P_s(u+a+r)
\partial_v P_t(v+r+b)
\,dr .
\end{align*}
Hence, using $\partial_u P_s(u+a+r)=\partial_r P_s(u+a+r)$ and $\partial_v P_t(v+r+b)=\partial_r P_t(v+r+b)$, we obtain
\begin{align}
(\partial_u+\partial_v-2\lambda)F(u,v)
=
e^{\lambda(u+v)}
\int_0^\infty
\partial_r\!\Bigl(P_s(u+a+r)\,P_t(v+r+b)\Bigr)\,dr 
=
-e^{\lambda(u+v)}\,P_s(u+a)\,P_t(v+b), \nonumber
\end{align}
where the second equality uses that the boundary term at \(r=\infty\) vanishes by Gaussian decay. Substituting the above in~\eqref{EqVolterraQuadSymmetric}, we obtain
\begin{align}
I
=
-\int_0^\infty\int_0^\infty \partial_u F(u,v)\,du\,dv
-\int_0^\infty\int_0^\infty \partial_v F(u,v)\,du\,dv
+
2\lambda\int_0^\infty\int_0^\infty F(u,v)\,du\,dv .
\label{EqAfterIntegratingPDE}
\end{align}
Since \(F(u,v)\to0\) as \(u\to\infty\) or \(v\to\infty\), another application of the
fundamental theorem of calculus yields
\begin{align*}
-\int_0^\infty\int_0^\infty \partial_u F(u,v)\,du\,dv
=
\int_0^\infty F(0,v)\,dv, 
\quad \textup{ and } \quad
-\int_0^\infty\int_0^\infty \partial_v F(u,v)\,du\,dv
=
\int_0^\infty F(u,0)\,du .
\end{align*}
Substituting these identities into~\eqref{EqAfterIntegratingPDE}, and then using the definition~\eqref{EqDefineFWeighted}, yields the desired identity.
\end{proof}

\subsection[Semigroup]{Proof of the semigroup property~\eqref{EqSemigroupKernelReliable}} \label{AppendixProofofSemigroupProperty}

Fix \(s,t>0\) and \(x,z\in \R^{3}\setminus\{0\}\). We begin from the left-hand side of~\eqref{EqSemigroupKernelReliable} and expand both kernels using~\eqref{DefPointKer3dBeta} and grouping together the terms of the same type, we write
\begin{align}
\int_{\R^3} P_s^\beta(x,y)\,P_t^\beta(y,z)\,dy
&=
I_1 + I_2 + I_3,
\label{EqSemigroupThreeBlocks}
\end{align}
where the three contributions are defined as follows. \vspace{.3cm}

\noindent \textit{(i) Free--free term.} Since \(P_t(x,y)\) is the free heat kernel~\eqref{DefFreeHeat3d} on \(\R^3\), its semigroup property yields
\begin{align}
I_1
\,:=\,
\int_{\R^3} P_s(x,y)\,P_t(y,z)\,dy
=
P_{s+t}(x,z).
\label{EqI1BlockEval}
\end{align}

\vspace{.2cm}

\noindent\textit{(ii) Single-singular terms.} These are the terms where exactly one of the two kernels contributes a singular factor. Grouping all such contributions together, we obtain
\begin{align}
I_2
&\,:=\,
\int_{\R^3}
P_s(x,y)\,
\frac{2t}{|y||z|}\,P_t(|y|+|z|)\,dy
+
\int_{\R^3}
\frac{2s}{|x||y|}\,P_s(|x|+|y|)\,
P_t(y,z)\,dy
\nonumber\\
&
+
\int_{\R^3}
P_s(x,y)\,
\frac{8\pi\beta t}{|y||z|}
\int_0^\infty e^{4\pi\beta v}\,P_t(v+|y|+|z|)\,dv\,dy
+
\int_{\R^3}
\frac{8\pi\beta s}{|x||y|}
\int_0^\infty e^{4\pi\beta u}\,P_s(u+|x|+|y|)\,du\;
P_t(y,z)\,dy. \nonumber
\end{align}
We denote the four terms on the right-hand side by \(I_{2}^{(i)}\), \(i=1,2,3,4\), respectively. Then,
\begin{align*}
I_{2}^{(1)}
\,:=\,
\frac{2t}{|z|}
\int_{\R^3}
P_s(x,y)\,\frac{P_t(|y|+|z|)}{|y|}\,dy
=
\frac{8\pi st}{|x||z|}
\int_0^\infty
\Bigl(
P_s(r-|x|)-P_s(r+|x|)
\Bigr)
P_t(r+|z|)\,dr,
\end{align*}
where the second equality uses~\eqref{RadialReductionIdentity} with $F(r):=P_t(r+|z|)$. Similarly, we have 
\begin{align*}
I_{2}^{(2)}
\,:=\,
\frac{2s}{|x|}
\int_{\R^3}
P_t(y,z)\,\frac{P_s(|x|+|y|)}{|y|}\,dy
=
\frac{8\pi st}{|x||z|}
\int_0^\infty
\Bigl(
P_t(r-|z|)-P_t(r+|z|)
\Bigr)
P_s(r+|x|)\,dr.
\end{align*}
Adding \(I_{2}^{(1)}\) and \(I_{2}^{(2)}\), we obtain
\begin{align*}
I_{2}^{(1)}+I_{2}^{(2)}
&=
\frac{8\pi st}{|x||z|}
\int_0^\infty
\bigg(
P_s(r-|x|)\,P_t(r+|z|)
+
 P_s(r+|x|)\, P_t(r-|z|) 
 -
2\, P_s(r+|x|)\,P_t(r+|z|) 
\bigg) \, dr \nonumber \\
&=
\frac{8\pi st}{|x||z|}
\int_{\R}
P_s(r-|x|)\,P_t(r+|z|)\,dr
-
\frac{16\pi st}{|x||z|}
\int_0^\infty
P_s(r+|x|)\,P_t(r+|z|)\,dr,
\end{align*}
where the second equality follows by combining the first two terms into a single integral over \(\R\), which is justified since, by the evenness of \(P_s\) and \(P_t\) together with the change of variables \(u:=-r\), we obtain
\begin{align*}
\int_0^\infty P_s(r+|x|)\,P_t(r-|z|)\,dr
=
\int_0^\infty P_s(-r-|x|)\,P_t(-r+|z|)\,dr
=
\int_{-\infty}^0 P_s(u-|x|)\,P_t(u+|z|)\,du.
\end{align*}
Thus, using~\eqref{EqSemigroupPropertyPonREquivalent} with \(a:=|x|\) and \(b:=-|z|\), we obtain
\begin{align} \label{I2SUMfirstTwoTerms}
I_{2}^{(1)}+I_{2}^{(2)}
=
\frac{2(s+t)}{|x||z|}\,P_{s+t}(|x|+|z|)
-
\frac{16\pi st}{|x||z|}
\int_0^\infty
P_s(r+|x|)\,P_t(r+|z|)\,dr.
\end{align}

For \(I_2^{(3)}\), we apply Tonelli's theorem to interchange the order of integration, which yields
\begin{align*}
I_{2}^{(3)}
&=
\frac{8\pi\beta t}{|z|}
\int_0^\infty e^{4\pi\beta v}
\left(
\int_{\R^3}
P_s(x,y)\,
\frac{1}{|y|}\,P_t(v+|y|+|z|)\,dy
\right)dv \\
&=
\frac{32\pi^2\beta st}{|x||z|}
\int_0^\infty
e^{4\pi\beta v}
\bigg[
\int_0^\infty
P_s(r-|x|)\, P_t(v+r+|z|)\,dr
-
\int_0^\infty
P_s(r+|x|)\, P_t(v+r+|z|)\,dr \bigg]\,dv
\end{align*}
where in the second equality we applied~\eqref{RadialReductionIdentity} to the inner spatial integral with \(F(r):=P_t(v+r+|z|)\). Next, in the first inner integral we make the change of variables \(w:=v+r\) to obtain
\begin{align*}
I_{2}^{(3)}
&=
\frac{32\pi^2\beta st}{|x||z|}
\int_0^\infty
e^{4\pi\beta v}
\bigg[
\int_v^\infty
P_s(w-v-|x|)\, P_t(w+|z|)\,dw
-
\int_0^\infty
P_s(r+|x|)\, P_t(v+r+|z|)\,dr
\bigg]\,dv \nonumber\\
&=
\frac{32\pi^2\beta st}{|x||z|}
\int_0^\infty
e^{4\pi\beta v}
\bigg[
\int_v^\infty
P_s(v+|x|-w)\, P_t(w+|z|)\,dw
-
\int_0^\infty
P_s(r+|x|)\, P_t(v+r+|z|)\,dr
\bigg]\,dv.
\end{align*}
where the second equality uses the evenness of \(P_s\). Similarly, using~\eqref{RadialReductionIdentity} with \(F(r):=P_s(u+|x|+r)\), we obtain
\begin{align*}
I_{2}^{(4)}
&=
\frac{32\pi^2\beta st}{|x||z|}
\int_0^\infty e^{4\pi\beta u}
\left[
\int_0^\infty
P_t(r-|z|)\,P_s(u+|x|+r)\,dr
-
\int_0^\infty
P_t(r+|z|)\,P_s(u+|x|+r)\,dr
\right]du \\
&=
\frac{32\pi^2\beta st}{|x||z|}
\int_0^\infty e^{4\pi\beta u}
\left[
\int_{-\infty}^0
P_t(w+|z|)\,P_s(u+|x|-w)\,dw
-
\int_0^\infty
P_t(r+|z|)\,P_s(u+|x|+r)\,dr
\right]du.
\end{align*}
where the second equality follows from the evenness of \(P_t\) and the change of variables \(w:=-r\) in the first inner integral. Renaming \(u\) as \(v\) in $I_2^{(4)}$ and adding with $I_2^{(3)}$, we obtain
\begin{align}  
I_{2}^{(3)}+I_{2}^{(4)}
&=
\frac{32\pi^2\beta st}{|x||z|}
\int_0^\infty
e^{4\pi\beta v}
\bigg[
\int_{\R}
P_s(v+|x|-w)\,P_t(w+|z|)\,dw
-
\int_0^v
P_s(v+|x|-w)\,P_t(w+|z|)\,dw \nonumber \\
&-
\int_0^\infty
P_s(r+|x|)\, P_t(v+r+|z|)\,dr
-
\int_0^\infty
P_t(r+|z|)\,P_s(v+|x|+r)\,dr
\bigg]\,dv. \nonumber
\end{align}
Applying~\eqref{EqSemigroupPropertyPonREquivalent} to the integral over \(\R\) (with \(a:=v+|x|\) and \(b:=-|z|\)), and renaming the dummy variable \(w\) as \(u\), we obtain
\begin{align} \label{I2SUMlastTwoTerms}
I_{2}^{(3)}+I_{2}^{(4)}
&=
\frac{8\pi\beta(s+t)}{|x||z|}
\int_0^\infty
e^{4\pi\beta v}\,
P_{s+t}(v+|x|+|z|)\,dv \nonumber \\
&-
\frac{32\pi^2\beta st}{|x||z|}
\int_0^\infty\int_0^v
e^{4\pi\beta v}
P_s(v+|x|-u)\,P_t(u+|z|)\,du\,dv \nonumber \\
&-
\frac{32\pi^2\beta st}{|x||z|}
\int_0^\infty\int_0^\infty
e^{4\pi\beta v}
P_s(r+|x|)\, P_t(v+r+|z|)\,dr\,dv \nonumber \\
&-
\frac{32\pi^2\beta st}{|x||z|}
\int_0^\infty\int_0^\infty
e^{4\pi\beta v}
P_t(r+|z|)\,P_s(v+|x|+r)\,dr\,dv.
\end{align}

\noindent \textit{(iii) Double-singular terms.} These are the remaining terms where both factors contribute singular parts
\begin{align}
I_3
&\,:=\,
\int_{\R^3}
\frac{2s}{|x||y|}\,P_s(|x|+|y|)\,
\frac{2t}{|y||z|}\,P_t(|y|+|z|)\,dy
\nonumber\\
&
+
\int_{\R^3}
\frac{2s}{|x||y|}\,P_s(|x|+|y|)\,
\frac{8\pi\beta t}{|y||z|}
\int_0^\infty e^{4\pi\beta v}\,P_t(v+|y|+|z|)\,dv\,dy
\nonumber\\
&
+
\int_{\R^3}
\frac{8\pi\beta s}{|x||y|}
\int_0^\infty e^{4\pi\beta u}\,P_s(u+|x|+|y|)\,du\;
\frac{2t}{|y||z|}\,P_t(|y|+|z|)\,dy
\nonumber\\
&
+
\int_{\R^3}
\frac{8\pi\beta s}{|x||y|}
\int_0^\infty e^{4\pi\beta u}\,P_s(u+|x|+|y|)\,du
\;
\frac{8\pi\beta t}{|y||z|}
\int_0^\infty e^{4\pi\beta v}\,P_t(v+|y|+|z|)\,dv\,dy. \nonumber
\end{align}
Let us denote the four terms on the right-hand side by \(I_{3}^{(i)}\), \(i=1,2,3,4\).

For \(I_{3}^{(1)}\), since the integrand depends on \(y\) only through \(|y|\), we pass to spherical coordinates by writing \(y=r\omega\), with \(r>0\), \(\omega\in S^2\), and \(dy=r^2\,d\omega\,dr\). Then,
\begin{align} \label{I3FirstTerm}
I_{3}^{(1)}
&=
\frac{4st}{|x||z|}
\int_{\R^3}
\frac{1}{|y|^2}\,
P_s(|x|+|y|)\,
P_t(|y|+|z|)\,dy \nonumber \\
&=
\frac{4st}{|x||z|}
\int_0^\infty\int_{S^2}
\frac{1}{r^2}\,
P_s(|x|+r)\,
P_t(r+|z|)\,
r^2\,d\omega\,dr
=
\frac{16\pi st}{|x||z|}
\int_0^\infty
P_s(|x|+r)\,P_t(r+|z|)\,dr,
\end{align}
where the last equality uses that \(\int_{S^2}d\omega=4\pi\). 

For $I_{3}^{(2)}$, since the integrand is nonnegative, Tonelli's theorem yields
\begin{align*}
I_{3}^{(2)}
&=
\frac{16\pi\beta st}{|x||z|}
\int_0^\infty e^{4\pi\beta v}
\left(
\int_{\R^3}
\frac{1}{|y|^2}\,
P_s(|x|+|y|)\,
P_t(v+|y|+|z|)\,dy
\right)\,dv \nonumber \\
&=
\frac{16\pi\beta st}{|x||z|}
\int_0^\infty e^{4\pi\beta v}
\left(
\int_0^\infty \int_{S^2}
\frac{1}{r^2}\,
P_s(|x|+r)\,
P_t(v+r+|z|)\,
r^2\,d\omega\,dr
\right)\,dv,
\end{align*}
where, in the second equality, we passed to spherical coordinates by writing \(y=r\omega\), with \(r>0\), \(\omega\in S^2\), and \(dy=r^2\,d\omega\,dr\). Thus, using \(\int_{S^2}d\omega=4\pi\), we obtain
\begin{align} \label{I3SecondTerm}
I_{3}^{(2)}
&=
\frac{64\pi^2\beta st}{|x||z|}
\int_0^\infty
e^{4\pi\beta v}
\int_0^\infty
P_s(|x|+r)\,P_t(v+r+|z|)\,dr\,dv.
\end{align}
Similarly,
\begin{align} \label{I3ThirdTerm}
I_{3}^{(3)}
&=
\frac{64\pi^2\beta st}{|x||z|}
\int_0^\infty
e^{4\pi\beta u}
\int_0^\infty
P_s(u+|x|+r)\,P_t(r+|z|)\,dr\,du.
\end{align}

For $I_{3}^{(4)}$, we apply Tonelli's theorem twice to obtain
\begin{align} \label{I3FourthTerm}
I_{3}^{(4)}
&=
\frac{64\pi^2\beta^2 st}{|x||z|}
\int_0^\infty\int_0^\infty
e^{4\pi\beta(u+v)}
\left(
\int_{\R^3}
\frac{1}{|y|^2}\,
P_s(u+|x|+|y|)\,
P_t(v+|y|+|z|)\,dy
\right)\,du\,dv \nonumber \\
&=
\frac{64\pi^2\beta^2 st}{|x||z|}
\int_0^\infty\int_0^\infty
e^{4\pi\beta(u+v)}
\left(
\int_0^\infty\int_{S^2}
\frac{1}{r^2}\,
P_s(u+|x|+r)\,
P_t(v+r+|z|)\,
r^2\,d\omega\,dr
\right)\,du\,dv \nonumber \\
&=
\frac{256\pi^3\beta^2 st}{|x||z|}
\int_0^\infty\int_0^\infty
e^{4\pi\beta(u+v)}
\int_0^\infty
P_s(u+|x|+r)\,P_t(v+r+|z|)\,dr\,du\,dv,
\end{align}
where in the second equality we used the change of variables \(y=r\omega\), with \(r>0\), \(\omega\in S^2\), and \(dy=r^2\,d\omega\,dr\).

Adding~\eqref{EqI1BlockEval}--\eqref{I3FourthTerm}, we obtain
\begin{align*}
\int_{\R^3} P_s^\beta(x,y)\,P_t^\beta(y,z)\,dy
&=
P_{s+t}(x,z) 
+
\frac{2(s+t)}{|x||z|}\,P_{s+t}(|x|+|z|) +
\frac{8\pi\beta (s+t)}{|x||z|}
\int_0^\infty e^{4\pi\beta v}\,P_{s+t}(v+|x|+|z|)\,dv \\
&-
\frac{32\pi^2\beta st}{|x||z|}
\int_0^\infty e^{4\pi\beta v}
\int_0^v
P_s(v+|x|-u)\,P_t(u+|z|)\,du\,dv \\
&+
\frac{32\pi^2\beta st}{|x||z|}
\int_0^\infty e^{4\pi\beta v}
\int_0^\infty
P_s(r+|x|)\,P_t(v+r+|z|)\,dr\,dv \\
&+
\frac{32\pi^2\beta st}{|x||z|}
\int_0^\infty e^{4\pi\beta u}
\int_0^\infty
P_s(u+|x|+r)\,P_t(r+|z|)\,dr\,du \\
&+
\frac{256\pi^3\beta^2 st}{|x||z|}
\int_0^\infty\int_0^\infty
e^{4\pi\beta(u+v)}
\int_0^\infty
P_s(u+|x|+r)\,P_t(v+r+|z|)\,dr\,du\,dv.
\end{align*}
Now applying Lemma~\ref{LemmaWeightedVolterraDecomposition} with \(a:=|x|\), \(b:=|z|\), and \(\lambda:=4\pi\beta\), the remaining four terms cancel exactly, and the expression reduces to the defining formula of \(P_{s+t}^\beta(x,z)\) by \eqref{DefPointKer3dBeta}. This proves~\eqref{EqSemigroupKernelReliable}.

\section[Proof]{Proof of~\eqref{PsiAsymptotics}} \label{AppendixProofPsiAsymptotics}

Let \(\psi\in\mathcal D(\Delta^{\beta})\). Then there exist \(k\in\mathbb C\) with \(\mathrm{Im}\,k>0\) and
\(\phi_k\in H^{2,2}(\R^3)\) such that
\begin{align}
\psi(x)
=
\phi_k(x)
+
(-\beta-ik/4\pi)^{-1}\phi_k(0)\,G_k(x),
\qquad x\in\R^3\setminus\{0\},
\nonumber
\end{align}
where $G_k(x) := (4\pi|x|)^{-1} e^{ik|x|}$. Since \(\phi_k\in H^{2,2}(\R^3)\), the Sobolev embedding theorem implies that
\(\phi_k\) is continuous on \(\R^3\). Hence $\phi_k(x) = \phi_k(0) + o(1)$ as $|x|\downarrow0$. Moreover, $e^{ik|x|} = 1+ik|x|+o(|x|)$ as $|x|\downarrow0$, and therefore
\begin{align}
G_k(x)
=
\frac{e^{ik|x|}}{4\pi|x|}
=
\frac{1}{4\pi|x|}
+
\frac{ik}{4\pi}
+
o(1),
\qquad |x|\downarrow0.
\nonumber
\end{align}
Define $\mathbf c_{\psi} := \frac{1}{4\pi}
(-\beta-ik/4\pi)^{-1}\phi_k(0)$. Substituting the above expansions into the decomposition of \(\psi\) yields
\begin{align}
\psi(x)
&=
\phi_k(0)
+
(-\beta-ik/4\pi)^{-1}\phi_k(0)
\left(
\frac{1}{4\pi|x|}
+
\frac{ik}{4\pi}
\right)
+
o(1)
\nonumber\\
&=
\frac{(-\beta-ik/4\pi)^{-1}\phi_k(0)}{4\pi|x|}
+
\left(
1+
\frac{ik}{4\pi}
(-\beta-ik/4\pi)^{-1}
\right)\phi_k(0)
+
o(1) \nonumber\\
&=
\frac{\mathbf c_{\psi}}{|x|}
+
\Big(
4\pi
(-\beta-ik/4\pi )+
ik
\Big) \frac{1}{4\pi}
(-\beta-ik/4\pi)^{-1} \, \phi_k(0)
+
o(1) \nonumber \\
&=
\frac{\mathbf c_{\psi}}{|x|}
-
4\pi\beta\,\mathbf c_{\psi}
+
o(1),
\qquad |x|\downarrow0, \nonumber
\end{align}
which is the claimed asymptotic expansion.
\end{appendix}

\end{document}